\documentclass{article}

\usepackage{arxiv}

\usepackage[english]{babel}
\usepackage{csquotes}
\usepackage{booktabs}
\usepackage{graphicx}
\usepackage{sidecap}
\usepackage[hypertexnames=false]{hyperref}
\usepackage{amssymb,amsmath,amsthm, mathrsfs}
\usepackage{mathtools}
\usepackage{accents}
\usepackage{dsfont}
\usepackage{diagbox,multirow}
\usepackage{hhline}
\usepackage[dvipsnames,x11names]{xcolor}
\usepackage{array}
\usepackage{colortbl}
\usepackage{subfig}
\usepackage{enumerate}
\usepackage[shortlabels]{enumitem}
\usepackage{makecell}
\usepackage{float}
\usepackage{framed}
\usepackage{varwidth}
\usepackage[export]{adjustbox}
\usepackage{xparse}
\usepackage {tikz}
\usetikzlibrary{cd}
\usetikzlibrary{arrows,decorations.pathmorphing,backgrounds,positioning,fit,matrix}
\usetikzlibrary{decorations.markings}
\usepackage{esint}
\usepackage{modernfrontmatter}
\usepackage{bm}

\newlength\mylen
\tikzset{
	bicolor/.style 2 args={
		dashed,dash pattern=on 20pt off 20pt,-,#1,
		postaction={draw,dashed,dash pattern=on 20pt off 20pt,-,#2,dash phase=20pt}
	},
}

\hypersetup{
	colorlinks,
	linkcolor={red!50!black},
	citecolor={blue!50!black},
	urlcolor={MJAccent}
}
\makeatletter
\newcommand{\conditem}[2]{%
  \item[#2]%
  \phantomsection
  \protected@edef\@currentlabel{#2}%
  \label{#1}%
}
\makeatother

\usepackage{etoolbox}
\apptocmd{\sloppy}{\hbadness 10000\relax}{}{}
\usepackage{amsbsy}

\newtheorem{lemma}{Lemma}[section]
\newtheorem{theorem}[lemma]{Theorem}
\newtheorem{proposition}[lemma]{Proposition}
\newtheorem{corollary}[lemma]{Corollary}

\theoremstyle{definition}
\newtheorem{definition}[lemma]{Definition}
\newtheorem*{example}{Example}

\theoremstyle{remark}
\newtheorem{remark}[lemma]{Remark}

\theoremstyle{plain}
\newtheorem{mainthm}{Theorem}

\numberwithin{equation}{section}

\newcommand{\R}{\mathbb{R}}
\newcommand{\Z}{\mathbb{Z}}

\newcommand{\UM}{\hyperref[UM]{\textup{\textbf{(UM)}}}}
\newcommand{\BD}{\hyperref[BD]{\textup{\textbf{(BD)}}}}
\newcommand{\C}{\hyperref[C]{\textup{\textbf{(C)}}}}

\newcommand{\SC}{\hyperref[SC]{\textup{\textbf{(SC)}}}}
\newcommand{\SCinf}{\hyperref[def:SC-infinity]{\textup{\textbf{(SC\textsubscript{$\infty$})}}}}
\newcommand{\FM}{\hyperref[FM]{\textup{\textbf{(FM)}}}}
\newcommand{\LG}{\hyperref[eq:LG]{\textup{\textbf{(LG)}}}}

\newcommand{\Q}{\mathcal{Q}}
\newcommand{\eps}{\varepsilon}

\newcommand{\ka}{\kappa}
\newcommand{\pt}{\partial_t}
\newcommand{\Om}{\Omega}

\newcommand{\GammaT}{\Gamma_T^\Omega}
\newcommand{\partialpGammaT}{\partial_p\GammaT}
\newcommand{\CtGammaT}{C_t^1(\GammaT)}
\newcommand{\CtGammaTClosure}{C_t^1(\overline{\GammaT})}
\newcommand{\ip}[2]{\left\langle #1,#2\right\rangle}
\newcommand{\dd}{\,d}
\newcommand{\ellp}[1]{\ell^{#1}(X,\mu)}
\newcommand{\norm}[2]{\left\|#1\right\|_{#2}}
\newcommand{\abs}[1]{\left|#1\right|}
\newcommand{\supp}{\operatorname{supp}}

\newcommand{\SP}[1][]{\hyperref[SP]{\textup{\textbf{(S\ifx&#1&\else\textsubscript{\ensuremath{#1}}\fi)}}}}

\makeatletter
\newcommand{\sbullet}{%
	\hbox{\fontfamily{lmr}\fontsize{.6\dimexpr(\f@size pt)}{0}\selectfont\textbullet}}

\makeatother

\DeclareMathSymbol{\shortminus}{\mathbin}{AMSa}{"39}

\newcommand{\dom}{\textnormal{dom}}

\newcommand{\Deg}{\operatorname{Deg}}

\newcommand{\Deltaphi}{\mathcal{L}}
\newcommand{\Deltaphip}{\mathcal{L}^{(p)}}

\NewDocumentCommand{\Deltadir}{o o}{%
	\Delta_{D\IfNoValueF{#1}{,#1}}%
	\IfNoValueF{#2}{^{#2}}%
}
\newcommand{\Deltadirn}{\Deltadir[n]}
\NewDocumentCommand{\Qdir}{o}{\Q\IfNoValueF{#1}{_{#1}}^{D}}

\newcommand{\id}{\operatorname{id}}
\newcommand{\ellone}{\ell^1(X,\mu)}
\newcommand{\embn}{\boldsymbol{\mathfrak{i}}_n}
\newcommand{\embom}{\boldsymbol{\mathfrak{i}}_\Omega}
\newcommand{\projn}{\boldsymbol{\pi}_n}
\newcommand{\projom}{\boldsymbol{\pi}_\Omega}

\usepackage[style=numeric,backend=bibtex,giveninits=true,maxbibnames=99,sortcites=true]{biblatex}
\allowdisplaybreaks
\title{The generalized porous medium equation on  graphs: well-posedness, extinction, and mass conservation}
\articlelabel{}
\paperstatus{Preprint}
\runningtitle{GPME on graphs: well-posedness, extinction, and mass conservation}
\runningauthors{Bianchi \textperiodcentered\ Hua \textperiodcentered\ Setti \textperiodcentered\ Wojciechowski}
\date{}

\paperauthor[1]{Davide Bianchi}
\paperauthor[2]{Bobo Hua}
\paperauthor[3]{Alberto G. Setti}
\paperauthor[4,5]{Rados{\l}aw K. Wojciechowski}

\paperaffiliation[1]{School of Mathematics (Zhuhai), Sun Yat-sen University, Zhuhai 518055, China}
\paperaffiliation[2]{School of Mathematical Sciences, LMNS, Fudan University, Shanghai 200433, China}
\paperaffiliation[3]{Dipartimento di Scienze e Alta Tecnologia, Universit\`a dell'Insubria, Como 22100, Italy}
\paperaffiliation[4]{Department of Mathematics and Computer Science, York College -- CUNY, Jamaica, USA}
\paperaffiliation[5]{Department of Mathematics, Graduate Center -- CUNY, New York, USA}

\authoremail[D.B.]{bianchid@mail.sysu.edu.cn}
\authoremail[B.H.]{bobohua@fudan.edu.cn}
\authoremail[A.G.S.]{alberto.setti@uninsubria.it}
\authoremail[R.K.W.]{rwojciechowski@gc.cuny.edu}


\begin{document}
\maketitle
\begin{abstract}
We study the Cauchy problem for the generalized porous medium equation on infinite weighted graphs. For a general nonlinearity, we establish Dirichlet comparison
and weak maximum principles on finite subgraphs and, through an exhaustion
argument, construct minimal and maximal  pointwise solutions for
arbitrary $\ell^\infty$ initial data, controlled by explicit, possibly time-dependent,  barriers.

For the porous nonlinearity $\phi(s)=s|s|^{m-1}$, 
assuming a $\nu$-Sobolev inequality with
$\nu>2$, we derive quantitative energy estimates for  $\ell^1$-mild solutions. These yield finite-time
extinction in the fast diffusion range $0<m<2/\nu$ and $\ell^1$-$\ell^q$
smoothing in the range $m>2/\nu$. Interestingly, we recover the Euclidean critical exponent  for several model graphs. Finally, we prove
an exact generalized mass balance for nonnegative 
$\ell^1$-mild solutions on graphs that are stochastically complete at
infinity, allowing for an arbitrary killing term. The same balance is
established for suitable classical and bounded pointwise solutions. In the
absence of killing, these reduce to conservation of mass.
\end{abstract}

\section{Introduction}\label{sec:introduction}

The \emph{generalized porous medium equation} (GPME) (also known as the \emph{filtration equation}) has a long history. 
For a thorough treatment and classical results, we refer the reader to the lecture notes of Aronson~\cite{Aronson1986} and to the seminal book of V\'azquez~\cite{vazquez2007porous}. 
The (GPME) reads as 
\begin{equation}\tag{GPME}\label{eq:GPME}
	(\partial_t+\Delta\Phi)u(t,x)=f(t,x) 
	\qquad \text{for }(t,x)\in(0,T)\times X
\end{equation}
where $f$ is referred to as the \emph{forcing term}.
When $f=0$, we call the corresponding \eqref{eq:GPME} \emph{homogeneous}.
Given an initial datum $u_0$, 
the associated Cauchy problem is
\begin{equation}\tag{Cauchy-GPME}\label{eq:C-D}
	\begin{cases}
		(\partial_t +\Delta\Phi)u(t,x) = f(t,x)
		& \text{if } (t,x) \in (0,T)\times X\\
		u(0,x) = u_0(x)
		& \text{if } x\in X.
	\end{cases}
\end{equation}
Research on the~\eqref{eq:GPME} 
has remained highly active in recent years, both in Euclidean spaces~\cite{kienzler2016flat,cortazar2017near,cortazar2018nearfield,jin2022singular,bowen2023cauchy,jin2023optimal,akagi2023rates,choi2023asymptotics,jin2023bubbling,bonforte2024cauchy,jin2025regularity,jin2026extinction,grillo2026rough,bonforte2010veryfast} and in the Riemannian setting~\cite{shen2016gradient,grillo2016smoothing,grillo2017largetime,grillo2018measure,grillo2018large,cao2018aronson,grillo2019superquadratic,wang2021gradient,fujitani2024aronson,grillo2025green,bianchi2018laplacian}. Alongside these, there has been growing interest in the metric and discrete settings~\cite{ma2022porous,bianchi2022generalized,li2025gradient,kou2026porous,berchio2026semilinear,berchio2026fractional}. This work belongs to the latter category. 

A weighted graph is a quadruple $G=(X,w,\kappa,\mu)$ consisting of a countable set of vertices $X$, an edge-weight function
$w\colon X\times X\to [0,\infty)$, a killing term $\kappa\colon X\to [0,\infty)$
and a vertex measure $\mu\colon X\to (0,\infty)$, see Section~\ref{sec:notation}.
The Laplacian is defined as
\begin{equation*}
\Delta f(x)
=
\frac1{\mu(x)}
\sum_{y\in X}w(x,y)(f(x)-f(y))
+
\frac{\kappa(x)}{\mu(x)}f(x)
\end{equation*}
for $x\in X$ and functions for which the sum is well-defined. 
Note that $\kappa$ plays the role of the potential term in a Schrödinger operator. Two vertices $x,y$ are called neighbors if $w(x,y)>0.$
We do not assume that the graph is locally finite, i.e., 
each vertex may have infinitely many neighbors.
This non-locally finite setting is not only interesting in itself but
becomes essential for considering nonlocal operators, such as the fractional Laplacian \cite{JXWang2023,ZhangLinYang2025}, the Dirichlet-to-Neumann operator \cite{HuaHuangWang2022}, and others.

We let \(\phi\colon\R\to\R\) be continuous and monotone increasing, where
monotone increasing is always understood in the weak sense, that is,
\(\phi(s_1)\leq\phi(s_2)\) whenever \(s_1 < s_2\). We also assume that
\(\phi(0)=0\). We refer to \(\phi\) as the \emph{nonlinearity}, we denote by
\[
\Phi u:=\phi\circ u
\]
its canonical extension to functions on \(X\), and we call \(\Delta\Phi\) the
associated \emph{nonlinear operator}.
The choice of \(\phi\) determines the specific instance of \eqref{eq:GPME}.
A typical example is
\[
\phi_m(s)=s^m:=s|s|^{m-1}
\]
with $m>0$.
For $m=1$, we recover the classical heat equation, while $m>1$ and $0<m<1$
correspond to the porous medium and fast diffusion equations, respectively.
Our focus is on the Cauchy problem for the generalized porous medium equation
\eqref{eq:C-D} on weighted graphs.

In the discrete setting, the lack of chain rules means that several techniques from the continuous framework cannot be carried over directly. In~\cite{bianchi2022generalized}, building on the works of B\'enilan, Crandall, Liggett, and Pazy~\cite{benilan1988evolution,crandall1971generation}, an implicit Euler scheme was introduced to construct $\ell^1$-mild solutions of \eqref{eq:C-D} for the maximal restriction operator $\mathcal{L}\subseteq \Delta\Phi$ under very weak assumptions on the graph. 

Pointwise solutions of the formal porous operator $\Delta\Phi$ are nonetheless of independent interest. In this manuscript we address two complementary aspects: 
first, the existence of pointwise solutions to the \eqref{eq:C-D} for the formal porous operator under bounded initial data; and second, qualitative properties of both these pointwise solutions and the $\ell^1$-mild solutions associated with the maximal restriction operator. For the precise definitions of the operators, and the various notions of solution and their relationships, see Section~\ref{sec:notation}.

In particular, we first develop a comparison theory for the formal operator
\(\Delta\Phi\) on finite subgraphs with prescribed exterior data, 
see Theorem~\ref{thm:comparison}. These finite-domain estimates are then
used in an exhaustion argument to construct pointwise solutions on
the entire graph which are defined globally in time. More precisely, for every bounded $u_0$ and forcing term $f$, 
we construct lower (resp., upper) extremal pointwise solutions, 
which are minimal (resp., maximal)
among all pointwise solutions satisfying certain bounds. The existence of these extremal solutions requires no additional assumptions on the graph.
More precisely, we prove the following existence result.
\begin{mainthm}[{see Theorem~\ref{thm:extremal-bounded-pointwise-solutions}}]
Let \(G\) be a graph,
$u_0\in\ell^\infty(X)$
and $f\in C\bigl([0,\infty);\ell^\infty(X)\bigr).$ 
Then, for any
$A\le \min\{0,\inf_Xu_0\}
\ \text{and}\ 
B\ge \max\{0,\sup_Xu_0\},$ there
 exist global pointwise solutions
$
\underline u,\ \overline u
\colon[0,\infty)\times X\to\R
$
of the \eqref{eq:C-D} with initial datum $u_0$ and forcing term $f$ such that 
\[
 a\leq\underline u
\le
\overline u\leq b
\qquad
\textup{on }[0,\infty)\times X
\]
with \[
a(t):=A-\int_0^t\|f^-(s,\cdot)\|_\infty\dd s
\qquad \text{and} \qquad
b(t):=B+\int_0^t\|f^+(s,\cdot)\|_\infty\dd s
\] where \(f^\pm\) denote the positive and negative parts of \(f\). Furthermore, every pointwise solution \(v\) of the \eqref{eq:C-D} on \([0,T]\times X\) for $T>0$ with the same data \((u_0,f)\) satisfying \(a\le v\le b\) satisfies
\[
 \underline u
\le v\le
\overline u
\qquad
\textup{on }[0,T]\times X.
\]
\end{mainthm}

We then identify conditions under which these pointwise solutions are genuine
Banach space classical solutions. Under the bounded degree assumption \BD, 
bounded pointwise solutions are \(\ell^\infty\)-classical and under the assumption that the entire
vertex set has finite measure \FM,
they are \(\ell^p\)-classical for every \(p\in[1,\infty]\), see 
Corollary~\ref{c:classical-solution}.
If, moreover, \(\phi\) is locally Lipschitz, the bounded degree assumption
gives global well-posedness in \(\ell^\infty(X)\), also in the presence of a forcing term, see Theorem~\ref{t:globalwp}.

For the rest of the introduction, we consider the homogeneous problem for the \eqref{eq:C-D}, i.e., when $f=0.$ The second part of the paper focuses on the porous nonlinearity $\phi=\phi_m$
and on  \(\ell^1\)-mild solutions. 
The Sobolev inequality is an important analytic tool on manifolds and graphs. We introduce a discrete analog \SP[\nu], $\nu>2$, of the Sobolev inequality
\[
\|u\|_{L^{\nu}}\leq C\|\nabla u\|_2.
\]
There are many graphs satisfying \SP[\nu] such as Cayley graphs of discrete groups of polynomial volume growth of order at least $N$ for $N > 2$, see Subsection~\ref{ssec:sobolev-examples}. Assuming a Sobolev inequality \SP[\nu]
with
\(\nu>2\), we prove quantitative energy estimates for the \(\ell^q\)-norms of
solutions, see Lemma~\ref{lem:energy-mild}.
These estimates imply finite-time extinction in the fast diffusion
range
$
0<m<\frac{2}{\nu}.
$
More specifically, we prove the following result.
\begin{mainthm}[{see Theorem~\ref{thm:flexible-extinction}}]
Assume that $G$ satisfies \SP[\nu] for some $\nu>2$ and let
    $0<m<\frac{2}{\nu}.$
For any $u_0\in\ellp{1}\cap\ellp{q}$ with $q\ge \frac{\nu(1-m)}{\nu-2},$ the mild solution $u$ 
of the \eqref{eq:C-D} with $f= 0$ and $\phi=\phi_m$
 has finite extinction time, i.e., there exists $T_1>0$ such that
\[
    u(t,\cdot)= 0
    \qquad\mbox{for every }\;
    t\ \ge\ T_1.
\]
\end{mainthm}
Note that, for 
Cayley graphs of polynomial volume growth \(N>2\), where the Sobolev exponent is
\(\nu=2N/(N-2)\), this gives the critical fast diffusion threshold of finite-time extinction
\[
m_c=\frac{N-2}{N}
\]
in agreement with the Euclidean exponent. In the complementary range
$
m>\frac{2}{\nu},
$
we establish \(\ell^1\)-\(\ell^q\) smoothing estimates for every \(q>1\),
see Theorem~\ref{thm:smoothing}. As a side note (Corollary~\ref{cor:minimal-pointwise-extinction}), we extend the extinction estimates to minimal positive pointwise solutions, even for initial data in $\ellp{\alpha}\cap\ell^\infty(X)$ that need not belong to $\ellone$.

Finally, we study the global mass balance for the homogeneous problem,
allowing an arbitrary killing term and distinguishing the mass
removed by killing from any additional loss at infinity. 
For the \eqref{eq:GPME} posed on the entire Euclidean space, conservation of mass is a classical feature of the Cauchy problem, in contrast with Dirichlet problems on proper subdomains, where part of the mass flows out through the boundary, see \cite[Proposition~9.15 and Subsection~3.3.3]{vazquez2007porous}. 
A similar property holds for finite subgraphs with Dirichlet boundary conditions.
On an infinite graph or an open manifold, the role of the boundary is played by infinity, and the corresponding dichotomy is governed by stochastic completeness.
Stochastic completeness at infinity becomes relevant in the presence of a killing term which instantly
removes heat at any vertex where it takes a positive value, thus making the
graph stochastically incomplete.
In this case, 
the notion of stochastic completeness at infinity ensures that the heat removed by the killing term
is recaptured and factored back into the total mass. Therefore, a killing term actually
helps stochastic completeness at infinity by removing the heat and capturing it
before it can reach the boundary at infinity and be killed, see \cite{KL12}
or \cite[Chapter~7]{KLW21} as well as \cite{MS20} for
the case of manifolds.

Thus, when considering cases with the killing term, 
we use the concept of stochastic completeness at infinity \SCinf{} 
and note that, for the case $\kappa= 0,$  
\SCinf{} reduces to standard stochastic completeness \SC. 
Under a growth condition \LG{} of \(\phi\) at the origin, which holds for
\(\phi_m\) with \(m\geq1\), we prove the generalized mass balance for nonnegative mild
solutions, that is, all mass loss is accounted for exactly by the killing term, without
any summability assumption on \(\kappa\).
More specifically, we prove the following result.

\begin{mainthm}[{see Theorem~\ref{thm:mild-generalized-mass}}]
Assume that \(G\) satisfies \SCinf{} and that
\begin{equation*}
    \mu(X)<\infty
    \qquad\text{or}\qquad
    \phi\text{ satisfies }\LG{}.
\end{equation*}
Then, for every
\(0\leq u_0\in\ell^{1}(X,\mu)\), the mild solution $u$ of the \eqref{eq:C-D} with $f= 0$ 
satisfies
\begin{equation*}
    \norm{u(t)}{1}
    +
    \int_0^t\sum_{x\in X}
        \kappa(x)\phi(u(s,x))\dd s
    =
    \norm{u_0}{1}\qquad  \text{for }
    t\ge0.
\end{equation*}
\end{mainthm}

We then prove the corresponding balance laws for classical and bounded
pointwise solutions, still allowing an arbitrary killing term. For
nonnegative solutions, the appropriate formulation is the generalized mass
balance displayed above. For possibly signed solutions, the corresponding
identity is the signed balance
\[
\sum_{x\in X}u(t,x)\mu(x)
+
\int_0^t\sum_{x\in X}\kappa(x)\phi(u(s,x))\dd s
=
\sum_{x\in X}u_0(x)\mu(x),
\]
whenever the killing contribution is absolutely integrable in time. If the
killing term vanishes, this reduces, without a sign restriction, to
conservation of signed mass:
\[
\sum_{x\in X}u(t,x)\mu(x)
=
\sum_{x\in X}u_0(x)\mu(x)
\qquad
\text{for } t \geq 0.
\]
For nonnegative solutions, conservation of signed mass is equivalently
\(\norm{u(t)}{1}=\norm{u_0}{1}\).
Our main result in this direction, Theorem~\ref{t:conservation},
establishes the appropriate mass-balance laws with an arbitrary killing
term in two settings. Under \SCinf, \UM{} and \C, every nonnegative
\(\ell^1\)-classical solution satisfies the generalized mass balance;
possibly signed \(\ell^1\)-classical solutions satisfy the differential
signed-mass identity and, whenever the killing contribution is absolutely
integrable in time, the integrated signed balance. Under \BD{} and \FM,
every bounded pointwise solution, possibly signed, satisfies the integrated
signed balance. Both statements follow from an abstract criterion,
Proposition~\ref{theorem:cons_mass}, which derives the balance from a
no-flux identity at infinity for \(\Phi u\). Consequently, when
\(\kappa=0\), every solution covered by the theorem conserves its signed
mass.

We emphasize that none of our main results require the underlying graphs to be locally finite. Local finiteness is typically imposed in several standard proof techniques to circumvent technical obstructions; our approach bypasses this requirement entirely by adding appropriate
conditions on the boundary data as needed.
Additionally, our framework naturally allows for killing terms. Consequently, our results apply to nonlocal operators and also Schrödinger operators with nonnegative potentials.

The manuscript is organized as follows. Section~\ref{sec:notation} contains the notation and
preliminaries used throughout the paper. We recall the weighted graph setting, the porous medium-type operators considered here, and the notions of pointwise, classical and mild solutions.
Section~\ref{sec:bounded-pointwise} is devoted to comparison principles and bounded pointwise solutions.
We prove the finite-domain exterior-Dirichlet comparison principle and the
weak maximum principle and then use monotone exhaustion to construct the
lower and upper extremal bounded pointwise solutions on infinite graphs.
We also give criteria ensuring that these pointwise solutions are classical
solutions and prove global \(\ell^\infty\)-well-posedness under bounded degree
and locally Lipschitz nonlinearities.

In Section~\ref{sec:extinction} we study the porous case \(\phi_m(s)=s^m\) for the
\(\ell^1\)-mild solutions associated with the maximal accretive realization.
After recalling the existence and uniqueness theory for such mild solutions,
we derive the basic energy inequality under the Sobolev assumption \SP[\nu].
This leads to the finite-time extinction theorem for \(0<m<2/\nu\) and to the
\(\ell^1\)-\(\ell^q\) smoothing theorem for \(m>2/\nu\).
Section~\ref{sec:conservation_of_mass} proves the generalized mass
balance on graphs that are stochastically complete at infinity, with
arbitrary killing: first for nonnegative  \(\ell^1\)-mild
solutions and then for suitable classical and bounded pointwise solutions,
recovering conservation of mass in the absence of killing.
Some auxiliary results used in Section~\ref{sec:extinction} are collected
in Appendix~\ref{sec:appendix}.

\section{Preliminaries}\label{sec:notation}
In this section, we lay out the  notation and key concepts. We begin by discussing graphs, alongside Laplacian and porous medium-type operators. We also discuss a specialization  of the Green's formula as well as different types of solutions that we will consider and introduce some additional assumptions on graphs that will be needed later.

\subsection{Function spaces and graphs}\label{ssec:function_spaces&graphs}

We start with notations for the identity operator and the Nemytskii operator arising from
composition of functions.
Given a countable set~$X$, we let $C(X)=\{ f\colon X \to \R\}$ denote the set
of all functions on $X$.
If $E\subseteq C(X)$,
we denote by $\id\colon E \to E$ the identity operator.
If $\phi \colon \dom\left(\phi\right)\subseteq \R \to \R$ is a function, then we denote by the capital letter $\Phi$ the canonical extension of $\phi$ to $E$,
i.e., the operator $\Phi \colon \dom\left(\Phi\right)\subseteq E \to C(X)$
given by
\begin{align*}
	\dom\left(\Phi\right)&= \left\{ u \in E \mid u(x) \in \dom\left(\phi\right) \textup{ for all } x \in X  \right\}\\
	\Phi u(x)&= \phi(u(x)).
\end{align*}
This is referred to as the Nemytskii operator associated to $\phi$.

Unless otherwise stated, we fix a continuous, monotone increasing function
\[
\phi \colon \R\to\R
\qquad \text{such that} \qquad
\phi(0)=0
\]
and use $\Phi u=\phi\circ u$ for the associated Nemytskii operator.
We do not assume that \(\phi\) is injective or surjective.

Given $u, v \in C(X)$, we write $u\geq v$ if $u(x)\geq v(x)$ for every $x \in X$. All other ordering symbols are defined accordingly. We say that a function $u$ is \emph{positive} if $u \geq 0$
and \emph{strictly positive} if $u>0$ and \emph{negative/strictly negative} when the
opposite inequalities hold.

We now introduce graphs and some associated function spaces.
For a detailed introduction to the graph setting as presented here, see \cite{KLW21}.

\begin{definition}\label{def:graph}
	A \textit{graph} is a quadruple $G=(X,w,\kappa,\mu)$ given by
	\begin{itemize}
		\item a countable set of \emph{vertices} $X$,
		\item a positive symmetric \emph{edge-weight} function $w\colon X\times X \to [0,\infty)$,
		\item a positive \emph{killing term} $\kappa \colon X \to [0,\infty)$,
		\item a strictly positive \emph{vertex measure} $\mu \colon X \to (0,\infty)$,
	\end{itemize}
	where the edge-weight function $w$ satisfies:
	\begin{enumerate}
		\item[(i)]\label{assumption:symmetry} Symmetry: $w(x,y)=w(y,x)$ for every $x,y \in X$.
		\item[(ii)]\label{assumption:loops} No loops: $w(x,x)=0$ for every $x \in X$.
		\item[(iii)]\label{assumption:degree} Local summability: $\sum_{y\in X} w(x,y) < \infty$ for every $x \in X$.
	\end{enumerate}
\end{definition}
These are the basic assumptions on every graph we consider here. Whenever the vertex set $X$ is finite, we call $G$ a \emph{finite graph}.
We call a sequence of finite subsets $(\Om_n)$
such that $\Om_n \subseteq \Om_{n+1}$ and $\bigcup_n \Om_n = X$
an \emph{exhaustion} of the graph.
We think of $x, y \in X$ with $w(x,y)>0$ as being \emph{connected}
by an edge with weight $w(x,y)$, write $x \sim y$, and call $x$ and $y$
\emph{neighbors} in this case.
We say that a graph is \emph{locally finite} if every vertex
has finitely many neighbors, i.e.,
$| \{ y \mid y \sim x \}| < \infty$
for every $x \in X$. We note that we never 
impose the assumption of local finiteness on the graph in this paper.

For $x \in X$,
we define the degree by
\[
\Deg(x):=\frac{1}{\mu(x)}
\left(\sum_{y \in X}w(x,y)+\kappa(x)\right).
\]
The local summability assumption implies that $\Deg(x)$ is finite at every
vertex.

A \emph{finite path} is a finite sequence $(x_0,\ldots,x_N)$ such that
$x_j\sim x_{j+1}$ for every $j=0,\ldots,N-1$.
We say that a graph is \emph{connected}
if for any two distinct vertices $x, y \in X$, there exists a finite path
that starts at $x$ and ends at $y$.
The \emph{length} of a finite path is the number of its edges.
The \emph{combinatorial graph distance} is the length of the
shortest path between two vertices $x, y \in X$. This is denoted by $d(x,y)$
where we note that $d(x,x)=0$.

Unless explicitly stated otherwise, all graphs considered below are assumed to
be connected.

If $\Omega\subseteq X$, we let
\[
\partial_e \Omega=\{x\not \in \Omega \mid \exists y \in \Omega \, \text{ such that } y \sim x \}
\]
and call it the \emph{exterior boundary}.
We write
\[
\overline\Omega:=\Omega\cup\partial_e\Omega.
\]

We have the following elementary exhaustion fact which is non-trivial as we do not assume local finiteness, see~\cite[Lemma A.4]{bianchi2022generalized} for an explicit construction.
\begin{lemma}\label{lem:exhaustion}
Every connected  graph admits an increasing exhaustion by finite connected subsets.
\end{lemma}

Let $C_c(X)$ denote the set of finitely
supported functions in $C(X)$. For $\Omega\subset X$, set 
\[
C_c(\Omega)
:=
\{u\in C_c(X) \mid \supp u\subseteq\Omega\}.
\]

For $\Omega \subseteq X$, we introduce the usual $\ell^p$ spaces by
\begin{align*}
\ell^p(\Omega,\mu)&= \{ f \in C(\Om) \mid \sum_{x \in \Omega} |f(x)|^p \mu(x)< \infty \}	&\mbox{for } p \in [1,\infty)\\
\ell^\infty(\Omega) &= \{ f\in C(\Omega) \mid \sup_{x \in \Omega}
|f(x)| < \infty\} &\mbox{for } p=\infty
\end{align*}
	with  norms
\[
\|f\|_p= \begin{cases}
	\left(\sum_{x\in \Omega} |f(x)|^p\mu(x)\right)^{{1}/{p}} & \mbox{for } p \in [1,\infty)\\
	\sup_{x\in \Omega}|f(x)| & \mbox{for } p =\infty.
\end{cases}
\]
Moreover, we use the standard notation for  the bilinear form
\[
\langle f, g \rangle_\Omega = \sum_{x\in \Omega} f(x)g(x)\mu(x),
\]
defined whenever the sum converges absolutely. This form
induces the Hilbert space structure on $\ell^2(\Omega, \mu)$.

\subsection{The formal Laplacian, energy form and Green's formula}
We now introduce the family of operators that we will consider. The full formal
domain for the Laplacian is
\[
\mathcal F
:=
\left\{
f\in C(X) \mid 
\sum_{y\in X}w(x,y)|f(y)|<\infty
\text{ for every }x\in X
\right\}.
\]
For $f\in\mathcal F$, the nonnegative formal (graph) Laplacian is defined as
 \[
 \Delta f(x)
 =
 \frac1{\mu(x)}
 \sum_{y\in X}w(x,y)(f(x)-f(y))
 +
 \frac{\kappa(x)}{\mu(x)}f(x)
 \]
 for $x \in X$.

The associated energy form is
\begin{equation}\label{eq:Q-def}
\Q(f,g)
=
\frac12\sum_{x,y\in X} w(x,y)\bigl(f(x)-f(y)\bigr)\bigl(g(x)-g(y)\bigr)
+
\sum_{x\in X}\kappa(x)f(x)g(x)
\end{equation}
defined whenever the sums converge absolutely. We write $\Q(f):=\Q(f,f)$.

For $v\in C(X)$, the \emph{restriction} of $v$ to $\Omega$ is
\[
\projom v := v_{|_\Omega}\in C(\Omega).
\]
For $u\in C(\Omega)$, $\embom u\in C(X)$, the \emph{zero extension} of $u$ is
\[
\embom u(x):=
\begin{cases}
u(x) & x\in\Omega\\
0 & x\in X\setminus\Omega.
\end{cases}
\]
Both maps are linear and satisfy $\projom\embom=\operatorname{id}_{C(\Omega)}$.
When $\Omega=X_n$, we write $\projn$ and $\embn$ for the corresponding
restriction and zero extension maps.

Given an exterior datum 
$\eta \colon X\setminus\Omega\to\R$, the \emph{boundary extension}
$E_\Omega^\eta \colon C(\Omega)\to C(X)$ is defined as
\[
E_\Omega^\eta u(x):=
\begin{cases}
u(x) & x\in\Omega\\
\eta(x) & x\in X\setminus\Omega.
\end{cases}
\]
In particular, $E_\Omega^{0}=\embom$ and $E_\Omega^\eta u$
is affine in $u$. Constant exterior data are identified with the corresponding
constant functions on $X\setminus\Omega$. If $\eta$ is prescribed only on
$\partial_e\Omega$, we set $\eta:=0$ on $X\setminus\overline{\Omega}$.

For a map $\theta\colon\R\to\R$ and an exterior datum
$\eta\colon X\setminus\Omega\to\R$, we call the datum \emph{$(\theta,\Omega)$-admissible} if
\[
\sum_{y\in X\setminus \Omega}w(x,y)|\theta(\eta(y))| < \infty
\qquad \mbox{for every }x\in\Omega.
\]

Let $\Omega\subseteq X$ be finite. Given an
$(\operatorname{id},\Omega)$-admissible exterior datum $\eta$, the
exterior-Dirichlet Laplacian with datum $\eta$,
$\Deltadir[\Omega][\eta] \colon C(\Omega) \to C(\Omega)$, is given by
\[
\begin{aligned}
\Deltadir[\Omega][\eta] u(x)
&:=
\frac1{\mu(x)}
\sum_{y\in\Omega}w(x,y)(u(x)-u(y))
+
\frac1{\mu(x)}
\sum_{y\in X\setminus\Omega}w(x,y)(u(x)-\eta(y))
+
\frac{\kappa(x)}{\mu(x)}u(x)
\end{aligned}
\]
for $x\in\Omega$.
The zero-Dirichlet operator is $\Deltadir[\Omega]:=\Deltadir[\Omega][0]$. 
When an exhaustion $(\Omega_n)$ is given, we use the shorthand
$\Deltadir[n][\eta]:=\Deltadir[\Omega_n][\eta]$ and
$\Deltadir[n]:=\Deltadir[n][0]$.
For a finite $\Omega\subseteq X$ and $f,g\in C(\Omega)$, the 
Dirichlet energy form is
\[
\Qdir[\Omega](f,g):=\Q(\embom f,\embom g).
\]
Equivalently,
\[
\Qdir[\Omega](f,g)
=
\frac12
\sum_{x,y\in\Omega}
w(x,y)(f(x)-f(y))(g(x)-g(y))+
\sum_{\substack{x\in\Omega\\ y\notin\Omega}}
w(x,y)f(x)g(x)
+
\sum_{x\in\Omega}\kappa(x)f(x)g(x).
\]

The following localized Green's formula extends \cite[Proposition~1.5]{KLW21}.
It allows us to consider functions that are summable on $\Om$ only. 
Although standard, we give a proof for the convenience of the reader.
\begin{proposition}[Green's formula]\label{prop:greenf}
Let $\Omega\subseteq X$, let $\psi\in C_c(\Omega)$, and let $V\in C(X)$ be such that
\[
\sum_{y\in X}w(x,y)|V(y)|<\infty \quad \mbox{for every } x\in \Omega.
\]
Then,
\begin{equation}\label{eq:localized-green}
\langle\Delta V,\psi\rangle_\Omega=\Q(V,\psi).
\end{equation}
In particular, if $\Omega$ is finite, then for all $f,g\in C(\Omega)$,
\[
\langle\Deltadir[\Omega] f,g\rangle_\Omega
=
\Qdir[\Omega](f,g).
\]
\end{proposition}

\begin{proof}
Since $\psi\in C_c(\Omega)$, its support
\[
S:=\supp\psi
\]
is a finite subset of $\Omega$.  By the hypothesis on $V$ and the local summability property of the
edge-weight of the graph, then, for every $x\in S$,
\begin{equation}\label{eq:V_summable}
\sum_{y\in X}w(x,y)|V(x)-V(y)|
\le
|V(x)|\sum_{y\in X}w(x,y)
+
\sum_{y\in X}w(x,y)|V(y)|
<\infty.
\end{equation}
Using the definition of $\Delta$, and using that $\psi=0$ outside $\Omega$, we
get
\[
\begin{aligned}
\langle \Delta V,\psi\rangle_\Omega
&=
\sum_{x\in\Omega}
\left[
  \frac1{\mu(x)}
  \sum_{y\in X}w(x,y)(V(x)-V(y))
  +
  \frac{\kappa(x)}{\mu(x)}V(x)
\right]\psi(x)\mu(x)
\\
&=
\sum_{x\in X}\sum_{y\in X}
w(x,y)(V(x)-V(y))\psi(x)
+
\sum_{x\in X}\kappa(x)V(x)\psi(x).
\end{aligned}
\]
By \eqref{eq:V_summable}, and since $S$ is finite, all sums above are absolutely convergent.
By symmetry of $w$,
\[
\sum_{x,y\in X}w(x,y)(V(x)-V(y))\psi(y)
=
-\sum_{x,y\in X}w(x,y)(V(x)-V(y))\psi(x).
\]
Therefore,
\[
\begin{aligned}
\sum_{x,y\in X}w(x,y)(V(x)-V(y))\psi(x)
&=
\frac12
\sum_{x,y\in X}
w(x,y)(V(x)-V(y))(\psi(x)-\psi(y)).
\end{aligned}
\]
Consequently, by the Fubini--Tonelli theorem,
\[
\langle \Delta V,\psi\rangle_\Omega
=
\frac12
\sum_{x,y\in X}
w(x,y)(V(x)-V(y))(\psi(x)-\psi(y))
+
\sum_{x\in X}\kappa(x)V(x)\psi(x)
=
\Q(V,\psi).
\]

If now $\Omega$ is finite and $f,g\in C(\Omega)$, then
\[
\embom f,\embom g\in C_c(X)\subseteq \mathcal F.
\]
Moreover,
\[
\Deltadir[\Omega] f=\projom\Delta(\embom f)
\]
and, by definition,
\[
\Qdir[\Omega](f,g)=\Q(\embom f,\embom g).
\]
Applying the first part with $V=\embom f$ and
$\psi=\embom g$ gives
\[
\begin{aligned}
\langle \Deltadir[\Omega] f,g\rangle_\Omega
=
\langle \Delta(\embom f),\embom g\rangle_\Omega
=
\Q(\embom f,\embom g)
=
\Qdir[\Omega](f,g).
\end{aligned}
\]
This proves the proposition.
\end{proof}

\subsection{Porous medium-type operators}
Next, we define the \emph{formal porous medium-type operator}. With the fixed
nonlinearity $\phi$, we denote by $\Delta \Phi$ the formal porous
medium-type operator
with domain given by 
\[
\mathcal{F}_\Phi
=
\{ f\in C(X) \mid \Phi f \in \mathcal F \}
\]
and acting as
\[
\Delta \Phi f(x) = \Delta (\phi \circ f)(x)
\]
for $x \in X$ and $f \in \mathcal{F}_\Phi$.

For $p\in[1,\infty]$, we let
\[
\dom(\Deltaphip)
=
\left\{
u\in \ell^p(X,\mu)\cap \mathcal{F}_\Phi
\mid \Delta\Phi u\in \ell^p(X,\mu)
\right\}
\]
and let $\Deltaphip u=\Delta\Phi u$ for $u \in \dom(\Deltaphip)$.
For $p=1$, we write simply $\Deltaphi$ for $\mathcal L^{(1)}$.

Typical examples of $\phi$ include
\[
\phi_m(s)=s^m:=s|s|^{m-1},
\qquad m>0.
\]

For a finite $\Omega\subseteq X$ and a $(\phi,\Omega)$-admissible exterior
datum $\eta$, the exterior datum $\phi\circ\eta$ is
$(\operatorname{id},\Omega)$-admissible. We therefore define
\[
\Deltadir[\Omega][\eta]\Phi u
:=
\Deltadir[\Omega][\phi\circ\eta](\phi\circ u)
\]
for $u\in C(\Omega)$.

\subsection{Solutions}\label{ssec:solutions}
We collect here the definitions of solutions of the~\eqref{eq:C-D} problem
used throughout this manuscript, namely pointwise, classical, and mild
solutions.

Throughout, we identify a space-time function $u(t,x)$ with the curve
$t\mapsto u(t):=u(t,\cdot)$ taking values in $\ell^p(X,\mu)$. Accordingly,
$C([0,T];\ell^p(X,\mu))$ and $C^1((0,T);\ell^p(X,\mu))$ denote, respectively,
the set of
continuous and continuously differentiable functions with respect to the
$\ell^p$-norm. When $\ell^p(X,\mu)$ is replaced by $\R$ in the above,
we simply omit it in the notations.

We begin with \emph{pointwise solutions}, that is, functions satisfying the
equation pointwise.

\begin{definition}[Pointwise solution]\label{def:pointwise_solution}
Let $u_0\in C(X)$ and $f(\cdot,x)\in C((0,T))$ for every $x\in X$. We say that
$u\colon [0,T]\times X\to\mathbb{R}$ is a \emph{pointwise solution} of
the~\eqref{eq:C-D} problem with data \((u_0,f)\) if
\begin{itemize}
    \item $u(t, \cdot)\in \mathcal{F}_\Phi$ for every $t\in(0,T)$;
    \item $u(\cdot,x)\in C([0,T])\cap C^1((0,T))$ for every $x\in X$;
    \item $(\partial_t+\Delta\Phi)u(t,x)=f(t,x)$ for every $(t,x)\in(0,T)\times X$;
    \item $u(0,x)=u_0(x)$ for every $x\in X$.
\end{itemize}
\end{definition}

We next introduce classical solutions, for which we require additional
continuity in the $\ell^p$-norm.

\begin{definition}[$\ell^p$-classical solution]\label{def:classical_solution}
Let $p\in[1,\infty]$, $u_0\in\ell^p(X,\mu)$ and
$f\in C((0,T);\ell^p(X,\mu))$. We say that
$u\colon [0,T]\to\ell^p(X,\mu)$ is an $\ell^p$-\emph{classical solution} of
the~\eqref{eq:C-D} problem with data \((u_0,f)\) if
\begin{itemize}
    \item $u\in C([0,T];\ell^p(X,\mu))\cap C^1((0,T);\ell^p(X,\mu))$;
    \item $u(t)\in\dom(\Deltaphip)$ for every $t\in(0,T)$;
    \item $(\partial_t+\Delta\Phi)u(t)=f(t)$ in $\ell^p(X,\mu)$ for every $t\in(0,T)$;
    \item $u(0)=u_0$ in $\ell^p(X,\mu)$.
\end{itemize}
\end{definition}

\begin{remark}[Classical solutions are pointwise solutions]
    Since $\mu(x)>0$ for every $x\in X$, equality in $\ell^p(X,\mu)$ implies
pointwise equality at every vertex. Moreover, for $p<\infty$ and $h \in \ell^p(X,\mu)$,
\[
|h(x)|\le \mu(x)^{-1/p}\|h\|_p,
\]
so $\ell^p$-continuity and differentiability imply pointwise continuity
and differentiability. For $p=\infty$, this is immediate from the supremum
norm. Hence, every $\ell^p$-classical solution is a pointwise solution.
\end{remark}

Next, we work towards introducing mild solutions (see e.g.,
\cite{barbu2010nonlinear} and \cite{benilan1988evolution}), whose definition
relies on the definition of $\eps$-approximate solutions given below.

\begin{definition}[$\eps$-discretization]\label{def:epsilon-discretization}
    Given a time interval $[0,T]$, with $T<\infty$, and a forcing term
    $f \in L^1_{\rm{loc}}([0,T] ; \ell^p\left(X,\mu\right))$ with $1\leq p<\infty$,
    we choose a partition of the time interval
    \[
    \mathcal{T}_N\coloneqq\{t_k\}_{k=0}^N,
    \qquad
    0=t_0<t_1<\cdots<t_N=T,
    \]
    and a sequence of approximate forcing terms
    \[
    \boldsymbol{f}_N\coloneqq\{f_k\}_{k=1}^{N},
    \qquad
    f_k\in\ell^p(X,\mu).
    \]
    Having fixed $\eps >0$, we call $\mathcal{D}_\eps\coloneqq (\mathcal{T}_N,\boldsymbol{f}_N)$ an \emph{$\eps$-discretization} of $([0,T]; f)$ if
    \begin{itemize}
        \item $t_k-t_{k-1} \leq \eps$ for every $k=1,\ldots, N$;
        \item $\sum_{k=1}^{N} \int_{t_{k-1}}^{t_{k}} \left\| f(t) - f_k \right\|_p dt \leq \eps$.
    \end{itemize}
\end{definition}

Given an $\eps$-discretization $\mathcal{D}_\eps$, we consider the
following system of difference equations which arises from an implicit
Euler-discretization of the \eqref{eq:C-D}:
\begin{equation}\label{implicit_Euler}
    (\id + \lambda_k \Deltaphip)u_k= u_{k-1} + \lambda_k f_k,
    \qquad
    \lambda_k\coloneqq t_k - t_{k-1}\mbox{ and } k=1,\ldots,N
\end{equation}
with $u_0\in\ell^p(X,\mu)$ given and $u_k \in \dom(\Deltaphip)$.

\begin{definition}[$\eps$-approximate solution]\label{epsilon-approximation}
    If the system \eqref{implicit_Euler} admits a solution
    $\boldsymbol{u}_\eps =\left\{u_k\right\}_{k=1}^N$ such that
    $u_k\in \dom(\Deltaphip)$ for every $k=1,\ldots, N$, then we define
    $u_\eps$ as the piecewise constant function
    \begin{equation}\label{epsilon_approximation}
    u_\eps(t)\coloneqq \begin{cases}
        \sum_{k=1}^N u_k\mathds{1}_{(t_{k-1},t_k]}(t) & \mbox{for } t\in (0, T]\\
        u_0 & \mbox{for } t=0
    \end{cases}
    \end{equation}
    and we call $u_\eps$ an \emph{$\eps$-approximate solution} of
    \eqref{eq:C-D} (subordinate to $\mathcal{D}_\eps$). 
\end{definition}

Finally, we come to the definition of mild solutions which are
uniform limits of $\eps$-approximate solutions. This notion of solutions is particularly natural in the framework of accretive operators. See for example~\cite{benilan1988evolution,barbu2010nonlinear}.

\begin{definition}[Mild solution]\label{def:weak_solution}
    Let $u_0\in\ell^p(X,\mu)$ and
    $f \in L^1_{\rm{loc}}([0,T] ; \ell^p\left(X,\mu\right))$ with $1\leq p<\infty$.
    If $T<\infty$, we say that $u \colon [0,T]\to \ell^p\left(X,\mu\right)$
    is a \emph{mild solution} of the \eqref{eq:C-D} with data \((u_0,f)\) if
    $u \in C\left( [0,T]; \ell^p\left(X,\mu\right)\right)$ and $u$ is obtained
    as a uniform limit of $\eps$-approximate solutions. Namely, for every
    $\eps>0$, there exists an $\eps$-discretization $\mathcal{D}_\eps$
    of $([0,T]; f)$, as in Definition \ref{def:epsilon-discretization}, and an
    $\eps$-approximate solution $u_\eps$ subordinate to
    $\mathcal{D}_\eps$, as in \eqref{epsilon_approximation}, such that
    \begin{equation*}
    \left\|u(t) - u_\eps(t) \right\|_p < \eps
    \qquad \mbox{for every } t\in [0,T].
    \end{equation*}

    If $T = \infty$, then we say that $u$ is a mild solution of the \eqref{eq:C-D}
    with data \((u_0,f)\)
    if the restriction of $u$ to each compact subinterval
    $[0,a]\subset [0,\infty)$ is a mild solution of the \eqref{eq:C-D} on $[0,a]$.
\end{definition}
Notice that $u_\eps(t) \in \dom(\Deltaphip)$ for every $t$, whereas a mild solution $u(t)$ may not be.

\begin{remark}[Classical solutions are mild solutions]
    Mild solutions as defined above are also known in the literature as $C^0$
solutions, see \cite[Chapter IV.8]{showalter2013monotone}. Every classical solution  is a mild solution;
see, e.g., \cite[Theorem 1.4]{benilan1988evolution}. 
\end{remark}

\subsection{Assumptions on the graph}\label{sec:assumptions}
We shall use the following additional hypotheses only when named:
\begin{enumerate}[label={}, leftmargin=*, align=left]
    \conditem{BD}{\textbf{(BD)}} $\displaystyle \sup_{x\in X}\Deg(x)<\infty$. \hfill(``Bounded degree'')
    \conditem{FM}{\textbf{(FM)}} $\mu(X)<\infty$. \hfill(``Finite measure'')
    \conditem{SP}{\textup{\textbf{(S\textsubscript{$\nu$})}}}
    For $\nu>2$, there is $C_\nu>0$ such that
    $\|\psi\|_\nu^2\le C_\nu\,\Q(\psi)$ for all $\psi\in C_c(X)$.
    \hfill(``$\nu$-Sobolev'')
    \conditem{UM}{\textbf{(UM)}} $\displaystyle \inf_{x\in X}\mu(x)>0$. \hfill(``Uniformly positive measure'')
   \conditem{SCinf}{\textup{\textbf{(SC\textsubscript{$\infty$})}}}\label{def:SC-infinity} $G$ is stochastically complete at infinity, i.e.,  if $h\in\ell^\infty(X)$ satisfies
   $h+\lambda\Delta h=0$ with $\lambda>0$,  then $h=0$.  \hspace*{\fill}(``Stochastic completeness at infinity'')
   \conditem{SC}{\textbf{(SC)}} $G$ is stochastically complete, i.e., the minimal heat semigroup for $\kappa=0$ preserves the constant function $1$. \hspace*{\fill}(``Stochastic completeness'')
   \conditem{C}{\textbf{(C)}} $\Phi(\ell^1(X,\mu))\subseteq\ell^1(X,\mu)$. \hfill(``$\ell^1$-containment'')
   \conditem{LG}{\textbf{(LG)}}\label{eq:LG} 
   $\limsup_{r\downarrow0}\frac{\phi(r)}{r}<\infty.$ \hfill(``(Super)-linear growth at $0$'')
\end{enumerate}

\section{Comparison principles and bounded pointwise solutions}
\label{sec:bounded-pointwise}

In this section, we study the generalized porous medium equation on finite subsets $\Omega$ with Dirichlet boundary data, and then, by exhaustion, construct extremal pointwise solutions of the global problem \eqref{eq:C-D}, defined for all times, with initial data in $\ell^\infty$ and forcing terms which are continuous in time.

For a finite $\Omega\subset X$, let 
\[
  \GammaT=(0,T]\times\Omega,
  \qquad
  \partialpGammaT
  =
  (\{0\}\times\Omega)\cup([0,T]\times\partial_e\Omega),
\] where $\GammaT$ is a discrete analog of the parabolic cylinder based on $\Omega$ 
and $\partialpGammaT$ is the parabolic boundary of $\GammaT$.
We also use the shorthand
\[
\CtGammaT
:=
\{u \colon [0,T]\times\Omega\to\R \mid
u(\cdot,x)\in C([0,T])\cap C^1((0,T))\ \text{for all } x\in\Omega\}.
\]
We also set
\[
\CtGammaTClosure
:=
\{u \colon [0,T]\times\overline\Omega\to\R \mid
u_{|_{[0,T]}\times\Omega}\in \CtGammaT\}.
\]
When a function on a finite $\Omega$ is used as a test function or in a finite
Dirichlet form, we identify it with its zero extension in $C_c(X)$.

\subsection{Comparison principles}
For the \eqref{eq:C-D} on Euclidean spaces, several parabolic comparison principles are available. We mention V\'azquez's monograph~\cite{vazquez2007porous}, in particular Theorem~5.5, as well as~\cite[Theorem~3.3]{kinnunen2016perron} and~\cite[Theorem~3.1]{avelin2017comparison}. In the graph setting, several others can be found with respect to closely related equations, for instance in~\cite{ma2022porous,biagi2025phragmen,berchio2026semilinear}.

In this subsection we state and prove an exterior-Dirichlet comparison principle tailored to the finite set exhaustion construction used in our proof of pointwise solutions. The proof's idea adapts that of \cite[Proposition 3.5]{vazquez2007porous}, with necessary modifications since the chain rule is unavailable in the discrete setting.
We note that the existence of such solutions on finite graphs will be addressed later when we construct
solutions on the entire graph via an exhaustion argument.

\begin{theorem}[Comparison principle]\label{thm:comparison}
Let $\Om\subseteq X$ be finite. For $i=1,2$, let
\[
g_i \colon [0,T]\times\partial_e\Omega\to\R
\]
be such that $g_i(t,\cdot)$ is $(\phi,\Omega)$-admissible for every $t\in[0,T]$ and
let $u_i\in \CtGammaT$ solve
\begin{equation}\label{eq:dirichlet-comparison}
\begin{cases}
  \left(\pt +\Deltadir[\Om][g_i(t,\cdot)]\Phi\right) u_i=f_i
  & \textup{on }(0,T)\times\Om\\
  u_i(0,\cdot)=u_{i,0} & \textup{on }\Om.
\end{cases}
\end{equation}
Assume
\[
  f_1\le f_2\quad\textup{on }(0,T)\times\Om,\qquad
  g_1\le g_2\quad\textup{on }[0,T]\times\partial_e\Om\qquad \text{and} \qquad
  u_{1,0}\le u_{2,0}\quad\textup{on }\Om.
\]
Then,
\[
  u_1\le u_2\qquad\textup{on }[0,T]\times\Om.
\]
\end{theorem}

\begin{proof}
Set
\[
  U_i(t,\cdot):=E_\Omega^{g_i(t,\cdot)}u_i(t,\cdot),
  \qquad
  V(t,\cdot):=\Phi U_1(t,\cdot)-\Phi U_2(t,\cdot),
\]
and let $u=u_1-u_2$ and $f=f_1-f_2$. Then, on $(0,T) \times \Omega$,
\[
  \pt u+\Delta V=f\le0.
\]
Moreover, $u(0,\cdot)\le0$ and, since $g_1\le g_2$ and $\phi$ is monotone increasing,
\[
V\le0
\qquad\text{on }[0,T]\times\partial_e\Omega.
\]

For $r\in\mathbb R$, write $r_+=\max\{r,0\}$ and define
\[
  P(t)=\sum_{x\in\Om}u_+(t,x)\mu(x).
\]
We will show that $P(t)=0$
for all $t \in [0,T]$, which implies $u_+=0$ and thus completes the proof.

Since $r\mapsto r_+$ is Lipschitz and $u(\cdot,x)$ is $C^1$ on $(0,T)$,
$P$ is locally absolutely continuous on $(0,T)$ 
and
for a.e. $t\in(0,T)$,
\[
  P'(t)=\sum_{\{x\in\Om \ \mid \,u(t,x)>0\}}\pt u(t,x)\mu(x).
\]
Let $\chi_t=\mathds{1}_{\{u(t,\cdot)>0\}}\in C(\Omega)$. By the
$(\phi,\Omega)$-admissibility of $g_1(t,\cdot)$ and $g_2(t,\cdot)$ and the
finiteness of $\Omega$,
$V(t,\cdot)$ satisfies the local summability hypothesis of the localized Green's formula
Proposition~\ref{prop:greenf} on $\Omega$ which gives
\begin{equation}\label{eq:P-prime-estimate}
  P'(t)=\langle f(t,\cdot),\chi_t\rangle_\Om
        -\langle \Delta V(t,\cdot),\chi_t\rangle_\Om
      =
        \langle f(t,\cdot),\chi_t\rangle_\Om
        -I(t)
\end{equation}
where
\[
\begin{aligned}
I(t)
&:=
\frac12\sum_{x,y\in\Omega}
w(x,y)(V(t,x)-V(t,y))(\chi_t(x)-\chi_t(y))\\
&\quad+
\sum_{\substack{x\in\Omega\\ y\in\partial_e\Omega}}
w(x,y)(V(t,x)-V(t,y))\chi_t(x)
+
\sum_{x\in\Omega}\kappa(x)V(t,x)\chi_t(x).
\end{aligned}
\]

We claim that
\begin{equation}\label{eq:Q-v-chi-nonnegative}
  I(t)\ge0.
\end{equation}
Indeed, for internal edges, if $\chi_t(x)=1$ and $\chi_t(y)=0$, then
$u(t,x)>0$ and $u(t,y)\le0$. Since $\phi$ is monotone increasing, this
implies $V(t,x)\ge0$ and $V(t,y)\le0$. Hence
\[
  [V(t,x)-V(t,y)][\chi_t(x)-\chi_t(y)]\ge0.
\]
The same conclusion holds when $\chi_t(x)=0$ and $\chi_t(y)=1$ and the
internal contribution is zero when $\chi_t(x)=\chi_t(y)$. For exterior edges,
only vertices $x\in\Omega$ with $\chi_t(x)=1$ contribute and then
$V(t,x)\ge0$ while $V(t,y)\le0$ because $g_1\le g_2$ on
$\partial_e\Omega$. The killing contribution is nonnegative because
$\kappa\ge0$ and $V(t,x)\ge0$ on $\{\chi_t=1\}$. This proves
\eqref{eq:Q-v-chi-nonnegative}.

Since $f\le0$ and $\chi_t\ge0$, \eqref{eq:P-prime-estimate} yields
$P'(t)\le0$ for a.e. $t\in(0,T)$. Thus, $P(t)\le P(s)$ whenever
$0<s<t<T$. Letting $s\downarrow0$ and using the continuity of $P$ together
with $P(0)=0$, we obtain $P(t)=0$ on $[0,T)$; the endpoint $t=T$ follows by
continuity. Therefore, $u_1\le u_2$ on $[0,T]\times\Omega$.
\end{proof}

A function $u\in \CtGammaTClosure$ is said to have a
\emph{$(\phi,\Omega)$-admissible exterior trace} 
if
$u(t,\cdot)|_{\partial_e\Omega}$ is $(\phi,\Omega)$-admissible for every
$t\in[0,T]$.
A function $u\in \CtGammaTClosure$ with a
$(\phi,\Omega)$-admissible exterior trace is called a
subsolution, respectively, supersolution, of the \eqref{eq:GPME} on $\GammaT$ if
\[
\left(\partial_t +\Deltadir[\Omega][u|_{\partial_e\Omega}]\Phi\right) u\le0,
\qquad\text{respectively,}\qquad
\left(\partial_t +\Deltadir[\Omega][u|_{\partial_e\Omega}]\Phi\right) u\ge0
\]
on $(0,T)\times\Omega$.

As a consequence of the last result, we now establish
comparisons between solutions based on their behavior on the parabolic
boundary as well as uniqueness of solutions. We will later
address the existence of solutions which follows by general theory.
\begin{corollary}\label{c:comparison}
Let $\Omega\subseteq X$ be finite and let
$u_1,u_2\in \CtGammaTClosure$ have
$(\phi,\Omega)$-admissible traces. Assume
\[
    \left(\partial_t  +\Deltadir[\Omega][u_1|_{\partial_e\Omega}]\Phi \right) u_1
    \leq
    \left(\partial_t  + \Deltadir[\Omega][u_2|_{\partial_e\Omega}]\Phi \right) u_2
    \qquad \text{on }(0,T)\times\Omega,
\]
and
\[
    u_1\le u_2
    \qquad \text{on }\partialpGammaT.
\]
Then, $u_1\leq u_2$ on $[0,T]\times\Omega$. In particular, ordered
subsolutions and supersolutions remain ordered under the same
$(\phi,\Omega)$-admissibility hypothesis.

Moreover, if
\[
    \left(\partial_t  +\Deltadir[\Omega][u_1|_{\partial_e\Omega}]\Phi \right) u_1
    =
    \left(\partial_t  + \Deltadir[\Omega][u_2|_{\partial_e\Omega}]\Phi \right) u_2
    \qquad \text{on }(0,T)\times\Omega
\]
and $u_1=u_2$ on $\partialpGammaT$, then $u_1=u_2$ on
$[0,T]\times\Omega$. That is, for fixed forcing, initial data and exterior
boundary data, the exterior-Dirichlet problem on a finite $\Omega$ admits at
most one solution in $\CtGammaTClosure$ with a
$(\phi,\Omega)$-admissible exterior datum and, since
$T>0$ is arbitrary, any two such solutions coincide on every common interval
of existence.
\end{corollary}
\begin{proof}
Apply Theorem~\ref{thm:comparison} with
$g_i=u_i|_{[0,T]\times\partial_e\Omega}$ and
$f_i=\left(\partial_t +\Deltadir[\Omega][g_i(t,\cdot)]\Phi \right) u_i$ on $\Omega$.
The uniqueness statement follows by applying the comparison twice, with the
roles of $u_1$ and $u_2$ interchanged.
\end{proof}

Next we show that sub/super solutions attain their extrema on the boundary.

\begin{lemma}[Weak maximum principle]\label{l:boundary_max}
Let $\Omega\subseteq X$ be finite. If
$u\in \CtGammaTClosure$ is a subsolution of the \eqref{eq:GPME} on
$\GammaT$ with a $(\phi,\Omega)$-admissible exterior trace, then
\[
  \sup_{[0,T]\times\Omega}u
  \le
  \max\left\{0,\sup_{\partialpGammaT}u\right\}.
\]
If $u\in \CtGammaTClosure$ is a supersolution of the
\eqref{eq:GPME} on $\GammaT$ with a $(\phi,\Omega)$-admissible exterior
trace, then
\[
  \inf_{[0,T]\times\Omega}u
  \ge
  \min\left\{0,\inf_{\partialpGammaT}u\right\}.
\]
If $\kappa =0$, the sharper boundary-only bounds hold:
\[
  \sup_{[0,T]\times\Omega}u
  \le
  \sup_{\partialpGammaT}u
  \qquad \text{and} \qquad
  \inf_{[0,T]\times\Omega}u
  \ge
  \inf_{\partialpGammaT}u.
\]
\end{lemma}

\begin{proof}
We address the estimate for the subsolution first.
If $\sup_{\partialpGammaT}u=\infty$, the upper estimate is trivial.
Otherwise, set
\[
C:=\max\left\{0,\sup_{\partialpGammaT}u\right\}.
\]
Every spatially constant function has a $(\phi,\Omega)$-admissible exterior
trace by the local summability of $w$. The constant function $C$ is a
supersolution because $C\ge0$ and
$(\partial_t +\Delta\Phi) C=(\kappa/\mu)\phi(C)\ge0$. Since $u\le C$ on the
parabolic boundary, Corollary~\ref{c:comparison} gives $u\le C$ on
$[0,T]\times\Omega$.

we next address the supersolution case.
If $\inf_{\partialpGammaT}u=-\infty$, the lower estimate is trivial.
Otherwise, the lower bound follows by comparing $u$ with the constant
$c:=\min\{0,\inf_{\partialpGammaT}u\}$, which is a subsolution because
$c\le0$.

If $\kappa=0$, every constant is both a subsolution and a supersolution, so
the same comparison argument gives the sharper estimates without the zero
barrier.
\end{proof}

\subsection{Extremal pointwise solutions}
We are now ready to prove the existence of extremal pointwise solutions to the \eqref{eq:C-D} with bounded initial datum and a forcing term which is continuous in time. The proof proceeds by first solving the \eqref{eq:C-D} on finite sets and then passing to the limit via monotone pointwise convergence along an exhaustion.
This is analogous to the construction of the heat kernel via an exhaustion argument as carried out for manifolds in \cite{Dod83} and for graphs in \cite{Woj08, Woj09, KL12, KLW21}.

The solutions we construct are defined for all times. To this end, we call a function $u\colon[0,\infty)\times X\to\R$ a \emph{global pointwise solution} of the \eqref{eq:C-D} if the restriction
of $u$ to $[0,T]\times X$ is a pointwise solution for every $T>0$ and a \emph{global bounded pointwise solution} if, in addition, $u$ is bounded on $[0,\infty)\times X$.

Our existence result allows for a non-homogeneous term in the \eqref{eq:C-D}. 
Specifically, we consider a forcing term
\[
f\in C\bigl([0,\infty);\ell^\infty(X)\bigr),
\]
that is, $t\mapsto f(t,\cdot)$ is continuous as an $\ell^\infty(X)$-valued map on $[0,\infty)$.
In particular, $f(\cdot,x)\in C([0,\infty))$ for every $x\in X$ and
\[
\sup_{t\in[0,T]}\|f(t,\cdot)\|_\infty<\infty
\qquad \text{for every }T>0.
\]
Writing
\[
f^{\pm}(t,x):=\max\{\pm f(t,x),0\}
\]
for the positive and negative parts of $f$, we associate, to every pair of constants $A\le0\le B$, the \emph{barrier functions}
\begin{equation}
\label{eq:barriers}
a(t):=A-\int_0^t\|f^-(s,\cdot)\|_\infty\dd s
\qquad \text{and} \qquad
b(t):=B+\int_0^t\|f^+(s,\cdot)\|_\infty\dd s.
\end{equation}
Note that $a$ is nonincreasing, $b$ is nondecreasing and
\[
a(t)\le A\le0\le B\le b(t)
\qquad \text{for every }t\ge0.
\]
We regard $a$ and $b$ also as functions on $[0,\infty)\times X$ which are constant in the spatial variable. In the homogeneous case $f=0$, the barriers reduce to the constants $a= A$ and $b= B$.

\begin{theorem}
\label{thm:extremal-bounded-pointwise-solutions}
Let \(G\) be a graph, let
\[
u_0\in\ell^\infty(X),
\qquad
f\in C\bigl([0,\infty);\ell^\infty(X)\bigr),
\]
and choose constants \(A,B\in\R\) such that
\[
A\le \min\{0,\inf_Xu_0\}
\qquad\text{and}\qquad
B\ge \max\{0,\sup_Xu_0\}.
\]
Let \(a\) and \(b\) be the associated barrier functions \eqref{eq:barriers}. Then, there exist global pointwise solutions
\[
\underline u^{A,f},\ \overline u^{B,f}
\colon[0,\infty)\times X\to\R
\]
of the \eqref{eq:C-D} with data \((u_0,f)\), bounded on
\([0,T]\times X\) for every \(T>0\), with the following properties:
\begin{enumerate}
\item[\textup{(i)}] \textup{(Barrier bounds)} For every \((t,x)\in[0,\infty)\times X\),
\begin{equation}
\label{eq:forced-extremal-bounds}
a(t)
\le
\underline u^{A,f}(t,x)
\le
\overline u^{B,f}(t,x)
\le
b(t).
\end{equation}
\item[\textup{(ii)}] \textup{(Minimality)} For every \(T>0\) and every bounded pointwise solution \(v\colon[0,T]\times X\to\R\) of the \eqref{eq:C-D} with the same data \((u_0,f)\) satisfying \(v\ge a\) on \([0,T]\times X\), one has
\[
\underline u^{A,f}\le v
\qquad
\textup{on }[0,T]\times X.
\]
\item[\textup{(iii)}] \textup{(Maximality)} Dually, for every \(T>0\) and every bounded pointwise solution \(v\colon[0,T]\times X\to\R\) of the \eqref{eq:C-D} with the same data \((u_0,f)\) satisfying \(v\le b\) on \([0,T]\times X\), one has
\[
v\le\overline u^{B,f}
\qquad
\textup{on }[0,T]\times X.
\]
\item[\textup{(iv)}] \textup{(Trapping)} Every bounded pointwise solution \(v\) of the \eqref{eq:C-D} on \([0,T]\times X\) with the same data \((u_0,f)\) such that \(a\le v\le b\) satisfies
\[
\underline u^{A,f}
\le v\le
\overline u^{B,f}
\qquad
\textup{on }[0,T]\times X.
\]
\item[\textup{(v)}] \textup{(Integrable forcing)} If, in addition,
\begin{equation}
\label{eq:integrable-forcing-infinity}
\int_0^\infty\|f(t,\cdot)\|_\infty\dd t<\infty,
\end{equation}
then \(\underline u^{A,f}\) and \(\overline u^{B,f}\) are global bounded pointwise solutions.
\item[\textup{(vi)}] \textup{(Independence of the exhaustion)}
The lower and upper extremal solutions obtained by the finite-domain
monotone constructions in the proof are independent of the chosen
exhaustion of \(X\) by finite connected subsets.
\end{enumerate}
\end{theorem}

\begin{remark}[The homogeneous case]
When \(f=0\), we use the abbreviations
\[
\underline u^{A}:=\underline u^{A,0}
\qquad \text{and} \qquad
\overline u^{B}:=\overline u^{B,0}.
\]
In this case, \(a= A\), \(b= B\), and
Theorem~\ref{thm:extremal-bounded-pointwise-solutions} yields global
bounded pointwise solutions of the homogeneous problem with initial datum
\(u_0\), satisfying
\[
A\le\underline u^{A}\le\overline u^{B}\le B
\qquad
\text{on }[0,\infty)\times X.
\]
\end{remark}

\begin{proof}
We divide the proof into several steps.

\medskip
\noindent\textbf{Step 1: Preliminary calculations.}
Since $f\in C([0,\infty);\ell^\infty(X))$ and the positive/negative-part maps are Lipschitz on $\ell^\infty(X)$, the maps $t\mapsto f^\pm(t,\cdot)$ are continuous with values in $\ell^\infty(X)$. Consequently, $a,b\in C^1([0,\infty))$ with
\[
a'(t)=-\|f^-(t,\cdot)\|_\infty
\qquad \text{and} \qquad
b'(t)=\|f^+(t,\cdot)\|_\infty.
\]
Moreover, $a(t)\le A\le0\le B\le b(t)$ and, since $\phi$ is monotone increasing with $\phi(0)=0$,
\[
\phi(a(t))\le0\le\phi(b(t))
\qquad
\text{for every }t\ge0.
\]
Viewing $a$ and $b$ as functions on $[0,\infty)\times X$ which are constant in the spatial variable, we therefore obtain, for every $x\in X$ and $t>0$,
\begin{equation}
\label{eq:lower-forced-barrier}
(\partial_t+\Delta\Phi)\,a(t,x)
=
-\|f^-(t,\cdot)\|_\infty
+
\frac{\kappa(x)}{\mu(x)}\phi(a(t))
\le
-\|f^-(t,\cdot)\|_\infty
\le f(t,x)
\end{equation}
and
\begin{equation}
\label{eq:upper-forced-barrier}
(\partial_t+\Delta\Phi)\,b(t,x)
=
\|f^+(t,\cdot)\|_\infty
+
\frac{\kappa(x)}{\mu(x)}\phi(b(t))
\ge
\|f^+(t,\cdot)\|_\infty
\ge f(t,x).
\end{equation}
Thus, $a$ and $b$ are, respectively, lower and upper comparison functions for the \eqref{eq:GPME}. Since the extension of a spatially constant function by its own value is constant on $X$, the same computation shows that, for every finite $\Omega\subseteq X$,
\[
\left(\partial_t+\Deltadir[\Omega][a(t)]\Phi\right)a\le f
\qquad\text{and}\qquad
\left(\partial_t+\Deltadir[\Omega][b(t)]\Phi\right)b\ge f
\qquad
\text{on }(0,\infty)\times\Omega.
\]
For every finite $\Omega\subseteq X$ and $t\ge0$, the spatially constant
exterior data $a(t)$ and $b(t)$ are $(\phi,\Omega)$-admissible by the local
summability of $w$.

\medskip
\noindent\textbf{Step 2: Existence and uniqueness on finite graphs.}
If \(X\) is finite, the \eqref{eq:C-D} is a finite-dimensional nonautonomous ODE whose vector field is jointly continuous in $(t,u)$. Since \(\phi\) is merely continuous and need not be locally Lipschitz, Peano's existence theorem
\cite[Theorem~1.1.2]{LakshmikanthamLeela1969}
yields a local \(C^1\)-solution.
Comparison with \(a\) and \(b\) on the finite domain \(\Omega=X\), Corollary~\ref{c:comparison}, using \eqref{eq:lower-forced-barrier} and \eqref{eq:upper-forced-barrier}, gives the bounds \(a(t)\le u\le b(t)\) on every compact subinterval of the interval of existence, together with uniqueness. In particular, the solution remains in a compact subset of \(\R^X\) on every finite time interval, so the finite-dimensional continuation theorem
\cite[Theorem~1.1.3]{LakshmikanthamLeela1969} extends it to
\([0,\infty)\). By uniqueness, this solution is both the minimal and the maximal solution in the classes stated in the theorem. Hence, in the rest of the proof, we assume that \(X\) is infinite.

\medskip
\noindent\textbf{Step 3: Basic properties of solutions on the exhaustion.}
We construct the lower extremal solution; the upper extremal solution is treated in Step~10. The construction is carried out directly on \([0,\infty)\): the barrier bounds below show that every finite-domain solution remains bounded on each finite time interval and therefore cannot blow up in finite time. Properties on \([0,\infty)\) are then obtained by applying the finite-interval comparison results on \([0,T]\) for arbitrary \(T>0\).

Let
\[
\Omega_1\subset\Omega_2\subset\cdots\subset X,
\qquad
\bigcup_{n=1}^{\infty}\Omega_n=X,
\]
be an exhaustion by finite connected subsets given by Lemma~\ref{lem:exhaustion}.  For each \(n\), we solve the finite problem with time-dependent exterior boundary value \(a(t)\).  More precisely, for a function
\[
u_n^{A,f} \colon [0,\infty)\times\Omega_n\to\R
\]
we denote by
\[
U_n^{A,f}(t,\cdot):=E_{\Omega_n}^{a(t)}u_n^{A,f}(t,\cdot)
\]
its extension to \(X\), namely,
\[
U_n^{A,f}(t,x)
=
\begin{cases}
 u_n^{A,f}(t,x)  & x\in\Omega_n\\
 a(t) & x\in X\setminus\Omega_n.
\end{cases}
\]
We consider
\begin{equation}
\label{eq:finite-forced-lower-problem}
\begin{cases}
\left(\partial_t +\Deltadir[n][a(t)]\Phi \right) u_n^{A,f}(t,x)=f(t,x)
& x\in\Omega_n,\ t\in(0,\infty)\\[2mm]
u_n^{A,f}(0,x)=u_0(x)
& x\in\Omega_n.
\end{cases}
\end{equation}
By definition, \(U_n^{A,f}(t,y)=a(t)\) for
\(y\in X\setminus\Omega_n\). Moreover, by the definition of
\(\Deltadir[n][a(t)]\), the global extension satisfies
\[
\left(\partial_t+\Delta\Phi\right)U_n^{A,f}(t,x)=f(t,x)
\qquad
\text{for }x\in\Omega_n,\ t\in(0,\infty).
\]
Equivalently, for \(x\in\Omega_n\),
\[
\partial_tu_n^{A,f}(t,x)
+
\frac1{\mu(x)}\sum_{y\in X}w(x,y)
\bigl(\phi(U_n^{A,f}(t,x))-\phi(U_n^{A,f}(t,y))\bigr)
+
\frac{\kappa(x)}{\mu(x)}\phi(U_n^{A,f}(t,x))
=f(t,x).
\]
For each fixed \(t\), the sum is absolutely convergent by the local summability of \(w\)
since $a(t)$ is constant in space. Moreover, the resulting finite-dimensional vector field is jointly continuous in \(t\) and in the unknown vector: the sum over \(y\in\Omega_n\) is finite, while the contribution of \(y\in X\setminus\Omega_n\) equals
\[
\bigl(\phi(u_n^{A,f}(t,x))-\phi(a(t))\bigr)
\sum_{y\in X\setminus\Omega_n}\frac{w(x,y)}{\mu(x)}
\]
with a finite coefficient and a continuous time dependence. As above, since \(\phi\) need not be locally Lipschitz, Peano's existence theorem gives a local \(C^1\)-solution. Uniqueness on every common interval of existence follows from the finite comparison principle, Corollary~\ref{c:comparison}. The barrier bounds established in Step~4 below rule out finite-time blow-up, so \(u_n^{A,f}\) is defined on all of \([0,\infty)\).

\medskip
\noindent\textbf{Step 4: Barrier bounds.}
We claim that
\begin{equation}
\label{eq:finite-forced-lower-bounds}
a(t)\le U_n^{A,f}(t,x)\le b(t)
\qquad
\text{for every }(t,x)\in[0,\infty)\times X.
\end{equation}
Outside of \(\Omega_n\), this is immediate since \(U_n^{A,f}=a\le b\) there.
On \(\Omega_n\), we apply the comparison principle on \([0,T]\), for
arbitrary \(T>0\) within the interval of existence of \(u_n^{A,f}\). As noted
above, the constant exterior data $a(t)$ and $b(t)$ are
$(\phi,\Omega_n)$-admissible. For the lower bound, \(a(0)=A\le u_0\) on
\(\Omega_n\), the exterior boundary values of \(a\) and of \(u_n^{A,f}\)
coincide, and, as calculated in \eqref{eq:lower-forced-barrier},
\[
\left(\partial_t+\Deltadir[n][a(t)]\Phi\right)a
\le f
=
\left(\partial_t+\Deltadir[n][a(t)]\Phi\right)u_n^{A,f}
\qquad \text{on }(0,T)\times\Omega_n.
\]
Corollary~\ref{c:comparison} gives \(a\le u_n^{A,f}\) on \([0,T]\times\Omega_n\). For the upper bound, \(u_0\le B=b(0)\) on \(\Omega_n\), the exterior boundary values satisfy \(a(t)\le b(t)\), and,
as calculated in \eqref{eq:upper-forced-barrier},
\[
\left(\partial_t+\Deltadir[n][a(t)]\Phi\right)u_n^{A,f}
=
f
\le
\left(\partial_t+\Deltadir[n][b(t)]\Phi\right)b
\qquad \text{on }(0,T)\times\Omega_n.
\]
Corollary~\ref{c:comparison} gives \(u_n^{A,f}\le b\) on \([0,T]\times\Omega_n\).

For every fixed \(T>0\), the functions \(a\) and \(b\) are bounded on \([0,T]\). Consequently, \(u_n^{A,f}\) remains in a compact subset of \(\R^{\Omega_n}\) on every compact time interval and the finite-dimensional continuation theorem rules out finite-time blow-up. Hence, \(u_n^{A,f}\) exists on \([0,\infty)\) and \eqref{eq:finite-forced-lower-bounds} holds.

\medskip
\noindent\textbf{Step 5: Monotonicity along the exhaustion.}
We prove that
\begin{equation}
\label{eq:forced-lower-exhaustion-monotonicity}
U_n^{A,f}\le U_{n+1}^{A,f}
\qquad
\text{on }[0,\infty)\times X.
\end{equation}
View \(U_n^{A,f}|_{\Omega_{n+1}}\) as a function on \(\Omega_{n+1}\): its extension from \(\Omega_{n+1}\) with exterior value \(a(t)\) is \(U_n^{A,f}\). On \(\Omega_n\), it satisfies the equation in \eqref{eq:finite-forced-lower-problem} exactly. Let
\[
x\in\Omega_{n+1}\setminus\Omega_n.
\]
Then,
\[
U_n^{A,f}(t,x)=a(t),
\qquad
\partial_tU_n^{A,f}(t,x)=a'(t)=-\|f^-(t,\cdot)\|_\infty.
\]
Moreover,
\[
\Deltadir[n+1][a(t)]\Phi\bigl(U_n^{A,f}|_{\Omega_{n+1}}\bigr)(t,x)
=
\frac1{\mu(x)}
\sum_{y\in X}w(x,y)
\bigl(\phi(a(t))-\phi(U_n^{A,f}(t,y))\bigr)
+
\frac{\kappa(x)}{\mu(x)}\phi(a(t)).
\]
By \eqref{eq:finite-forced-lower-bounds}, \(U_n^{A,f}(t,y)\ge a(t)\) for every \(y\in X\). Since \(\phi\) is monotone increasing, every edge contribution in the preceding display is nonpositive, and the killing contribution is nonpositive because \(\phi(a(t))\le0\). It follows that
\[
\left(\partial_t+\Deltadir[n+1][a(t)]\Phi\right)\bigl(U_n^{A,f}|_{\Omega_{n+1}}\bigr)
\le
-\|f^-(t,\cdot)\|_\infty
\le
f
=
\left(\partial_t+\Deltadir[n+1][a(t)]\Phi\right)u_{n+1}^{A,f}
\]
on \((0,\infty)\times\Omega_{n+1}\).

At time \(t=0\), one has
\[
U_n^{A,f}(0,x)
=
\begin{cases}
u_0(x) & x\in\Omega_n\\
A & x\in\Omega_{n+1}\setminus\Omega_n
\end{cases}
\]
and, hence, \(U_n^{A,f}(0,\cdot)\le u_0=u_{n+1}^{A,f}(0,\cdot)\) on
\(\Omega_{n+1}\) because \(A\le u_0\). The two functions have the same
exterior value \(a(t)\) outside \(\Omega_{n+1}\), which is
\((\phi,\Omega_{n+1})\)-admissible for every \(t\ge0\).
Corollary~\ref{c:comparison}, applied on \([0,T]\) for arbitrary \(T>0\), gives
\[
U_n^{A,f}\le u_{n+1}^{A,f}
\qquad
\text{on }[0,\infty)\times\Omega_{n+1}.
\]
After extension by \(a(t)\), this proves \eqref{eq:forced-lower-exhaustion-monotonicity} on all of \(X\).

Consequently, the sequence \((U_n^{A,f})_n\) is pointwise increasing and
\[
a(t)\le U_1^{A,f}(t,x)\le U_2^{A,f}(t,x)\le\cdots\le b(t).
\]
We may therefore define
\begin{equation}
\label{eq:def-forced-lower-extremal-limit}
\underline u^{A,f}(t,x):=\lim_{n\to\infty}U_n^{A,f}(t,x),
\qquad
(t,x)\in[0,\infty)\times X.
\end{equation}
It follows immediately that
\[
a(t)\le \underline u^{A,f}(t,x)\le b(t)
\qquad
\text{for every }(t,x)\in[0,\infty)\times X.
\]

\medskip
\noindent\textbf{Step 6: Time regularity.}
Fix \(x\in X\) and \(T>0\).  For all sufficiently large \(n\), one has \(x\in\Omega_n\).  Set
\[
M_T:=\sup_{t\in[0,T]}\|f(t,\cdot)\|_\infty<\infty
\qquad \text{and} \qquad
K_T:=\max_{r\in[a(T),b(T)]}|\phi(r)|<\infty.
\]
Since \(a\) is nonincreasing and \(b\) is nondecreasing, \eqref{eq:finite-forced-lower-bounds} gives
\[
a(T)\le U_n^{A,f}(t,y)\le b(T)
\qquad
\text{for every }t\in[0,T],\ y\in X \text{ and } n.
\]
Thus, for all sufficiently large \(n\),
\[
\begin{aligned}
|\partial_t U_n^{A,f}(t,x)|
&=|f(t,x)-\Delta\Phi U_n^{A,f}(t,x)|\\
&\le
M_T
+
\frac1{\mu(x)}\sum_{y\in X}w(x,y)
\left|\phi(U_n^{A,f}(t,x))-\phi(U_n^{A,f}(t,y))\right|
+
\frac{\kappa(x)}{\mu(x)}|\phi(U_n^{A,f}(t,x))|\\
&\le
M_T
+
\frac{2K_T}{\mu(x)}\sum_{y\in X}w(x,y)
+
\frac{\kappa(x)}{\mu(x)}K_T
=:
C_{x,T}.
\end{aligned}
\]
The constant \(C_{x,T}\) is independent of \(n\) and \(t\).  Hence,
\[
|U_n^{A,f}(t,x)-U_n^{A,f}(s,x)|\le C_{x,T}|t-s|
\qquad
\text{for all }s,t\in[0,T].
\]
Passing to the pointwise limit gives
\[
|\underline u^{A,f}(t,x)-\underline u^{A,f}(s,x)|\le C_{x,T}|t-s|.
\]
Thus, \(t\mapsto \underline u^{A,f}(t,x)\) is Lipschitz, hence continuous, on \([0,T]\).

For fixed \(x\), the sequence \(U_n^{A,f}(\cdot,x)\) is increasing in \(n\), consists of continuous functions on \([0,T]\), and converges pointwise to the continuous function \(\underline u^{A,f}(\cdot,x)\).  By Dini's theorem,
\[
U_n^{A,f}(\cdot,x)\to \underline u^{A,f}(\cdot,x)
\qquad
\text{uniformly on }[0,T].
\]
Since \(\phi\) is uniformly continuous on \([a(T),b(T)]\), it follows that
\begin{equation}
\label{eq:uniform-forced-Phi-convergence}
\phi(U_n^{A,f}(\cdot,x))
\to
\phi(\underline u^{A,f}(\cdot,x))
\qquad
\text{uniformly on }[0,T]
\end{equation}
for every fixed \(x\in X\) and every \(T>0\).

\medskip
\noindent\textbf{Step 7: Passage to the limit.}
Fix \(x\in X\) and \(T>0\). Since
\[
a(T)\le\underline u^{A,f}(t,y)\le b(T)
\qquad
\text{for every }(t,y)\in[0,T]\times X,
\]
the function \(\Phi\underline u^{A,f}(t,\cdot)\) is bounded for every \(t\in[0,T]\) and the local summability of \(w\) implies
\[
\underline u^{A,f}(t,\cdot)\in \mathcal{F}_\Phi
\qquad
\text{for every }t\in[0,T].
\]
For all sufficiently large \(n\), one has \(x\in\Omega_n\) and
\begin{equation}
\label{eq:finite-forced-integral-identity}
U_n^{A,f}(t,x)-u_0(x)
=
\int_0^t\bigl(f(s,x)-\Delta\Phi U_n^{A,f}(s,x)\bigr)\dd s.
\end{equation}
We claim that
\begin{equation}
\label{eq:forced-laplacian-uniform-convergence}
\Delta\Phi U_n^{A,f}(\cdot,x)
\to
\Delta\Phi\underline u^{A,f}(\cdot,x)
\qquad
\text{uniformly on }[0,T].
\end{equation}
Indeed,
\[
\begin{aligned}
\sup_{t\in[0,T]}
\left|
\Delta\Phi U_n^{A,f}(t,x)-\Delta\Phi\underline u^{A,f}(t,x)
\right| \le
&\left(\frac{1}{\mu(x)}\sum_{y\in X}  w(x,y)+\frac{\kappa(x)}{\mu(x)} \right)
\sup_{t\in[0,T]}
\left|
\phi(U_n^{A,f}(t,x))-\phi(\underline u^{A,f}(t,x))
\right|\\
&\qquad+
\frac1{\mu(x)}\sum_{y\in X}w(x,y)
\sup_{t\in[0,T]}
\left|
\phi(U_n^{A,f}(t,y))-\phi(\underline u^{A,f}(t,y))
\right|.
\end{aligned}
\]
For every fixed \(y\), the corresponding supremum tends to zero by \eqref{eq:uniform-forced-Phi-convergence}.  Moreover, it is bounded by \(2K_T\), and
\[
\sum_{y\in X}w(x,y)<\infty.
\]
Thus, dominated convergence gives \eqref{eq:forced-laplacian-uniform-convergence}.

For every fixed \(n\) with \(x\in\Omega_n\), the map \(t\mapsto\Delta\Phi U_n^{A,f}(t,x)\) is continuous: the contribution of \(y\in\Omega_n\) is a finite sum of continuous functions, while the contribution of \(y\in X\setminus\Omega_n\) equals
\[
\bigl(\phi(u_n^{A,f}(t,x))-\phi(a(t))\bigr)
\sum_{y\in X\setminus\Omega_n}\frac{w(x,y)}{\mu(x)}
\]
which is continuous in \(t\). By \eqref{eq:forced-laplacian-uniform-convergence}, the map \(t\mapsto\Delta\Phi\underline u^{A,f}(t,x)\) is continuous on \([0,T]\) as well.

Passing to the limit in \eqref{eq:finite-forced-integral-identity}, we obtain
\[
\underline u^{A,f}(t,x)-u_0(x)
=
\int_0^t\bigl(f(s,x)-\Delta\Phi\underline u^{A,f}(s,x)\bigr)\dd s.
\]
Since \(s\mapsto f(s,x)\) and \(s\mapsto\Delta\Phi\underline u^{A,f}(s,x)\) are continuous, the fundamental theorem of calculus yields
\[
\underline u^{A,f}(\cdot,x)\in C^1((0,T))
\qquad \text{and} \qquad
(\partial_t +\Delta\Phi)\underline u^{A,f}(t,x)=f(t,x)
\]
for every \(t\in(0,T)\) and \(x\in X\).  Since \(T>0\) was arbitrary, \(\underline u^{A,f}\) is a global pointwise solution of the \eqref{eq:C-D}, bounded on every compact time interval by \eqref{eq:finite-forced-lower-bounds}.

\medskip
\noindent\textbf{Step 8: Minimality.}
Let \(T>0\) and let \(v\colon[0,T]\times X\to\R\) be a bounded pointwise solution of the \eqref{eq:C-D} with the same data \((u_0,f)\) satisfying
\[
v(t,x)\ge a(t)
\qquad
\text{for every }(t,x)\in[0,T]\times X.
\]
Fix \(n\).  On \(\Omega_n\), the function \(u_n^{A,f}\) solves the finite problem with exterior boundary value \(a(t)\), while the restriction of \(v\) to \(\Omega_n\) solves the same equation with exterior boundary values
\[
v(t,y),
\qquad y\in X\setminus\Omega_n.
\]
Since \(v\) is bounded on \([0,T]\times X\), \(\phi\) is bounded on bounded intervals and \(w\) is locally summable, these exterior data are \((\phi,\Omega_n)\)-admissible.

The initial data and the forcing terms agree, while the exterior boundary data are ordered:
\[
a(t)\le v(t,y)
\qquad
\text{for every }t\in[0,T] \text{ and } y\in X\setminus\Omega_n.
\]
Corollary~\ref{c:comparison} gives
\[
u_n^{A,f}\le v
\qquad
\text{on }[0,T]\times\Omega_n.
\]
Letting \(n\to\infty\), we obtain \(\underline u^{A,f}\le v\) on \([0,T]\times X\). This proves the minimality property.

\medskip
\noindent\textbf{Step 9: Independence of the exhaustion.}
Let \((\Omega_n)\) and \((\Lambda_k)\) be two exhaustions by finite and connected sets and let
\[
\underline u^{A,f,\Omega},
\qquad
\underline u^{A,f,\Lambda}
\]
be the corresponding limits.  Both are global pointwise solutions bounded on every compact
time interval, both lie above \(a\), and both satisfy the minimality property of Step~8 on every \([0,T]\times X\).  Therefore,
\[
\underline u^{A,f,\Omega}\le \underline u^{A,f,\Lambda}
\qquad \text{and} \qquad
\underline u^{A,f,\Lambda}\le \underline u^{A,f,\Omega},
\]
that is, the two limits coincide. Hence, the lower extremal solution is independent of the chosen exhaustion.

\medskip
\noindent\textbf{Step 10: The upper extremal solution.}
For each \(n\), consider the finite problem with time-dependent exterior boundary value \(b(t)\):
\[
\begin{cases}
\left(\partial_t +\Deltadir[n][b(t)]\Phi \right) u_n^{B,f}(t,x)=f(t,x)
& x\in\Omega_n,\ t\in(0,\infty)\\[2mm]
u_n^{B,f}(0,x)=u_0(x)
& x\in\Omega_n,
\end{cases}
\]
and set \(U_n^{B,f}(t,\cdot):=E_{\Omega_n}^{b(t)}u_n^{B,f}(t,\cdot)\).  As before, Peano's theorem combined with comparison against \(a\) and \(b\) yields a solution defined on all of \([0,\infty)\) which satisfies
\[
a(t)\le U_n^{B,f}(t,x)\le b(t)
\qquad
\text{for every }(t,x)\in[0,\infty)\times X.
\]

We claim that the sequence is now pointwise decreasing:
\[
U_{n+1}^{B,f}\le U_n^{B,f}
\qquad
\text{on }[0,\infty)\times X.
\]
View \(U_n^{B,f}|_{\Omega_{n+1}}\) as a function on \(\Omega_{n+1}\).  On \(\Omega_n\), it satisfies the equation exactly.  If \(x\in\Omega_{n+1}\setminus\Omega_n\), then
\[
U_n^{B,f}(t,x)=b(t),
\qquad
\partial_tU_n^{B,f}(t,x)=b'(t)=\|f^+(t,\cdot)\|_\infty,
\]
and
\[
\Deltadir[n+1][b(t)]\Phi\bigl(U_n^{B,f}|_{\Omega_{n+1}}\bigr)(t,x)
=
\frac1{\mu(x)}
\sum_{y\in X}w(x,y)
\bigl(\phi(b(t))-\phi(U_n^{B,f}(t,y))\bigr)
+
\frac{\kappa(x)}{\mu(x)}\phi(b(t))
\ge0
\]
because \(U_n^{B,f}(t,y)\le b(t)\) for every \(y\in X\), \(\phi\) is monotone increasing and \(\phi(b(t))\ge0\).  Therefore,
\[
\left(\partial_t+\Deltadir[n+1][b(t)]\Phi\right)\bigl(U_n^{B,f}|_{\Omega_{n+1}}\bigr)
\ge
\|f^+(t,\cdot)\|_\infty
\ge
f
=
\left(\partial_t+\Deltadir[n+1][b(t)]\Phi\right)u_{n+1}^{B,f}
\]
on \((0,\infty)\times\Omega_{n+1}\).  At \(t=0\), one has \(u_{n+1}^{B,f}(0,\cdot)=u_0\le U_n^{B,f}(0,\cdot)\) on \(\Omega_{n+1}\) because \(u_0\le B=b(0)\), and the two functions have the same exterior value \(b(t)\) outside \(\Omega_{n+1}\).  Corollary~\ref{c:comparison}, applied on \([0,T]\) for arbitrary \(T>0\), gives the claim.

Hence, the pointwise limit
\[
\overline u^{B,f}(t,x):=\lim_{n\to\infty}U_n^{B,f}(t,x)
\]
exists and satisfies \(a(t)\le\overline u^{B,f}(t,x)\le b(t)\) on \([0,\infty)\times X\).  The estimate of Step~6 applies verbatim to \(U_n^{B,f}\), so \(\overline u^{B,f}(\cdot,x)\) is Lipschitz on \([0,T]\) for every \(x\in X\) and \(T>0\); Dini's theorem applies to the decreasing sequence \(U_n^{B,f}(\cdot,x)\), and the argument of Step~7 shows that \(\overline u^{B,f}\) is a global pointwise solution of the \eqref{eq:C-D}, bounded on every compact time interval.

For the maximality, let \(T>0\) and let \(v\) be a bounded pointwise solution
on \([0,T]\times X\) with the same data \((u_0,f)\) satisfying \(v\le b\).
By the boundedness of \(v\), the exterior datum
\(v(t,\cdot)|_{X\setminus\Omega_n}\) is
\((\phi,\Omega_n)\)-admissible for every \(t\in[0,T]\); it is also bounded
above by \(b(t)\). Hence, Corollary~\ref{c:comparison} gives
\[
v\le u_n^{B,f}
\qquad
\text{on }[0,T]\times\Omega_n.
\]
Letting \(n\to\infty\) yields \(v\le\overline u^{B,f}\) on \([0,T]\times X\).  As in Step~9, maximality also shows that \(\overline u^{B,f}\) is independent of the chosen exhaustion.

Finally, for every \(n\), the lower and upper finite-domain solutions share the same forcing and initial datum, while their exterior boundary values satisfy \(a(t)\le b(t)\).  Hence, Corollary~\ref{c:comparison} gives
\[
U_n^{A,f}\le U_n^{B,f}
\qquad
\text{on }[0,\infty)\times X
\]
and, passing to the limits, yields \(\underline u^{A,f}\le\overline u^{B,f}\).  Together with Step~4 and the bounds above, this proves the barrier bounds \eqref{eq:forced-extremal-bounds}. The trapping property follows immediately by combining minimality and maximality.

Assume now \eqref{eq:integrable-forcing-infinity}, i.e., the integrability
of the forcing term.  Then, for every \(t\ge0\),
\[
0\le b(t)-B\le\int_0^\infty\|f(s,\cdot)\|_\infty\dd s
\qquad \text{and} \qquad
0\le A-a(t)\le\int_0^\infty\|f(s,\cdot)\|_\infty\dd s,
\]
so \(a\) and \(b\) are bounded on \([0,\infty)\).  By \eqref{eq:forced-extremal-bounds}, the two extremal solutions are bounded on \([0,\infty)\times X\), and hence are global bounded pointwise solutions.
\end{proof}

The next remark, which answers a natural question, shows in which sense the barriers of Theorem~\ref{thm:extremal-bounded-pointwise-solutions} are optimal: they can always be anchored at the natural levels of the initial datum on one side, but not on both sides simultaneously, unless the constants are chosen canonically.

\begin{remark}[Sharper barriers and their optimality]
\label{rem:sharpened-barriers}
Set
\[
A_0:=\min\{0,\inf_Xu_0\}
\qquad \text{and} \qquad
B_0:=\max\{0,\sup_Xu_0\},
\]
the largest permitted lower-barrier level and the smallest permitted
upper-barrier level, respectively, in
Theorem~\ref{thm:extremal-bounded-pointwise-solutions}. Let
\[
a_0(t):=A_0-\int_0^t\|f^-(s,\cdot)\|_\infty\dd s
\qquad \text{and} \qquad
b_0(t):=B_0+\int_0^t\|f^+(s,\cdot)\|_\infty\dd s
\]
be the associated barrier functions, so that \(a\le a_0\le0\le b_0\le b\) for general \(A\le A_0\) and \(B\ge B_0\).
\begin{enumerate}[(1)]
\item \textup{(One-sided sharpening)} For every \(A\le A_0\) and \(B\ge B_0\), the extremal solutions satisfy the sharper bounds
\[
a(t)\le\underline u^{A,f}(t,x)\le b_0(t)
\qquad \text{and} \qquad
a_0(t)\le\overline u^{B,f}(t,x)\le b(t)
\]
on \([0,\infty)\times X\). Indeed, in Step~4 of the proof, the comparison with the upper barrier only uses \(u_0\le b(0)\) and the ordering between the exterior values; both remain valid with \(b_0\) in place of \(b\), since \(u_0\le B_0=b_0(0)\) and \(a\le0\le b_0\). Hence, \(u_n^{A,f}\le b_0\) for every \(n\) and, in the limit, \(\underline u^{A,f}\le b_0\). The bound for \(\overline u^{B,f}\) is dual.

\item \textup{(Failure of the two-sided sharpening)} In general, the dual bounds \(a_0\le\underline u^{A,f}\) and \(\overline u^{B,f}\le b_0\) are false when \(A<A_0\) or \(B>B_0\), already in the homogeneous case. Consider \(\phi=\operatorname{id}\), \(\kappa=0\), \(f=0\) and \(u_0=0\), so that \(A_0=B_0=0\), and take the exterior level \(A=-1\). Since \(\kappa=0\), the function \(V_n:=U_n^{A}+\mathbf1\) is the zero extension of \(v_n:=u_n^{A}+1\), and it solves the finite heat equation with zero exterior value and initial datum \(\mathbf1_{\Omega_n}\), that is,
\[
v_n(t)=e^{-t\Deltadir[n]}\mathbf1_{\Omega_n}.
\]
As \(n\to\infty\), these finite Dirichlet semigroups increase to the minimal heat semigroup \((P_t)_{t\ge0}\) of the graph, see \cite{Woj08} and \cite[Chapter~7]{KLW21}, whence
\[
\underline u^{A}(t,\cdot)=P_t\mathbf1-\mathbf1.
\]
If \(G\) is stochastically incomplete, that is, if \ref{SC} fails and \(P_t\mathbf1<\mathbf1\) somewhere -- see \cite{Woj08,Woj09,KL12} and \cite[Chapter~7]{KLW21} for large classes of such graphs -- then
\[
\underline u^{A}(t,x)<0=A_0
\]
for some \((t,x)\). Dually, for $B=1,$ \(\overline u^{B}=\mathbf1-P_t\mathbf1>0=B_0\) somewhere. Thus, the exterior condition ``at infinity'' imposed along the exhaustion can push the extremal solutions strictly outside the interval \([A_0,B_0]\) determined by the initial datum.
\item \textup{(Canonical box)} For the canonical choice \(A=A_0\) and \(B=B_0\), Theorem~\ref{thm:extremal-bounded-pointwise-solutions} yields
\[
a_0(t)
\le
\underline u^{A_0,f}(t,x)
\le
\overline u^{B_0,f}(t,x)
\le
b_0(t)
\qquad
\text{on }[0,\infty)\times X,
\]
which, for \(f\equiv0\), reads \(A_0\le\underline u^{A_0}\le\overline u^{B_0}\le B_0\).
\end{enumerate}
\end{remark}

We now derive some consequences of the preceding results. The first gives the existence
of a minimal positive solution whenever the initial datum and the forcing term are positive.

\begin{corollary}
\label{cor:canonical-minimal-nonnegative-solution}
Let $u_0\in\ell^\infty(X)$ with $u_0\ge0$ and let $f\in C([0,\infty);\ell^\infty(X))$ with $f\ge0$.
Then, the lower-extremal construction of
Theorem~\ref{thm:extremal-bounded-pointwise-solutions} with \(A=0\), and
hence with zero exterior datum \(a\equiv0\), produces a global positive
pointwise solution
$\underline u^{0,f}\colon[0,\infty)\times X\to\R$ of the \eqref{eq:C-D} with data $(u_0,f)$, bounded on every compact time interval, such that
\[
0\le \underline u^{0,f}(t,x)\le \|u_0\|_\infty+\int_0^t\|f(s,\cdot)\|_\infty\dd s
\qquad
\text{on }[0,\infty)\times X.
\]
Moreover, \(\underline u^{0,f}\) is minimal among all bounded nonnegative
pointwise solutions with the same data \((u_0,f)\) on every
\([0,T]\times X\), \(T>0\). In particular, for $f=0$,
\[
0\le \underline u^0(t,x)\le \|u_0\|_\infty
\qquad
\text{on }[0,\infty)\times X.
\]
\end{corollary}

\begin{proof}
Since $u_0\ge0$ and $f\ge0$, one has $\min\{0,\inf_Xu_0\}=0$, $f^-=0$ and $\|f^+(s,\cdot)\|_\infty=\|f(s,\cdot)\|_\infty$. Apply Theorem~\ref{thm:extremal-bounded-pointwise-solutions} with
$A=0$ and $B=\|u_0\|_\infty$: the associated barrier functions are
\[
a=0
\qquad \text{and} \qquad
b(t)=\|u_0\|_\infty+\int_0^t\|f(s,\cdot)\|_\infty\dd s
\]
and the class of bounded pointwise solutions lying above $a$ is precisely the class of nonnegative ones.
\end{proof}

Next, we single out the extremal solutions associated with the canonical choice of the constants.

\begin{corollary}
\label{cor:canonical-extremal-box}
Let $u_0\in\ell^\infty(X)$, let $f\in C([0,\infty);\ell^\infty(X))$ and set
\[
A_0:=\min\{0,\inf_Xu_0\}
\qquad \text{and} \qquad
B_0:=\max\{0,\sup_Xu_0\}
\]
with associated barrier functions $a_0$ and $b_0$ as in Remark~\ref{rem:sharpened-barriers}. Then, the lower and upper extremal constructions yield global pointwise solutions
$\underline u^{A_0,f}, \overline u^{B_0,f}\colon[0,\infty)\times X\to\R$
of the \eqref{eq:C-D} with data $(u_0,f)$,
bounded on every compact time interval. These solutions may
coincide and satisfy
\[
a_0(t)\le \underline u^{A_0,f}(t,x)\le \overline u^{B_0,f}(t,x)\le b_0(t)
\qquad
\text{on }[0,\infty)\times X.
\]
Moreover, if \(T>0\) and \(v\) is any bounded pointwise solution on \([0,T]\times X\) with the same data \((u_0,f)\) satisfying
\[
a_0\le v\le b_0,
\]
then
\[
\underline u^{A_0,f}\le v\le \overline u^{B_0,f}
\qquad
\text{on }[0,T]\times X.
\]
In particular, for $f=0$, the extremal solutions $\underline u^{A_0}$ and $\overline u^{B_0}$ are global bounded pointwise solutions and satisfy
\[
A_0\le \underline u^{A_0}\le \overline u^{B_0}\le B_0
\qquad
\text{on }[0,\infty)\times X.
\]
\end{corollary}
\begin{proof}
    This is a direct consequence of Theorem~\ref{thm:extremal-bounded-pointwise-solutions}
    with $A=A_0$ and $B=B_0$.
\end{proof}

Having explored pointwise solutions, 
we now give a criterion for when a bounded pointwise solution is a classical
solution.

\begin{lemma}\label{l:classical-solution}
Let $G$ satisfy $\BD$, let $f\in C([0,T];\ell^\infty(X))$ and let $u$ be a bounded pointwise solution of the \eqref{eq:C-D} with forcing term $f$ on $[0,T]\times X$. Then, $u$ is an $\ell^\infty$-classical solution of the \eqref{eq:C-D}.
If, in addition, $G$ satisfies $\FM$, then $u$ is an
$\ell^p$-classical solution of the \eqref{eq:C-D} for every $p\in[1,\infty]$.
\end{lemma}

\begin{proof}
Let $R:=\sup_{(t,x)\in[0,T]\times X}|u(t,x)|$ and
$K:=\sup_{|s|\le R}|\phi(s)|$. Since we assume $\BD$, 
we let $D=\sup_{x \in X}\Deg(x)$. As a consequence of $\BD$,
$\Delta$ is bounded
on $\ell^\infty(X)$ with operator norm bounded by $2D$, see
\cite[Theorem~2.15]{KLW21}. Therefore,
\[
  \|\Delta\Phi u(t,\cdot)\|_\infty
  \le
  2D\|\Phi u(t,\cdot)\|_\infty
  \le
  2DK.
\]
The \eqref{eq:C-D} gives
\[
u(t,x)-u(s,x)
=
\int_s^t\bigl(f(r,x)-\Delta\Phi u(r,x)\bigr)\dd r
\]
for every $x\in X$. Since $\sup_{r\in[0,T]}\|f(r,\cdot)\|_\infty<\infty$, the map $t\mapsto u(t,\cdot)$ is Lipschitz in
$\ell^\infty(X)$. Since $\phi$ is uniformly continuous on $[-R,R]$ and
$\Delta \colon \ell^\infty(X)\to\ell^\infty(X)$ is bounded under $\BD$, the map
$t\mapsto\Delta\Phi u(t,\cdot)$ is continuous in $\ell^\infty(X)$ and so is
$t\mapsto f(t,\cdot)-\Delta\Phi u(t,\cdot)$.
The preceding integral identity therefore holds in $\ell^\infty(X)$, and the
fundamental theorem of calculus in Banach spaces gives
\[
u\in C^1((0,T);\ell^\infty(X)),
\qquad
\partial_t u(t)=f(t)-\Delta\Phi u(t).
\]
Thus, $u$ is an $\ell^\infty(X)$-classical solution.

If $\FM$ holds, then $u(t,\cdot),\Delta\Phi u(t,\cdot),f(t,\cdot)\in\ell^p(X,\mu)$ for
every finite $p$, and
\[
\|h\|_p\le \mu(X)^{1/p}\|h\|_\infty
\]
for every $h \in \ell^p(X,\mu)$.
In particular, $f\in C([0,T];\ell^p(X,\mu))$, and the same continuity and differentiability conclusions
as above hold in
$\ell^p(X,\mu)$ for $1\le p<\infty$. This proves the claim for every
$p\in[1,\infty]$.
\end{proof}

As a consequence, in the bounded degree regime and under finite measure, we always
get the existence of $\ell^p$-classical solutions, compare with Corollary~4.2 in \cite{bianchi2022generalized}
which gives the conclusion for $p=1$ under bounded degree and continuity of $\Phi$. 

\begin{corollary}\label{c:classical-solution}
Let $G$ satisfy $\BD$ and $\FM$. Then, for every $u_0\in\ell^\infty(X)$,
$f\in C([0,\infty);\ell^\infty(X))$,
$T>0$ and $p\in[1,\infty]$, there exists an
$\ell^p$-classical solution on $[0,T]$ of the \eqref{eq:C-D} with data $(u_0,f)$.
\end{corollary}
\begin{proof}
Apply Theorem~\ref{thm:extremal-bounded-pointwise-solutions}, restrict the resulting global pointwise solution to $[0,T]\times X$, where it is bounded, and then apply
Lemma~\ref{l:classical-solution}.
\end{proof}

Next, we establish the existence and uniqueness of global $\ell^\infty$-classical solutions
under bounded degree and a locally Lipschitz assumption on $\phi$.

\begin{theorem}\label{t:globalwp}
Let $G$ satisfy $\BD$ and assume that $\phi$ is locally Lipschitz. Then,
for every $u_0\in\ell^\infty(X)$ and every $f\in C([0,\infty);\ell^\infty(X))$, there exists a unique global
$\ell^\infty(X)$-classical solution $u$ of the \eqref{eq:C-D} with data $(u_0,f)$, i.e.,
\[
u\in C([0,\infty);\ell^\infty(X))
\cap C^1((0,\infty);\ell^\infty(X)),
\qquad
u(t)\in \dom(\mathcal L^{(\infty)})
\quad \text{for } t>0
\]
and
\[
( \partial_t +\Delta \Phi)u(t)=f(t)
\quad\text{in }\ell^\infty(X)\quad \text{for } t>0,
\qquad \text{with } u(0)=u_0.
\]
Moreover, if $A_0$, $B_0$, $a_0$ and $b_0$ are as in Remark~\ref{rem:sharpened-barriers}, then
\[
a_0(t)\le u(t)\le b_0(t)
\qquad
\text{for all } t\ge0.
\]
In particular,
\[
\|u(t,\cdot)\|_\infty
\le
\|u_0\|_\infty+\int_0^t\|f(s,\cdot)\|_\infty\dd s
\qquad \text{for all } t\ge0
\]
and, for $f=0$,
\[
A_0\le u(t)\le B_0
\qquad \text{and} \qquad
\|u(t,\cdot)\|_\infty\le\|u_0\|_\infty
\qquad \text{for all } t\ge0.
\]
\end{theorem}
\begin{proof}
Set $F(t,u):=f(t)-\Delta\Phi u$ and
$D:=\sup_{x\in X}\Deg(x)<\infty.$
Under \BD, the operator $\Delta\colon\ell^\infty(X)\to\ell^\infty(X)$
is bounded with norm at most $2D$, see \cite[Theorem~2.15]{KLW21}. 
Fix $R>0$, and let $L_R$ be a Lipschitz
constant for $\phi$ on $[-R,R]$. If
$\|u\|_\infty,\|v\|_\infty\le R$, then, for every $t\ge0$,
\[
\|F(t,u)-F(t,v)\|_\infty
=
\|\Delta\Phi u-\Delta\Phi v\|_\infty
\le 2D\,\|\Phi u-\Phi v\|_\infty
\le 2DL_R\|u-v\|_\infty.
\]
Hence, $F(t,\cdot) \colon \ell^\infty(X)\to\ell^\infty(X)$ is locally Lipschitz, uniformly in $t\ge0$.

By Corollary~\ref{cor:canonical-extremal-box}, there exists a global
pointwise solution
\[
u:=\underline u^{A_0,f}\colon[0,\infty)\times X\to\R,
\]
bounded on every compact time interval, such that
\[
a_0(t)\le u(t,\cdot)\le b_0(t)
\qquad\text{for every }t\ge0.
\]
For every $T>0$, Lemma~\ref{l:classical-solution}, applied to the restriction
of $u$ to $[0,T]\times X$, shows that this restriction is an
$\ell^\infty(X)$-classical solution. Since $T>0$ is arbitrary, $u$ is a
global $\ell^\infty(X)$-classical solution with the regularity and operator
domain properties stated in the theorem.

If $v$ is another global $\ell^\infty(X)$-classical solution with the same
data, then, on every compact time interval, continuity at $t=0$ and the
differential equation for $t>0$ give the integral identities associated with
the Banach-space ODE
\[
w'(t)=F(t,w(t)),
\qquad
w(0)=u_0.
\]
The ranges of $u$ and $v$ on such an interval lie in a common bounded ball of
$\ell^\infty(X)$, where $F(t,\cdot)$ is Lipschitz uniformly in $t$; moreover,
the forcing term cancels in the difference of the two integral identities.
Gronwall's inequality therefore gives $u=v$ on every
finite interval and hence on $[0,\infty)$.

Finally, since $\max\{-A_0,B_0\}=\|u_0\|_\infty$ and
$\|f(s,\cdot)\|_\infty=\max\{\|f^+(s,\cdot)\|_\infty,\|f^-(s,\cdot)\|_\infty\}$,
the barrier bounds imply
\[
\|u(t,\cdot)\|_\infty
\le
\max\{-a_0(t),\,b_0(t)\}
\le
\|u_0\|_\infty+\int_0^t\|f(s,\cdot)\|_\infty\dd s
\qquad\text{for every }t\ge0.
\]
For $f=0$, one has $a_0= A_0$ and $b_0= B_0.$ This completes the proof.
\end{proof}

\section{Properties of mild solutions: existence, extinction, and smoothing}\label{sec:extinction}

In this section, we study mild solutions of the \eqref{eq:C-D} introduced in Definition~\ref{def:weak_solution}. We first establish existence and uniqueness for a general nonlinearity and a forcing term. We then specialize to the porous medium nonlinearity and to the homogeneous problem, that is, $\phi_m(s)=s |s|^{m-1}$ for $m>0$ and $f= 0$: assuming the Sobolev inequality \SP[\nu], we prove finite-time extinction in the fast diffusion range $0<m<2/\nu$ and $\ell^1$-$\ell^q$ smoothing estimates in the range $m>2/\nu$. This discrete Sobolev dichotomy is parallel to the Euclidean critical-line theory for porous-medium and fast-diffusion equations; see \cite[Chapters~2, 3 and~5]{vazquez2006smoothing}. Related Riemannian results obtained through Laplacian cut-offs, including \(L^1\)-contractivity and extinction-time estimates, appear in \cite[Section~4]{bianchi2018laplacian}. We close the section with examples of graphs which satisfy \SP[\nu] and for which the resulting critical exponent is the same as in the Euclidean setting. We stress that in this section no restriction is placed on the killing term, i.e., $\kappa\geq 0$ is arbitrary.

\subsection{Existence and uniqueness of mild solutions}\label{ssec:mild-existence}

The first question is whether mild solutions exist. Building on recent results from~\cite{bksw}, we state an improved version of the existence and uniqueness theorem for mild solutions of~\cite[Theorem~2]{bianchi2022generalized}.
We note that this holds for a general nonlinearity on
\(\ell^1(X,\mu)\) and for the linear Laplacian on general
\(\ell^p(X,\mu)\).

\begin{theorem}\label{thm:existence-mild}
Let $G$ be a graph. Fix $p=1$ if $\Phi\neq \operatorname{id}$ and fix
$p\in [1,\infty)$ if $\Phi = \operatorname{id}$. Assume that
\begin{enumerate}
    \item[\textup{(i)}] $u_0 \in \ell^p(X,\mu)$;
    \item[\textup{(ii)}] $f \in L^1_{\rm{loc}}\left([0,T);\ell^p(X,\mu)\right)$.
\end{enumerate}
Then there exists a unique mild solution $u$ of the \eqref{eq:C-D} with
data \((u_0,f)\).
Furthermore, $u(t) \in \ell^p(X,\mu)$ for all $t\in[0,T]$ and there exists a continuous function
$\delta\colon [0,\infty)\to[0,\infty)$ with $\delta(0)=0$, depending only on $u_0$, $f$ and $T$, such that, for every $\eps>0$ and every
$\eps$-approximate solution $u_\eps$ of the \eqref{eq:C-D},
\begin{equation*}
    \norm{u(t)-u_\eps(t)}{p} \leq \delta(\eps)
    \qquad \mbox{for every } t\in[0,T-\eps].
\end{equation*}

Moreover, for any two pairs $(u_0,f)$ and $(\hat u_0,\hat f)$ satisfying
(i) and (ii), the corresponding mild solutions $u,\hat u$ satisfy
\begin{equation*}
    \norm{u(t_2)-\hat u(t_2)}{p}
    \leq
    \norm{u(t_1)-\hat u(t_1)}{p}
    +
    \int_{t_1}^{t_2}\norm{f(s)-\hat f(s)}{p}\dd s
    \qquad \text{for all } 0\leq t_1<t_2\leq T.
\end{equation*}
Finally, if $u_0\geq 0$ and $f(t)\geq 0$ (or $u_0\leq 0$ and $f(t)\leq 0$)
for every $t\in[0,T]$, then $u(t)\geq 0$ (or $u(t)\leq 0$) for every
$t\in[0,T]$.
\end{theorem}
\begin{proof}
By \cite[Theorem~3.11 and its proof, and Theorem~4.6]{bksw}, there exists an
m-accretive operator $\mathcal{A}$ on $\ell^p(X,\mu)$ that is a restriction of $\Deltaphip$ and whose
domain is dense in $\ell^p(X,\mu)$, that is, $\mathcal{A}\subseteq\Deltaphip$ and $\overline{\dom(\mathcal{A})}=\ell^p(X,\mu)$.
Since $\mathcal{A}$ is m-accretive, $u_0\in\overline{\dom(\mathcal{A})}$ and
$f\in L^1_{\rm loc}([0,T];\ell^p(X,\mu))$, the abstract theory of evolution equations governed
by accretive operators applies, see \cite[Proposition~3.7]{benilan1988evolution} or
\cite[Corollary~4.1]{barbu2010nonlinear}. It yields a unique mild solution
$u\in C([0,T];\ell^p(X,\mu))$ of $(\partial_t+\mathcal{A})u = f$ with $u(0)=u_0$, realized as the uniform limit of
$\eps$-approximate solutions. Since $\mathcal{A}\subseteq\Deltaphip$, $u$ is a mild solution
of the \eqref{eq:C-D} in the sense of Definition~\ref{def:weak_solution}. The same
theory provides the convergence estimate with a continuous $\delta$ satisfying
$\delta(0)=0$ and the comparison estimate for any two data pairs. Moreover,
$u(t)\in\overline{\dom(\mathcal{A})}=\ell^p(X,\mu)$ for every $t$, so $u(t)\in\ell^p(X,\mu)$ for all
$t\in[0,T]$.

It remains to prove positivity. The resolvents $(\id+\lambda \mathcal{A})^{-1}$, $\lambda>0$,
are order-preserving (see \cite[Theorem~3.11 and its proof, and Theorem~4.6]{bksw}). Hence,
if $u_0\ge0$ and $f(t)\ge0$ for every $t$, then
each $\eps$-approximate solution is obtained by applying these resolvents to
nonnegative data and is therefore nonnegative. Passing to the limit gives
$u\ge0$.
The case $u_0\le0$ and $f(t)\le0$ is analogous.
\end{proof}

\begin{remark}\label{rem:minimal-solutions}
In both Theorem~\ref{thm:extremal-bounded-pointwise-solutions} and Theorem~\ref{thm:existence-mild}, we impose no assumptions on the graph $G$ beyond the standing ones in Definition~\ref{def:graph}. In particular, we assume neither local finiteness, nor uniform positivity of the measure, nor boundedness of the degree. The key point is that both proofs are constructive, based on a finite exhaustion strategy over Dirichlet subgraphs, see, in particular, Theorem~3.11 and its proof and Theorem~4.6 in~\cite{bksw}. This is what allows us to keep the assumptions on $G$ to a minimum.
\end{remark}

For the homogeneous problem in \(\ellone\), let
\(\mathcal A\subseteq\Deltaphi\) be the densely defined m-accretive realization
selected in the proof of Theorem~\ref{thm:existence-mild}, and let
\((S(t))_{t\ge0}\) be the nonlinear contraction semigroup generated by
\(-\mathcal A\). By the homogeneous semigroup representation
\cite[Corollary~4.2]{barbu2010nonlinear}, the unique mild solution with initial
datum \(u_0\in\ellone\) is \(u(t)=S(t)u_0\).

The next lemma states the resulting time-translation property of the
homogeneous evolution: mild solutions can be restarted at any positive time.
This is standard, and we include a short proof for the convenience of the
reader.

\begin{lemma}\label{lem:translation}
Let $u_0\in\ellone$ and let $u$ be the corresponding mild solution of
\eqref{eq:C-D} with $f=0$. Then, for every $s\ge0$, the time translate
$t\mapsto u(s+t)$ is the mild solution of the homogeneous problem with initial
datum $u(s)$.
\end{lemma}
\begin{proof}
By the representation above, $u(t)=S(t)u_0$. Hence, by the semigroup property
\cite[Proposition~4.2]{barbu2010nonlinear},
\[
u(s+t)=S(s+t)u_0=S(t)S(s)u_0=S(t)u(s).
\]
The last term is, by uniqueness, the homogeneous mild solution with initial
datum $u(s)$.
\end{proof}

\subsection{Energy estimates for the homogeneous problem}\label{ssec:energy}

For the remainder of this section we work under the following  assumptions:
\begin{itemize}
    \item we fix $m>0$ and take
    \[
        \phi=\phi_m,
        \qquad
        \phi_m(s)=s^m:=s\abs{s}^{m-1},
    \]
    with canonical extension $\Phi$;
    \item  $f\equiv0$ in \eqref{eq:C-D};
    \item given $u_0\in\ellone$, we let $u$ be the mild solution of \eqref{eq:C-D} with data \((u_0,0)\) provided by Theorem~\ref{thm:existence-mild} with $p=1$. Since $T\in(0,\infty)$ is arbitrary and, by uniqueness, the solutions on nested intervals $[0,T]$ agree, $u$ is defined for all $t\ge0$.
\end{itemize}

Two consequences of Theorem~\ref{thm:existence-mild} will be used repeatedly. For any initial data $u_0,\hat u_0\in\ellone$, the corresponding mild solutions $u$ and $\hat u$ satisfy the $\ell^1$-contraction estimate
\begin{equation}\label{eq:l1-contraction}
    \norm{u(t)-\hat u(t)}{1}\le\norm{u_0-\hat u_0}{1}
    \qquad\mbox{for every } t\ge0.
\end{equation}
By uniqueness, the zero initial datum yields the zero solution: $u_0=0$ implies $u\equiv0$. Taking $\hat u_0=0$ in \eqref{eq:l1-contraction}, so that $\hat u\equiv0$, we obtain the decay estimate
\begin{equation}\label{eq:l1-decay}
    \norm{u(t)}{1}\le\norm{u_0}{1}
    \qquad\mbox{for every } t\ge0.
\end{equation}
Since $u\equiv0$ whenever $u_0=0$, assuming $\norm{u_0}{1}>0$ in the results below entails no loss of generality.

We also need the resolvents of the m-accretive operator $\mathcal A$ fixed
above:
\[
    J_\lambda:=(\id+\lambda\mathcal A)^{-1}\colon\ellone\to\dom(\mathcal A),
    \qquad
    \lambda>0.
\]
The finite-domain approximation facts needed below are established in
\cite[Lemma~3.8 and Theorem~3.11 and its proof]{bksw}.
\begin{lemma}
\label{lem:finite-dirichlet-resolvent}
For every \(\lambda>0\), the following assertions hold.
\begin{enumerate}[(a)]
    \item for every finite $\Om\subset X$ and every $g\in\ellone$ there is
    a unique $u_\Om\in C(\Om)$ solving
    \[
        u_\Om+\lambda\Deltadir[\Omega] u_\Om^m=\projom g
        \qquad\mbox{on }\Om,
    \]
    and it satisfies $\norm{u_\Om}{1}\le\norm{\projom g}{1}$;
    \item along the fixed finite connected exhaustion $(X_n)$ used to
    select \(\mathcal A\), the zero extensions $\embn u_{X_n}$ of the
    solutions in \textup{(a)} with $\Omega=X_n$ converge to $J_\lambda g$
    in $\ellone$;
    \item $J_\lambda$ is an $\ell^1$-contraction fixing the origin:
    \[
        J_\lambda0=0
        \qquad\text{and}\qquad
        \norm{J_\lambda g-J_\lambda h}{1}\le\norm{g-h}{1}
        \quad\text{for all }g,h\in\ellone.
    \]
\end{enumerate}
\end{lemma}

These resolvents generate the Euler approximations of $u$. Let $0=t_0<t_1<\cdots<t_N=T$ be a partition of $[0,T]$ with $\lambda_k:=t_k-t_{k-1}\le\eps$ for every $k$. Since $\mathcal A\subseteq\Deltaphi$, the resolvent chain
\begin{equation}\label{eq:resolvent-chain}
    v_0:=u_0,
    \qquad
    v_k:=J_{\lambda_k}v_{k-1},
    \qquad k=1,\ldots,N,
\end{equation}
together with the choice $f_k:=0$, solves the implicit Euler system \eqref{implicit_Euler}, so the associated piecewise-constant function \eqref{epsilon_approximation} is an $\eps$-approximate solution of the homogeneous problem. The error estimate of Theorem~\ref{thm:existence-mild} then yields
\begin{equation}\label{eq:chain-convergence}
    \norm{u(t)-u_\eps(t)}{1}\le\delta(\eps)
    \qquad\mbox{for every } t\in[0,T-\eps],
\end{equation}
with $\delta(\eps)\to0$ as $\eps\downarrow0$.

We shall work assuming the Sobolev inequality \SP[\nu] for some $\nu>2$, with associated constant $C_\nu$. Our estimates are governed by the following exponents: for $m>0$ and $q>1$, we set
\begin{equation}\label{eq:exponents}
    \alpha:=\frac{\nu(1-m)}{\nu-2},
    \qquad
    \delta_q:=1+\frac{\nu m-2}{\nu(q-1)}
    \qquad \text{and} \qquad
    \gamma_q:=\frac{(\nu-2)(q-\alpha)}{\nu(q-1)}.
\end{equation}
We also set the constants
\begin{equation}\label{eq:constants}
    c_{m,q}:=\frac{4m(q-1)}{(m+q-1)^2}\in(0,1]
    \qquad \text{and} \qquad
    K_{m,q,\nu}:=\frac{q\,c_{m,q}}{C_\nu}.
\end{equation}
Let us collect the properties of these exponents which should be kept in mind:
\begin{itemize}
    \item $\gamma_q\ge 0$ if and only if $q\ge\alpha$ and, in this case, $\delta_q>0$;
    \item $\delta_q<1$ if and only if $m<2/\nu$ and $\delta_q>1$ if and only if $m>2/\nu$.
\end{itemize}
As we will see, $m<2/\nu$ corresponds to the finite-extinction regime, $m>2/\nu$ to the smoothing regime, and $\alpha$ plays the role of a critical integrability exponent for the initial datum. The borderline case $m=2/\nu$, for which $\delta_q=1$, is not treated here. For \(\nu=2N/(N-2)\), the threshold \(2/\nu\)
becomes
\(m_c=(N-2)/N\), while \(\alpha=N(1-m)/2\), the Euclidean
critical-line exponent; see \cite[Chapter~5]{vazquez2006smoothing}. 

For $q>1$, we further set for $v \in C(X)$
\[
    Y_q(v):=\norm{v}{q}^q\in[0,\infty],
    \qquad
    r_q:=\frac{m+q-1}{2}
    \qquad \text{and} \qquad
    s_q:=\nu r_q,
\]
and, for a finite subset $\Om\subset X$ and $v\in C(\Om)$,
\[
    Y_{q,\Om}(v):=\sum_{x\in\Om}\abs{v(x)}^q\,\mu(x).
\]
A direct computation from \eqref{eq:exponents} shows that
\begin{equation}\label{eq:sq-identities}
    s_q\ge q
    \iff
    q\ge\alpha,
    \qquad
    \delta_q=\frac{2(s_q-1)}{\nu(q-1)}
    \qquad \text{and} \qquad
    \gamma_q=\frac{2(s_q-q)}{\nu(q-1)}.
\end{equation}
In particular, if $q\ge\alpha$, then $1<q\le s_q$ and $\delta_q>0$.

The core of the analysis is the following energy inequality, whose proof is based on a Sobolev-interpolation argument at the level of the implicit Euler scheme.

\begin{lemma}\phantomsection\label{lem:energy-mild}
Assume \SP[\nu] for some $\nu>2$. Let $q>1$ with $q\ge\alpha$ and let
$u_0\in\ellp{1}\cap\ellp{q}$ with $M:=\norm{u_0}{1}>0$.
Then, the mild solution $u$
of the \eqref{eq:C-D} is in $\ellp{q}$ and satisfies
\begin{equation}\label{eq:energy-integral}
    \norm{u(t)}{q}^q
    +
    K_{m,q,\nu}M^{-\gamma_q}
    \int_s^t \norm{u(\tau)}{q}^{q\delta_q}\,d\tau
    \le
    \norm{u(s)}{q}^q
    \qquad\mbox{for all } 0\le s\le t.
\end{equation}
\end{lemma}
\begin{proof}
Recall that $q>1$, $q\ge\alpha$ and $M=\norm{u_0}{1}>0$; by \eqref{eq:sq-identities}, $1<q\le s_q$ and $\delta_q>0$, and $\gamma_q\ge0$. We split the proof into five steps.

\medskip
\textbf{Step 1: The finite-graph resolvent estimate.} We claim that, for every $\lambda>0$, every finite $\Om\subset X$ and every $g\in\ellp1\cap\ellp q$ with $\norm{g}{1}\le M$, the solution $u_\Om$ furnished by
Lemma~\ref{lem:finite-dirichlet-resolvent}\textup{(a)}
satisfies
\begin{equation}\label{eq:finite-resolvent-energy}
    Y_{q,\Om}(u_\Om)
    +
    \lambda K_{m,q,\nu}M^{-\gamma_q}\,Y_{q,\Om}(u_\Om)^{\delta_q}
    \le
    Y_{q,\Om}(\projom g).
\end{equation}

By the convexity of $a\mapsto \abs{a}^q$, whose derivative at $a$ is $q\,a^{q-1}$ in the signed-power notation, we have, for every $a,b\in\R$,
\[
    \abs{b}^q-\abs{a}^q
    \ge
    q\, a^{q-1}(b-a).
\]
Applying this with $a=u_\Om(x)$ and $b=\projom g(x)$, multiplying by $\mu(x)$ and summing over $\Om$
gives
\[
    Y_{q,\Om}(\projom g)-Y_{q,\Om}(u_\Om)
    \ge
    q\,\ip{u_\Om^{q-1}}{\projom g-u_\Om}_\Om
    =
    q\lambda\,
    \ip{u_\Om^{q-1}}{\Deltadir[\Omega] u_\Om^m}_\Om
\]
where in the last equality we used the equation $\projom g-u_\Om=\lambda\Deltadir[\Omega] u_\Om^m$. Since $\Om$ is finite, the localized Green formula of Proposition~\ref{prop:greenf} gives
\[
    \ip{u_\Om^{q-1}}{\Deltadir[\Omega] u_\Om^m}_\Om
    =
    \Qdir[\Om]\!\left(u_\Om^m,u_\Om^{q-1}\right).
\]
We now apply Lemma~\ref{lem:signed-power-ineq} with $\sigma=m$ and $\tau=q-1$, so that $c_{\sigma,\tau}=c_{m,q}$ and $(\sigma+\tau)/2=r_q$, to each edge term of $\Qdir[\Om]$; for the boundary and killing terms, which are of the form $w(x,y)\,u_\Om^m(x)u_\Om^{q-1}(x)$ and $\kappa(x)\,u_\Om^m(x)u_\Om^{q-1}(x)$, respectively, we use
\[
    u_\Om^m\,u_\Om^{q-1}
    =
    \abs{u_\Om}^{m+q-1}
    =
    \abs{u_\Om^{r_q}}^2
    \ge
    c_{m,q}\abs{u_\Om^{r_q}}^2
\]
which holds since $c_{m,q}\le 1$. Altogether,
\[
    \Qdir[\Om]\!\left(u_\Om^m,u_\Om^{q-1}\right)
    \ge
    c_{m,q}\,\Qdir[\Om]\!\left(u_\Om^{r_q}\right).
\]
Since $\Qdir[\Om](u_\Om^{r_q})=\Q(\embom u_\Om^{r_q})$ and
$\embom u_\Om^{r_q}\in C_0(\Omega)\subseteq C_c(X)$, the Sobolev
inequality \SP[\nu] gives
\[
    \Qdir[\Om]\!\left(u_\Om^{r_q}\right)
    \ge
    \frac{1}{C_\nu}\,
    \norm{\embom u_\Om^{r_q}}{\nu}^2
    =
    \frac{1}{C_\nu}
    \left(\sum_{x\in\Om}\abs{u_\Om(x)}^{s_q}\,\mu(x)\right)^{2/\nu}
    =
    \frac{1}{C_\nu}\,
    \norm{u_\Om}{s_q}^{2r_q}
\]
where we used $s_q=\nu r_q$. Collecting the estimates so far, we get
\begin{equation}\label{eq:pre-interpolation}
    Y_{q,\Om}(\projom g)-Y_{q,\Om}(u_\Om)
    \ge
    \lambda K_{m,q,\nu}\norm{u_\Om}{s_q}^{2r_q}.
\end{equation}

If $u_\Om=0$, then \eqref{eq:finite-resolvent-energy} is trivial, so assume $u_\Om\neq0$. Since $1\le q\le s_q$, the standard interpolation inequality gives
\[
    \norm{u_\Om}{q}
    \le
    \norm{u_\Om}{1}^a\,
    \norm{u_\Om}{s_q}^{1-a}
    \qquad \text{and} \qquad
    a=\frac{s_q-q}{q(s_q-1)}\in[0,1).
\]
Raising to the power $2r_q/(1-a)$ and using the identities
\[
    1-a=\frac{s_q(q-1)}{q(s_q-1)},
    \qquad
    \frac{2r_q}{1-a}
    =\frac{2r_q}{s_q}\cdot\frac{q(s_q-1)}{q-1}
    =q\,\delta_q
    \qquad \text{and} \qquad
    \frac{2r_q\,a}{1-a}
    =\frac{2r_q}{s_q}\cdot\frac{s_q-q}{q-1}
    =\gamma_q
\]
which follow from $2r_q/s_q=2/\nu$ and \eqref{eq:sq-identities}, we obtain
\[
    \norm{u_\Om}{s_q}^{2r_q}
    \ge
    \norm{u_\Om}{1}^{-\gamma_q}\,
    Y_{q,\Om}(u_\Om)^{\delta_q}.
\]
By Lemma~\ref{lem:finite-dirichlet-resolvent}\textup{(a)},
$\norm{u_\Om}{1}\le\norm{\projom g}{1}\le\norm{g}{1}\le M$ and, since $\gamma_q\ge 0$,
\[
    \norm{u_\Om}{s_q}^{2r_q}
    \ge
    M^{-\gamma_q}\,
    Y_{q,\Om}(u_\Om)^{\delta_q}
\]
which, combined with \eqref{eq:pre-interpolation}, proves \eqref{eq:finite-resolvent-energy}.

\medskip
\textbf{Step 2: The resolvent estimate on $X$.} We claim that, for every $\lambda>0$ and every $g\in\ellp1\cap\ellp q$ with $\norm{g}{1}\le M$, we have $J_\lambda g\in\ellp q$ and
\begin{equation}\label{eq:resolvent-energy}
    Y_q(J_\lambda g)
    +
    \lambda K_{m,q,\nu}M^{-\gamma_q}\,Y_q(J_\lambda g)^{\delta_q}
    \le
    Y_q(g).
\end{equation}
Using the fixed exhaustion \((X_n)\) selecting \(\mathcal A\), let
$u_n\in C(X_n)$ be the finite Dirichlet solutions furnished by
Lemma~\ref{lem:finite-dirichlet-resolvent}\textup{(a)} with
\(\Om=X_n\), and let $v_n:=\embn u_n$ be their zero extensions. By
Lemma~\ref{lem:finite-dirichlet-resolvent}\textup{(b)},
$v_n\to J_\lambda g$ in $\ellone$ and hence pointwise on $X$.
Setting $g_n:=g\,\mathds{1}_{X_n}$, we have
\[
    Y_q(v_n)=Y_{q,X_n}(u_n)
    \qquad \text{and} \qquad
    Y_{q,X_n}(\projn g)=Y_q(g_n)\le Y_q(g),
\]
so \eqref{eq:finite-resolvent-energy} reads
\[
    F_\lambda\left(Y_q(v_n)\right)\le Y_q(g_n),
    \qquad\mbox{where}\quad
    F_\lambda(z):=z+\lambda K_{m,q,\nu}M^{-\gamma_q}z^{\delta_q}.
\]
By Fatou's lemma, $Y_q(J_\lambda g)\le\liminf_n Y_q(v_n)\le\liminf_nY_q(g_n)\le Y_q(g)<\infty$, in particular, $J_\lambda g\in\ellp q$. Since $\delta_q>0$, the function $F_\lambda$ is continuous and increasing on $[0,\infty)$ and $Y_q(g_n)\to Y_q(g)$ by dominated convergence, therefore,
\[
    F_\lambda\left(Y_q(J_\lambda g)\right)
    \le
    \liminf_{n\to\infty}F_\lambda\left(Y_q(v_n)\right)
    \le
    \lim_{n\to\infty}Y_q(g_n)
    =
    Y_q(g),
\]
which is \eqref{eq:resolvent-energy}.

\medskip
\textbf{Step 3: The discrete energy inequality.} Fix $T>0$, $\eps>0$, a partition $0=t_0<\cdots<t_N=T$ of $[0,T]$ with mesh $\lambda_{k}=t_k-t_{k-1}$ at most $\eps$, and let $u_\eps$ be the $\eps$-approximate solution given by the resolvent chain \eqref{eq:resolvent-chain}, that is, $u_{\eps,k}=J_{\lambda_k}u_{\eps,k-1}$ with $u_{\eps,0}=u_0$. By
Lemma~\ref{lem:finite-dirichlet-resolvent}\textup{(c)},
\[
    \norm{u_{\eps,k}}{1}\le\norm{u_0}{1}=M
    \qquad\mbox{for every } k.
\]
Hence, by Step~2 and induction on $k$, we get $u_{\eps,k}\in\ellp1\cap\ellp q$ and
\begin{equation}\label{eq:discrete-energy}
    Y_{\eps,k}
    +
    K_{m,q,\nu}M^{-\gamma_q}\lambda_k\,Y_{\eps,k}^{\delta_q}
    \le
    Y_{\eps,k-1}
    \qquad \text{where }  
    Y_{\eps,k}:=Y_q(u_{\eps,k}).
\end{equation}
In particular, $Y_{\eps,k}\le Y_{\eps,k-1}\le\cdots\le Y_q(u_0)$ for every $k$.

\medskip
\textbf{Step 4: Passage to the limit.} Write $Y_\eps(t):=Y_q(u_\eps(t))$, so that $Y_\eps(0)=Y_q(u_0)$ and $Y_\eps(t)=Y_{\eps,l}$ for $t\in(t_{l-1},t_l]$. Given $t\in(0,T]$, let $l$ be such that $t\in(t_{l-1},t_l]$. Since $Y_\eps$ is piecewise constant and $t-t_{l-1}\le\lambda_l$, we get
\[
    \int_0^t Y_\eps(\tau)^{\delta_q}\,d\tau
    =
    \sum_{k=1}^{l-1}\lambda_k Y_{\eps,k}^{\delta_q}
    +(t-t_{l-1})Y_{\eps,l}^{\delta_q}
    \le
    \sum_{k=1}^{l}\lambda_k Y_{\eps,k}^{\delta_q}
\]
and summing \eqref{eq:discrete-energy} over $k=1,\ldots,l$ yields
\begin{equation}\label{eq:eps-energy}
    Y_\eps(t)
    +
    K_{m,q,\nu}M^{-\gamma_q}
    \int_0^t Y_\eps(\tau)^{\delta_q}\,d\tau
    \le
    Y_q(u_0)
    \qquad\mbox{for every } t\in[0,T].
\end{equation}
We now let $\eps\downarrow0$. By \eqref{eq:chain-convergence}, for every fixed
$\tau\in[0,T)$ we have $u_\eps(\tau)\to u(\tau)$ in $\ellone$, hence
pointwise on $X$. Therefore, Fatou's lemma yields
\[
    Y_q(u(\tau))\le\liminf_{\eps\downarrow0}Y_\eps(\tau).
\]
Since $z\mapsto z^{\delta_q}$ is continuous and increasing, also $Y_q(u(\tau))^{\delta_q}\le\liminf_{\eps\downarrow0}Y_\eps(\tau)^{\delta_q}$, and Fatou's lemma in the time variable gives, for every $t\in[0,T)$,
\[
    \int_0^t Y_q(u(\tau))^{\delta_q}\,d\tau
    \le
    \liminf_{\eps\downarrow 0}
    \int_0^t Y_\eps(\tau)^{\delta_q}\,d\tau.
\]
By the superadditivity of the limit inferior and \eqref{eq:eps-energy},
\begin{equation}\label{eq:integral-zero}
    Y_q(u(t))
    +
    K_{m,q,\nu}M^{-\gamma_q}
    \int_0^t Y_q(u(\tau))^{\delta_q}\,d\tau
    \le
    Y_q(u_0).
\end{equation}
Since $T>0$ was arbitrary, \eqref{eq:integral-zero} holds for every $t\ge0$. In particular, $u(t)\in\ellp q$ and $Y_q(u(t))\le Y_q(u_0)$ for every $t\ge0$.

\medskip
\textbf{Step 5: Restarting at time $s$.} Let $0\le s\le t$. By Lemma~\ref{lem:translation}, $u(s+\cdot)$ is the mild solution with initial datum $u(s)$, and $u(s)\in\ellp1\cap\ellp q$ by Step~4. Set $M_s:=\norm{u(s)}{1}$, so that $M_s\le M$ by \eqref{eq:l1-decay}. If $M_s=0$, then $u(s)=0$, hence $u(\tau)=0$ for every $\tau\ge s$ and \eqref{eq:energy-integral} is trivial. If $M_s>0$, then Steps~1--4 applied with $u(s)$ in place of $u_0$ and $M_s$ in place of $M$, give
\[
    Y_q(u(t))
    +
    K_{m,q,\nu}M_s^{-\gamma_q}
    \int_s^t Y_q(u(\tau))^{\delta_q}\,d\tau
    \le
    Y_q(u(s)),
\]
after the change of variables $\tau\mapsto s+\tau$ in the integral. Since $\gamma_q\ge0$ and $M_s\le M$, we have
$M_s^{-\gamma_q}\ge M^{-\gamma_q}$ and \eqref{eq:energy-integral} follows.
This completes the proof.
\end{proof}
The extinction and smoothing estimates below follow by combining
\eqref{eq:energy-integral} with the scalar comparison principle in
Lemma~\ref{lem:ode-comparison} from Appendix~\ref{sec:appendix}.

\subsection{Extinction and smoothing}\label{ssec:ext-smooth}

We now derive the extinction and smoothing
consequences of the preceding energy estimate.
In the first result, we show extinction in the range
$0<m<{2}/{\nu}$.
\begin{theorem}[Extinction]\label{thm:flexible-extinction}
Assume \SP[\nu] for some $\nu>2$ and let
\[
    0<m<\frac{2}{\nu}
    \qquad\mbox{so that}\qquad
    \alpha=\frac{\nu(1-m)}{\nu-2}>1.
\]
Let $q\ge\alpha$ and $u_0\in\ellp{1}\cap\ellp{q}$ with $M:=\norm{u_0}{1}>0$. Then, $\delta_q<1$, $\gamma_q\ge0$, and the mild solution $u$ of the \eqref{eq:C-D}
satisfies
\begin{equation}\label{eq:flexible-extinction-estimate}
    \norm{u(t)}{q}^q
    \le
    \left[
        \norm{u_0}{q}^{q(1-\delta_q)}
        -
        K_{m,q,\nu}(1-\delta_q)M^{-\gamma_q}\,t
    \right]_+^{\frac{1}{1-\delta_q}}
    \qquad\mbox{for every } t\ge0.
\end{equation}
In particular, $u$ has finite extinction time:
\[
    u(t,\cdot)= 0
    \qquad\mbox{for every }\;
    t\ \ge\
    \frac{M^{\gamma_q}\,\norm{u_0}{q}^{q(1-\delta_q)}}{K_{m,q,\nu}\,(1-\delta_q)}.
\]
\end{theorem}

\begin{proof}
Since $m<2/\nu$, we have $1-m>1-2/\nu=(\nu-2)/\nu$, whence $\alpha>1$. Moreover, $\nu m-2<0$, so $\delta_q<1$ by \eqref{eq:exponents} and $q\ge\alpha$ gives $\gamma_q\ge0$.

Set $Y(t):=\norm{u(t)}{q}^q$. The function $Y$ is measurable: indeed, it is lower semicontinuous, because if $t_n\to t$, then $u(t_n)\to u(t)$ in $\ellone$ and, hence, pointwise on $X$, so that Fatou's lemma applies. By Lemma~\ref{lem:energy-mild}, $Y$ satisfies the hypothesis of Lemma~\ref{lem:ode-comparison} with
\[
    C=K_{m,q,\nu}M^{-\gamma_q},
    \qquad
    \delta=\delta_q<1,
\]
and Lemma~\ref{lem:ode-comparison}~(i) yields \eqref{eq:flexible-extinction-estimate}, together with $Y(t)=0$ for every $t\ge Y(0)^{1-\delta_q}/(C(1-\delta_q))$, which is the stated bound on the extinction time. Finally, if $Y(t)=0$, then $u(t,x)=0$ for every $x\in X$, because $\mu(x)>0$ for every vertex.
\end{proof}

\begin{remark}[The critical choice $q=\alpha$]\label{rem:critical-choice}
For $q=\alpha$, one has
\[
    \gamma_\alpha=0,
    \qquad
    \delta_\alpha=\frac{2}{\nu}
    \qquad \text{and} \qquad
    \alpha\left(1-\delta_\alpha\right)=1-m.
\]
Thus, Theorem~\ref{thm:flexible-extinction} gives the extinction estimate
\[
    \norm{u(t)}{\alpha}^{\alpha}
    \le
    \left[
        \norm{u_0}{\alpha}^{1-m}
        -
        K_{m,\alpha,\nu}
        \left(1-\frac{2}{\nu}\right)t
    \right]_+^{\frac{\nu}{\nu-2}}
\]
in which the coefficient of $t$ does not depend on $u_0$, and the extinction time is bounded by
\[
    \frac{\nu\, C_\nu}{(\nu-2)\,\alpha\, c_{m,\alpha}}\,
    \norm{u_0}{\alpha}^{1-m}.
\]
This is the graph counterpart of the Euclidean critical-line choice
\(p=p_*\); compare \cite[Theorem~5.2 and Corollaries~5.3~and~5.4]{vazquez2006smoothing}
and the Riemannian extinction-time estimates in
\cite[Proposition~4.6 and Theorem~4.8]{bianchi2018laplacian}.
\end{remark}

We next show that the critical extinction estimate of Remark~\ref{rem:critical-choice} extends to the minimal positive pointwise solution constructed in Corollary~\ref{cor:canonical-minimal-nonnegative-solution}, even for initial data in $\ellp{\alpha}\cap\ell^\infty(X)$ which need not belong to $\ellone$.
\begin{corollary}
\label{cor:minimal-pointwise-extinction}
Assume \SP[\nu] for some $\nu>2$, let
\[
    0<m<\frac{2}{\nu}
    \qquad\mbox{and}\qquad
    \alpha:=\frac{\nu(1-m)}{\nu-2}>1.
\]
Let
\[
    0\le u_0\in\ellp{\alpha}\cap\ell^\infty(X),
\]
and let $\underline u^0$ be the minimal positive global pointwise solution
of the homogeneous problem given by
Corollary~\ref{cor:canonical-minimal-nonnegative-solution}. Then
\[
    \norm{\underline u^0(t)}{\alpha}^{\alpha}
    \le
    \left[
        \norm{u_0}{\alpha}^{1-m}
        -
        K_{m,\alpha,\nu}
        \left(1-\frac{2}{\nu}\right)t
    \right]_+^{\frac{\nu}{\nu-2}}
    \qquad\mbox{for every }t\ge0.
\]
In particular,
\[
    \underline u^0(t,\cdot)=0
    \qquad\mbox{for every }\;
    t\ge
    \frac{\nu C_\nu}
         {(\nu-2)\alpha c_{m,\alpha}}\,
    \norm{u_0}{\alpha}^{1-m}.
\]
\end{corollary}

\begin{proof}
If $u_0=0$, the conclusion is immediate. Let $(\Omega_n)$ be the finite
connected exhaustion used in the construction of $\underline u^0$, and
write
\[
    u_n^0:=u_n^{0,0},
    \qquad
    U_n^0:=U_n^{0,0}=\embn u_n^0
\]
for the corresponding zero-Dirichlet solutions and their zero
extensions.

Regard $\Omega_n$ as the finite graph obtained by restricting $w$ and
$\mu$ to $\Omega_n$ and introducing the killing term
\[
    \ka_n(x):=\ka(x)+\sum_{y\in X\setminus\Omega_n}w(x,y),
    \qquad x\in\Omega_n.
\]
Its Laplacian is $\Deltadirn$, while its energy form satisfies
\[
    \Qdir[\Omega_n](v)=\Q(\embn v),
    \qquad v\in C(\Omega_n).
\]
Consequently, this finite graph satisfies \SP[\nu] with the same
constant $C_\nu$.

The finite-dimensional solution $u_n^0$ is classical and hence mild;
by the uniqueness in Theorem~\ref{thm:existence-mild}, it is the mild
solution on this finite graph with initial datum $\projn u_0$.
For every $n$ such that $\projn u_0\ne0$,
Theorem~\ref{thm:flexible-extinction} and
Remark~\ref{rem:critical-choice} therefore give
\[
    \norm{U_n^0(t)}{\alpha}^{\alpha}
    \le
    \left[
        \norm{\projn u_0}{\alpha}^{1-m}
        -
        K_{m,\alpha,\nu}
        \left(1-\frac{2}{\nu}\right)t
    \right]_+^{\frac{\nu}{\nu-2}}.
\]
If $\projn u_0=0$, the same estimate is immediate because $U_n^0=0$.

By \eqref{eq:forced-lower-exhaustion-monotonicity} and
\eqref{eq:def-forced-lower-extremal-limit},
\[
    0\le U_n^0\uparrow\underline u^0
    \qquad\mbox{pointwise on }[0,\infty)\times X.
\]
Moreover,
\[
    \norm{\projn u_0}{\alpha}
    \uparrow
    \norm{u_0}{\alpha}.
\]
Thus, monotone convergence yields the asserted estimate for
$\underline u^0$. The extinction-time bound follows because its
right-hand side vanishes at the stated time.
\end{proof}

We now turn to the complementary range \(m>2/\nu\). In this regime, mild
solutions with \(\ell^1\) initial data instantaneously belong to \(\ell^q\)
and satisfy quantitative smoothing bounds.

\begin{theorem}[$\ell^1$-$\ell^q$ smoothing]\label{thm:smoothing}
Assume \SP[\nu] for some $\nu>2$ and let
\[
    m>\frac{2}{\nu}.
\]
Let $u_0\in\ellp{1}$ with $M:=\norm{u_0}{1}$ and let $q\in(1,\infty)$. Then the mild solution $u$ satisfies $u(t)\in\ellp{q}$ for every $t>0$ and
\begin{equation}\label{eq:smoothing-norm}
    \norm{u(t)}{q}
    \le
    \mathcal C_{m,q,\nu}\,
    t^{-\theta_q}\,
    M^{\sigma_q}
    \qquad\mbox{for every } t>0,
\end{equation}
where
\[
    \theta_q:=\frac{\nu(q-1)}{q(\nu m-2)},
    \qquad
    \sigma_q:=\frac{(\nu-2)q-\nu(1-m)}{q(\nu m-2)}
    \qquad \text{and} \qquad
    \mathcal C_{m,q,\nu}
    :=
    \left[
        K_{m,q,\nu}\,
        \frac{\nu m-2}{\nu(q-1)}
    \right]^{-\theta_q}.
\]
\end{theorem}

\begin{proof}
If $M=0$, then $u=0$ and there is nothing to prove. Therefore, assume $M>0$. Since $m>2/\nu$, we have $\nu(1-m)<\nu-2$, whence $\alpha<1<q$, consequently, $\gamma_q>0$ and, by \eqref{eq:exponents}, $\delta_q>1$ with
\begin{equation}\label{eq:delta-minus-one}
    \delta_q-1=\frac{\nu m-2}{\nu(q-1)}.
\end{equation}

\textbf{Step 1: The case $u_0\in\ellp{1}\cap\ellp{q}$.} Since $q>\alpha$, Lemma~\ref{lem:energy-mild} applies and, arguing exactly as in the proof of Theorem~\ref{thm:flexible-extinction}, so does Lemma~\ref{lem:ode-comparison}, now through case~(ii). With $Y(t):=\norm{u(t)}{q}^q$, $C=K_{m,q,\nu}M^{-\gamma_q}$ and $\delta=\delta_q$, we obtain
\begin{equation}\label{eq:smoothing-Y}
    \norm{u(t)}{q}^q
    \le
    \left[
        K_{m,q,\nu}(\delta_q-1)M^{-\gamma_q}\,t
    \right]^{-\frac{1}{\delta_q-1}}
    \qquad\mbox{for every } t>0.
\end{equation}

\textbf{Step 2: The general case $u_0\in\ellp{1}$.} Let $(X_n)$ be an exhaustion of $X$ by finite sets and define
\[
    u_{0,n}:=u_0\mathds{1}_{X_n}
    \in C_c(X)\subseteq \ellp{1}\cap\ellp{q}
\]
so that $u_{0,n}\to u_0$ in $\ellp{1}$ and $M_n:=\norm{u_{0,n}}{1}\le M$. Let $u_n$ be the mild solution with initial datum $u_{0,n}$. If $M_n=0$, then $u_n=0$. Otherwise, Step~1 applies to $u_n$ and, since $\gamma_q>0$ and $M_n\le M$ imply $M_n^{-\gamma_q}\ge M^{-\gamma_q}$, in both cases
\[
    \norm{u_n(t)}{q}^q
    \le
    \left[
        K_{m,q,\nu}(\delta_q-1)M^{-\gamma_q}\,t
    \right]^{-\frac{1}{\delta_q-1}}
    \qquad\mbox{for every } t>0.
\]
By the contraction estimate \eqref{eq:l1-contraction},
\[
    \norm{u_n(t)-u(t)}{1}
    \le
    \norm{u_{0,n}-u_0}{1}
    \to 0
    \qquad\mbox{for every } t\ge 0.
\]
Hence, for every fixed $t>0$, after passing to a subsequence, $u_n(t,x)\to u(t,x)$ for every $x\in X$. 
Fatou's lemma then gives
\[
    \norm{u(t)}{q}^q
    \le
    \liminf_{n\to\infty}\norm{u_n(t)}{q}^q
    \le
    \left[
        K_{m,q,\nu}(\delta_q-1)M^{-\gamma_q}\,t
    \right]^{-\frac{1}{\delta_q-1}}.
\]
Thus, $u(t)\in\ellp{q}$ for every $t>0$ and \eqref{eq:smoothing-Y} holds for every $u_0\in\ellp{1}$.

\textbf{Step 3: Conclusion.} Taking the $q$-th root in \eqref{eq:smoothing-Y} and using \eqref{eq:delta-minus-one} together with
\[
    \frac{1}{q(\delta_q-1)}
    =
    \frac{\nu(q-1)}{q(\nu m-2)}
    =\theta_q
    \qquad \text{and} \qquad
    \frac{\gamma_q}{q(\delta_q-1)}
    =
    \frac{(\nu-2)q-\nu(1-m)}{q(\nu m-2)}
    =\sigma_q,
\]
we obtain
\[
    \norm{u(t)}{q}
    \le
    \left[
        K_{m,q,\nu}(\delta_q-1)
    \right]^{-\theta_q}
    t^{-\theta_q}
    M^{\sigma_q}
\]
which is \eqref{eq:smoothing-norm}. This completes the proof.
\end{proof}

The estimate \eqref{eq:smoothing-norm} is the graph analogue, under
\SP[\nu], of Euclidean \(L^p\)-\(L^q\) smoothing effects, see
\cite[Theorem~3.6]{vazquez2006smoothing}.

\subsection{Examples: Sobolev inequalities and the Euclidean critical exponent}\label{ssec:sobolev-examples}

A natural question is which graphs satisfy the Sobolev inequality \SP[\nu]. For recent results see~\cite{keller2026gaussian,HuaLi2021DiscreteSobolev} and references therein. In this subsection, we show that Cayley graphs with polynomial volume growth of order $N\ge3$ satisfy
\SP[\nu], with
\[
    \nu=\frac{2N}{N-2}
\]
and that, on such graphs, the extinction/smoothing threshold $2/\nu$ provided by Theorems~\ref{thm:flexible-extinction} and~\ref{thm:smoothing} coincides with the critical exponent
\[
    m_c=\frac{N-2}{N}
\]
of the fast diffusion equation on $\R^N$.
The recovery of \(m_c=(N-2)/N\) is parallel to the Riemannian recovery
of the Euclidean critical exponent under nonnegative Ricci curvature in
\cite[Remark~4.7 and Theorem~4.8]{bianchi2018laplacian}.

We first recall the definition of a Cayley graph, formulated within the framework of Definition~\ref{def:graph}. For background, see, e.g., \cite{Woe00}.

\begin{definition}[Cayley graph]\label{def:cayley}
Let $\Gamma$ be a finitely generated group with identity element $e$ and let $S\subseteq\Gamma\setminus\{e\}$ be a finite generating set which is symmetric, i.e., $S=S^{-1}$. The \emph{Cayley graph} $\operatorname{Cay}(\Gamma,S)$ is the graph $G=(\Gamma,w,\kappa,\mu)$ with vertex set $\Gamma$, unit edge weights, no killing term and counting measure, that is,
\[
    w(x,y):=
    \begin{cases}
        1 & \mbox{if } x^{-1}y\in S\\
        0 & \mbox{otherwise}
    \end{cases}
    \qquad
    \kappa= 0,
    \qquad
    \mu= 1.
\]
\end{definition}

Observe that $w$ is symmetric since $S=S^{-1}$ and that $w(x,x)=0$ since $e\notin S$. Moreover, $\operatorname{Cay}(\Gamma,S)$ is $\abs{S}$-regular, in particular, locally finite and connected, because $S$ generates $\Gamma$. Hence, $\operatorname{Cay}(\Gamma,S)$ is indeed a connected graph in the sense of Definition~\ref{def:graph}. Denoting by
\[
    \abs{g}_S:=\min\left\{n\ge 0 \mid g=s_1\cdots s_n,\ s_1,\ldots,s_n\in S\right\}
\]
the \emph{word length} of $g\in\Gamma$, the combinatorial graph distance of $\operatorname{Cay}(\Gamma,S)$ is given by $d(x,y)=\abs{x^{-1}y}_S$. We write
\[
    V_S(r):=\#\left\{g\in\Gamma \mid \abs{g}_S\le r\right\},
    \qquad r\ge 0,
\]
for the cardinality of the ball of radius $r$ centered at the identity. By the left-invariance of $d$, balls of the same radius centered at different vertices have the same cardinality.

\begin{corollary}\label{cor:cayley-extinction}
Let $G=\operatorname{Cay}(\Gamma,S)$ be the Cayley graph of a finitely generated group $\Gamma$ and assume that, for some $N\ge3$ and $c>0$,
\begin{equation}\label{eq:volume-growth}
    V_S(r)\ge c\, r^N
    \qquad\mbox{for every } r\ge 1.
\end{equation}
Then $G$ satisfies \SP[\nu] with $\nu=2N/(N-2)$, so that $2/\nu=m_c=(N-2)/N$ and for every $u_0\in\ellone$ with $u_0\neq0$ the mild solution $u$ of the homogeneous problem satisfies:
\begin{enumerate}[(i)]
    \item[\textup{(i)}] \emph{(Extinction)} If $0<m<m_c$, then
    \[
        u(t,\cdot)= 0
        \qquad\mbox{for every }\;
        t\ge
        \frac{\nu\, C_\nu}{(\nu-2)\,\alpha\, c_{m,\alpha}}\,
        \norm{u_0}{1}^{1-m}.
    \]
    \item[\textup{(ii)}] \emph{(Smoothing)} If $m>m_c$, then the estimates \eqref{eq:smoothing-norm} hold for every $q\in(1,\infty)$.
\end{enumerate}
\end{corollary}

\begin{proof}
By \cite[Inequality~(1.9)]{HuaLi2021DiscreteSobolev}, the volume growth condition \eqref{eq:volume-growth} implies the following discrete Sobolev inequalities on $G$: for every $1\le p<N$ and every $q\ge Np/(N-p)$, there exists $C_{p,q}>0$ such that
\[
    \norm{f}{q}
    \le
    C_{p,q}
    \left(
    \sum_{x\in X}\sum_{y\sim x}
    \abs{f(y)-f(x)}^p
    \right)^{\!1/p}
    \qquad\mbox{for every } f\in C_c(X).
\]
This follows from a standard argument based on the isoperimetric estimate \cite[Theorem~4.18]{Woe00}; see also \cite[Theorem~3.6]{HuaMugnolo15}, where discrete Sobolev inequalities are derived from isoperimetric inequalities on general uniformly locally finite graphs. We choose $p=2$ and
\[
    q=\nu=\frac{2N}{N-2}=\frac{Np}{N-p},
\]
which satisfies $1\le p<N$ and $q\ge Np/(N-p)$ 
since $N\ge3$. Because $\mu\equiv1$, $\kappa\equiv0$ 
and the edge weights take values in $\{0,1\}$, we have
\[
    \sum_{x\in\Gamma}\sum_{y\sim x}\abs{f(y)-f(x)}^2
    =
    2\,\Q(f)
\]
and hence
\[
    \norm{f}{\nu}^2
    \le
    2\,C_{2,\nu}^2\,\Q(f)
    \qquad\mbox{for every } f\in C_c(\Gamma),
\]
that is, $G$ satisfies \SP[\nu] with $C_\nu=2C_{2,\nu}^2$.

Moreover, since $\mu=1$, we have $\norm{f}{\infty}\le\norm{f}{1}$ and, therefore, for every $q\geq 1$,
\begin{equation}\label{eq:l1-lq-counting}
    \norm{f}{q}^q
    \le
    \norm{f}{\infty}^{q-1}\norm{f}{1}
    \le
    \norm{f}{1}^q,
\end{equation}
that is, $\ellp{1}\subseteq\ellp{q}$ with $\norm{f}{q}\le\norm{f}{1}$; see also \cite[Lemma~7]{HuaLi2021DiscreteSobolev}.

(i): Let $0<m<m_c=2/\nu$. By \eqref{eq:l1-lq-counting}, the assumption $u_0\in\ellp{1}\cap\ellp{\alpha}$ of Theorem~\ref{thm:flexible-extinction} with the critical choice $q=\alpha$ reduces to $u_0\in\ellp{1}$. By Remark~\ref{rem:critical-choice}, the extinction time is at most
\[
    \frac{\nu\, C_\nu}{(\nu-2)\,\alpha\, c_{m,\alpha}}\,
    \norm{u_0}{\alpha}^{1-m}
    \le
    \frac{\nu\, C_\nu}{(\nu-2)\,\alpha\, c_{m,\alpha}}\,
    \norm{u_0}{1}^{1-m}
\]
where we used $\norm{u_0}{\alpha}\le\norm{u_0}{1}$ and $1-m>0$.

(ii): This is Theorem~\ref{thm:smoothing}.
\end{proof}

\begin{remark}
We point out that, in the supercritical regime $q>Np/(N-p)$, the best constants in the above discrete Sobolev inequalities are attained: the existence of extremal functions has recently been established in \cite[Theorem~2]{HuaLi2021DiscreteSobolev} for Cayley graphs satisfying \eqref{eq:volume-growth} and in \cite[Theorem~1]{HuaLi2021DiscreteSobolev} for lattice graphs.
\end{remark}

We now give some concrete examples to illustrate the above.
\begin{example}
\begin{enumerate}[(i)]
    \item For $\Gamma=\Z^N$ with the standard generators and $N\ge3$, one has
    $V_S(r)\asymp r^N$, so the volume-growth order is $N$. Hence,
    \[
        \nu=\frac{2N}{N-2}
        \qquad \text{and} \qquad
        m_c=\frac{N-2}{N}
    \]
    exactly as for the fast diffusion equation on $\R^N$. In this case, the Sobolev inequality employed in the proof of Corollary~\ref{cor:cayley-extinction} is the classical discrete Sobolev inequality on the lattice $\Z^N$, see \cite[Theorem~3.6]{HuaMugnolo15} and \cite[Inequality~(1.3)]{HuaLi2021DiscreteSobolev}.

    \item For the discrete Heisenberg group
    \[
        \mathbb H_\Z
        =
        \left\{
        \begin{pmatrix}
        1&a&c\\
        0&1&b\\
        0&0&1
        \end{pmatrix}
        \mid a,b,c\in\Z
        \right\}
    \]
    endowed with any finite generating set $S$, one has
    \[
        V_S(r)\asymp r^4.
    \]
    Indeed, $\mathbb H_\Z$ is two-step nilpotent, with
    \[
        \mathbb H_\Z/[\mathbb H_\Z,\mathbb H_\Z]\cong\Z^2
        \qquad \text{and} \qquad
        [\mathbb H_\Z,\mathbb H_\Z]\cong\Z.
    \]
    Hence, by the Bass--Guivarc'h formula
    \cite{Bass1972,Guivarch1973}, the polynomial growth degree is
    \[
        N
        =
        1\cdot
        \operatorname{rank}\bigl(
          \mathbb H_\Z/[\mathbb H_\Z,\mathbb H_\Z]
        \bigr)
        +
        2\cdot
        \operatorname{rank}[\mathbb H_\Z,\mathbb H_\Z]
        =
        1\cdot2+2\cdot1
        =
        4.
    \]
    Consequently,
    \[
        \nu=4
        \qquad\text{and}\qquad
        m_c=\frac12.
    \]
\end{enumerate}
\end{example}

\begin{remark}[On the killing term]
We conclude by observing that the conclusions of Corollary~\ref{cor:cayley-extinction} are stable under adding a killing term $\ka\ge0$: since $\Q_\ka(f)\ge\Q_0(f)$, the Sobolev inequality \SP[\nu] remains valid, with the same constant, for the graph $(X,w,\kappa,\mu)$, and Theorems~\ref{thm:flexible-extinction} and~\ref{thm:smoothing} apply verbatim, as no assumption on $\kappa$ is used in this section.
\end{remark}

\section{Generalized mass balance and conservation of mass}\label{sec:conservation_of_mass}
In this final section, we study the global mass balance for the homogeneous
problem. For a recent survey on this topic we refer to~\cite{vazquez2025survey}. In the Euclidean setting, solutions of the Cauchy problem for the \eqref{eq:GPME} on the entire space $\R^n$ with zero forcing term preserve their total mass, see, e.g., \cite[Proposition~9.15]{vazquez2007porous}. The same happens for the homogeneous Neumann problem on smooth bounded domains, see \cite[Theorem~11.2]{vazquez2007porous}, and on compact Riemannian manifolds without boundary, see \cite[Section~11.5]{vazquez2007porous}. On the other hand, when the entire space is replaced by a proper subdomain $\Omega$ with homogeneous Dirichlet boundary conditions, mass is, in general, not conserved but dissipated, as a part of it flows out through the boundary $\partial\Omega$; see \cite[Subsection~3.3.3]{vazquez2007porous}. Finally, on noncompact Riemannian manifolds, conservation of mass for the initial value problem holds provided, for instance, that the Ricci curvature admits a (not necessarily constant) lower bound; see \cite[Proposition~4.4]{bianchi2018laplacian}.
Otherwise, mass can be lost at infinity.

On an infinite graph, there are two distinct mechanisms by which mass can be
lost: absorption by the killing term \(\kappa\), which acts as a sink at
every vertex where it does not vanish, and loss ``through infinity.''
The second
mechanism is ruled out precisely by \emph{stochastic completeness at
infinity}, introduced below, which, for \(\kappa=0\), reduces to the
usual stochastic completeness \SC{} of the graph; see
\cite[Chapter~7]{KLW21}.

Throughout this section, the killing term \(\kappa\ge0\) is arbitrary and
the problem is homogeneous, that is, \(f=0\). Set
\[
    K:=\frac{\kappa}{\mu}.
\]
For \(v\in\ell^1(X,\mu)\), its \emph{total mass} is the absolutely
convergent sum
\[
\sum_{x\in X}v(x)\,\mu(x).
\]
Indeed,
\[
\left|\sum_{x\in X}v(x)\,\mu(x)\right|\le\norm{v}{1},
\]
so the map
\[
v\longmapsto\sum_{x\in X}v(x)\,\mu(x)
\]
is a bounded linear functional on \(\ell^1(X,\mu)\).
For positive \(v\), the total mass coincides with \(\norm{v}{1}\).

We recall the definition of stochastic completeness at infinity.
\begin{definition}[Stochastic completeness at infinity \SCinf{}]
We say that \(G=(X,w,\kappa,\mu)\) satisfies \SCinf{} if, for one, and hence
every, \(\lambda>0\), the only \(h\in\ell^\infty(X)\) satisfying
\[
    (\id+\lambda\Delta) h=0
\]
is \(h=0\).
\end{definition}

For arbitrary \(\kappa\),
it expresses that all heat lost by the minimal heat semigroup is accounted
for by the killing term, with no additional loss at infinity; see
\cite[Chapter~7]{KLW21}, in particular the characterizations in
\cite[Theorem~7.16 and Corollary~7.27]{KLW21}, and \cite{KL12}.

In the presence of killing, ordinary conservation of mass is replaced by
balance laws which account for the mass removed by the killing term. We
write
\[
\ell^{1,+}(X,\mu):=\{u\in \ell^{1}(X,\mu):u\geq0\}.
\]

\begin{definition}[Mass balances and conservation of mass]
\label{def:generalized-mass-balance}
\label{def:mass}
Let \(T>0\), let \(u_0\in\ell^1(X,\mu)\), and let
\(u\colon[0,T]\to\ell^1(X,\mu)\) be a solution of the homogeneous
\eqref{eq:C-D} with initial datum \(u_0\), in any of the senses of
Subsection~\ref{ssec:solutions}.
\begin{enumerate}[(i)]
\item If \(u_0\geq0\) and \(u\geq0\), we say that \(u\) satisfies the
\emph{generalized mass balance} on \([0,T]\) if
\begin{equation}\label{eq:generalized-mass-balance}
    \norm{u(t)}{1}
    +
    \int_0^t\sum_{x\in X}
        \kappa(x)\phi(u(s,x))\dd s
    =
    \norm{u_0}{1}
    \qquad \text{for all }0\le t\le T.
\end{equation}
The terms on the left are initially understood in the extended nonnegative
sense. Thus, the identity also asserts their finiteness.
\item We say that \(u\) satisfies the \emph{signed mass balance} on
\([0,T]\) if
\begin{equation}\label{eq:signed-killing-integrability}
    \int_0^T\sum_{x\in X}
        \kappa(x)\abs{\phi(u(s,x))}\dd s<\infty
\end{equation}
and
\begin{equation}\label{eq:signed-generalized-mass}
    \sum_{x\in X}u(t,x)\,\mu(x)
    +
    \int_0^t\sum_{x\in X}
        \kappa(x)\phi(u(s,x))\dd s
    =
    \sum_{x\in X}u_0(x)\,\mu(x)
    \qquad\text{for every }t\in[0,T].
\end{equation}
\item We say that \(u\) \emph{conserves mass} on \([0,T]\) if
\[
    \sum_{x\in X}u(t,x)\,\mu(x)
    =
    \sum_{x\in X}u_0(x)\,\mu(x)
    \qquad\text{for every }t\in[0,T].
\]
\end{enumerate}
\end{definition}

For nonnegative solutions, the generalized and signed mass balances
coincide, because
\(\sum_{x\in X}u(t,x)\,\mu(x)=\norm{u(t)}{1}\) and
\(\phi(u)\geq0\). When \(\kappa=0\), the signed mass balance is exactly
conservation of mass; for nonnegative solutions it can equivalently be
written as \(\norm{u(t)}{1}=\norm{u_0}{1}\). For signed solutions, however,
conservation concerns the signed total mass and does not in general imply
conservation of the \(\ell^1\)-norm.

We establish these balance laws in two settings: in
Subsection~\ref{ssec:mild-generalized-mass}, for nonnegative 
\(\ell^1\)-mild solutions and, in
Subsection~\ref{ssec:classical-generalized-mass}, for \(\ell^1\)-classical
and bounded pointwise solutions, which are allowed to be signed.

\subsection{Generalized mass balance for mild solutions}
\label{ssec:mild-generalized-mass}

For graphs of infinite measure, we use the following one-sided growth
condition at the origin \LG{}:
\[
\limsup_{r\downarrow0}\frac{\phi(r)}{r}<\infty.
\]
Since \(\phi\) is continuous and monotone increasing, with \(\phi(0)=0\), condition
\LG{} is equivalent to
\begin{equation}\label{eq:LG-on-bounded-intervals}
    C_R:=\sup_{0<r\le R}\frac{\phi(r)}{r}<\infty
    \qquad\text{for every }R>0.
\end{equation}
Only nonnegative arguments occur below. For the signed powers
\(\phi(r)=\phi_m(r)=r|r|^{m-1}\), condition \LG{} holds exactly when \(m\ge1\).

We first record the global no-flux identity needed below. Let \(L\)
be the nonnegative self-adjoint realization of \(\Delta\) associated with
the regular Dirichlet form obtained by closing \(\Q|_{C_c(X)}\) on $\ell^2(X,\mu)$, and set
\[
    R_\beta:=(\beta+L)^{-1}
    \qquad \text{for } \beta>0.
\]
We use the same notation for the consistent sub-Markovian extensions of
\(R_\beta\) to \(\ell^1(X,\mu)\) and \(\ell^\infty(X)\). The following identity is a direct consequence of the
resolvent integrability result and the Green-formula characterization of
stochastic completeness at infinity proven in \cite{GrigoryanMasamune2013} for manifolds
and in \cite{HKLMS17} for graphs, see also
\cite[Lemma~7.22 and Corollary~7.27]{KLW21}.
We include the short argument because our hypotheses do not explicitly
assume summability with respect to the killing term.

\begin{lemma}
\label{lem:no-flux-killing}
Assume that \(G\) satisfies \SCinf. If
\[
    v,\Delta v\in\ell^1(X,\mu)\cap\ell^\infty(X),
\]
then \(Kv\in\ell^1(X,\mu)\) and
\begin{equation}\label{eq:no-flux-killing}
    \sum_{x\in X}\Delta v(x)\,\mu(x)
    =
    \sum_{x\in X}\kappa(x)v(x)
\end{equation}
where both series converge absolutely.
\end{lemma}

\begin{proof}
Fix \(\beta>0\) and set
\[
    g:=(\beta +\Delta) v.
\]
Then,
\[
    g\in\ell^1(X,\mu)\cap\ell^\infty(X)
    \subseteq\ell^2(X,\mu).
\]
Let
\[
    z:=R_\beta g.
\]
The resolvent is contractive on both \(\ell^1\) and \(\ell^\infty\), so
\(z\in\ell^1(X,\mu)\cap\ell^\infty(X)\), and
\[
    (\beta+\Delta)z=g
\]
pointwise. Consequently, \(h:=v-z\) is bounded and satisfies
\[
    (\beta+\Delta)h=0.
\]
By \SCinf, \(h=0\), and hence
\[
    v=R_\beta g.
\]

Since \(g\in\ell^1(X,\mu)\), \cite[Lemma~7.22]{KLW21} gives
\[
    \sum_{x\in X}\kappa(x)|v(x)|<\infty.
\]
Moreover, since \(g\in\ell^2(X,\mu)\), the identity \(v=R_\beta g\)
implies that \(v\) belongs to the operator domain of the Dirichlet
Laplacian and, in particular, to the finite-energy class occurring in
\cite[Corollary~7.27]{KLW21}. That corollary therefore yields
\[
    \sum_{x\in X}\Delta v(x)\,\mu(x)
    =
    \sum_{x\in X}\kappa(x)v(x)
\]
as we assume \SCinf{}.
Both series converge absolutely by the assumptions and the preceding
killing-integrability conclusion.
\end{proof}

\begin{remark}[The m-accretivity restriction]
\label{rem:zero-Dirichlet-realization}
No additional hypothesis on the graph is needed to construct the
realization used below. Indeed, \cite[Theorem~3.11]{bksw}
provides a nested finite exhaustion \((Y_n)\) and a dense set
\(\Omega\subseteq\dom(\Deltaphi)\) such that
\[
    \mathcal A:=\Deltaphi|_{\Omega}
\]
is m-accretive on \(\ell^1(X,\mu)\). Its resolvents
\[
    J_\lambda
    :=
    (\id+\lambda\mathcal A)^{-1}
    \qquad \text{for }\lambda>0,
\]
are order-preserving \(\ell^1\)-contractions, satisfy \(J_\lambda0=0\),
and are the \(\ell^1\)-limits, along \((Y_n)\), of the corresponding
finite resolvent solutions.

The operator \(\mathcal A\) is an exhaustion-selected restriction of the
maximal \(\ell^1\) realization \(\Deltaphi\). Without additional hypotheses,
it need not coincide with \(\Deltaphi\). For nonnegative data, the
resolvent limit is independent of the exhaustion, see
Step~1 in the proof of \cite[Theorem~3.11]{bksw}.

For \(p=1\), we use this choice of the operator in
Theorem~\ref{thm:existence-mild}. Let \((S(t))_{t\ge0}\) be the
order-preserving contraction semigroup generated by \(-\mathcal A\).
By the homogeneous semigroup representation
\cite[Corollary~4.2]{barbu2010nonlinear}, the corresponding
\(\ell^1\)-mild solution with initial datum \(u_0\) is
\[
    u(t)=S(t)u_0.
\]
See also Lemma~\ref{lem:translation} for the resulting time-translation
property.
\end{remark}

The following identity~\eqref{eq:mild-generalized-mass} is the graph counterpart of the classical mass-evolution law for nonlinear diffusion equations with absorption and, in the linear case, of the generalized conservativeness of Schrödinger heat semigroups. To the best of our knowledge, it is new in this full graph-theoretic generality.

\begin{theorem}
\label{thm:mild-generalized-mass}
Assume that \(G\) satisfies \SCinf{} and that
\begin{equation}\label{eq:finite-or-LG}
    \mu(X)<\infty
    \qquad\text{or}\qquad
    \phi\text{ satisfies }\LG{}.
\end{equation}
Then, for every, possibly unbounded,
\(u_0\in\ell^{1,+}(X,\mu)\), the  mild solution
\[
    u(t)=S(t)u_0
\]
satisfies
\begin{equation}\label{eq:mild-generalized-mass}
    \norm{u(t)}{1}
    +
    \int_0^t\sum_{x\in X}
        \kappa(x)\phi(u(s,x))\dd s
    =
    \norm{u_0}{1} \qquad
    \text{for all } t\ge0.
\end{equation}
In particular, \(t\mapsto\norm{u(t)}{1}\) is absolutely continuous on
every compact interval and
\begin{equation}\label{eq:differential-mass-balance}
    \partial_t \norm{u(t)}{1}
    =
    -\sum_{x\in X}\kappa(x)\phi(u(t,x))
\end{equation}
for almost every \(t>0\).
\end{theorem}

\begin{proof}
We divide the proof into two steps.

\medskip
\noindent\textbf{Step 1: Bounded nonnegative initial data.}
Let 
\[
    0\le u_0\in\ell^1(X,\mu)\cap\ell^\infty(X)
    \qquad \text{and} \qquad
    R:=\norm{u_0}{\infty}.
\]
The case \(R=0\) is immediate, so assume \(R>0\). Fix \(t>0\), and choose
implicit-Euler approximations on partitions
\[
    0=t_0<t_1<\cdots<t_N=t,
    \qquad
    \lambda_k:=t_k-t_{k-1},
\]
whose meshes $\lambda_k$ tend to zero
\begin{equation}\label{eq:mild-Euler-chain}
    (\id
    +
    \lambda_k\Delta\Phi) u_{\eps,k}
    =
    u_{\eps,k-1},
    \qquad
    u_{\eps,0}=u_0.
\end{equation}
Here, \(u_{\eps,k}=J_{\lambda_k}u_{\eps,k-1}\), so
\eqref{eq:mild-Euler-chain} holds pointwise because
\(\mathcal A\subseteq\Deltaphi\).

The order preservation and \(\ell^1\)-contractivity of \(J_\lambda\),
together with its finite zero-Dirichlet approximation and the finite
maximum principle, give
\begin{equation}\label{eq:Euler-L1-Linfty-bounds}
    0\le u_{\eps,k}\le R
    \qquad \text{and} \qquad
    \norm{u_{\eps,k}}{1}\le\norm{u_0}{1}
    \qquad \text{for } k=1,\ldots,N.
\end{equation}
The \(\ell^\infty\) estimate is first obtained on every finite
zero-Dirichlet problem and then passed to its \(\ell^1\)-limit, thus, it
does not require the constant function \(R\) to belong to
\(\ell^1(X,\mu)\).

Let \(u_\eps\) be the associated piecewise constant curve. The
Crandall--Liggett approximation gives
\begin{equation}\label{eq:Euler-uniform-L1-convergence}
    \sup_{0\le s\le t}
    \norm{u_\eps(s)-u(s)}{1}
    \to 0
\end{equation}
as $\eps \to 0$.
Set
\[
    H_\eps
    :=
    \sum_{k=1}^N\lambda_k\Phi u_{\eps,k}
    =
    \int_0^t\Phi u_\eps(s)\dd s.
\]
Summing \eqref{eq:mild-Euler-chain} over \(k\) and using the linearity of
\(\Delta\) yields the pointwise identity
\begin{equation}\label{eq:Euler-integrated-equation}
    \Delta H_\eps=u_0-u_{\eps,N}.
\end{equation}
Moreover,
\[
    0\le H_\eps\le t\phi(R).
\]

If \(\mu(X)<\infty\), then
\[
    \norm{H_\eps}{1}
    \le
    t\phi(R)\mu(X).
\]
If \(\mu(X)=\infty\), condition \LG{} holds by assumption
\eqref{eq:finite-or-LG}, hence, by
\eqref{eq:LG-on-bounded-intervals} and
\eqref{eq:Euler-L1-Linfty-bounds},
\[
    \norm{H_\eps}{1}
    \le
    C_R\sum_{k=1}^N
        \lambda_k\norm{u_{\eps,k}}{1}
    \le
    C_Rt\norm{u_0}{1}.
\]
Thus, \((H_\eps)\) is uniformly bounded in \(\ell^1(X,\mu)\) in either
case.

For every \(x\in X\), \eqref{eq:Euler-uniform-L1-convergence} and
\[
    |v(x)|\le\frac{\norm{v}{1}}{\mu(x)}
\]
give uniform convergence \(u_\eps(\cdot,x)\to u(\cdot,x)\) on
\([0,t]\) as $\eps \to 0$. Since \(\phi\) is uniformly continuous on \([0,R]\),
\[
    H_\eps(x)\to
    F_t(x):=
    \int_0^t\phi(u(s,x))\dd s
\]
as $\eps \to 0$.
Fatou's lemma gives
\[
    F_t\in\ell^1(X,\mu),
    \qquad
    0\le F_t\le t\phi(R).
\]
The local summability of \(w(x,\cdot)\), the uniform bounds
\(0\le H_\eps,F_t\le t\phi(R)\), and dominated convergence imply
\[
    \Delta H_\eps(x)\to \Delta F_t(x)
    \qquad\text{for every }x\in X.
\]
Passing to the limit in \eqref{eq:Euler-integrated-equation} gives
\begin{equation}\label{eq:mild-integrated-pointwise}
    \Delta F_t=u_0-u(t).
\end{equation}
The estimates above also give
\[
    u(t)\in\ell^1(X,\mu)\cap\ell^\infty(X)
\]
and, therefore,
\[
    F_t,\Delta F_t
    \in
    \ell^1(X,\mu)\cap\ell^\infty(X).
\]
Lemma~\ref{lem:no-flux-killing}, followed by Tonelli's theorem, now yields
\begin{align*}
    \norm{u_0}{1}-\norm{u(t)}{1}
    &= \sum_{x\in X}u_0(x)\,\mu(x)
       -\sum_{x\in X}u(t,x)\,\mu(x) \\
    &= \sum_{x\in X}\Delta F_t(x)\,\mu(x)                              \\
    &=
    \sum_{x\in X}\kappa(x)F_t(x)                                      \\
    &=
    \int_0^t\sum_{x\in X}
        \kappa(x)\phi(u(s,x))\dd s.
\end{align*}
This proves \eqref{eq:mild-generalized-mass} for bounded nonnegative
initial data.

\medskip
\noindent\textbf{Step 2: Arbitrary nonnegative \(\ell^1\) data.}
Let \(u_0\in\ell^{1,+}(X,\mu)\), and set
\[
    u_{0,n}:=u_0\wedge n,
    \qquad
    u_n(t):=S(t)u_{0,n}.
\]
The order preservation and \(\ell^1\)-contractivity of \(S(t)\) give
\[
    0\le u_n(t)\le u_{n+1}(t)\le u(t)
\]
and
\[
    \norm{u_n(t)-u(t)}{1}
    \le
    \norm{u_{0,n}-u_0}{1}
    \to 0.
\]
as $n \to \infty$.
In particular,
\[
    u_n(t,x)\uparrow u(t,x)
    \qquad\text{for every }(t,x)\in[0,\infty)\times X.
\]
Step~1 gives
\[
    \norm{u_n(t)}{1}
    +
    \int_0^t\sum_{x\in X}
        \kappa(x)\phi(u_n(s,x))\dd s
    =
    \norm{u_{0,n}}{1}.
\]
Letting \(n\to\infty\) and applying monotone convergence to all three
terms proves \eqref{eq:mild-generalized-mass} in this case.

Finally, set
\[
    q(t):=
    \sum_{x\in X}\kappa(x)\phi(u(t,x)).
\]
This is measurable as the increasing limit of finite partial sums of
measurable functions. The balance on every compact time interval gives
\(q\in L^1_{\rm loc}([0,\infty))\) and
\[
    \norm{u(t)}{1}
    =
    \norm{u_0}{1}
    -
    \int_0^tq(s)\dd s.
\]
Consequently, \(t\mapsto\norm{u(t)}{1}\) is locally absolutely continuous
and \eqref{eq:differential-mass-balance} holds almost everywhere.
\end{proof}

\begin{remark}[Mass balance for mild solutions does not characterize \SC]
\label{rem:mass-balance-not-characterization}
Consider the birth--death graph
\[
    X=\mathbb N_0,\qquad
    \mu=1,\qquad
    \kappa=0,\qquad
    w(n,n+1)=(n+1)^3,
\]
and let $\phi(s)=s|s|$. Since
\[
    \sum_{n=0}^{\infty}
    \frac{\mu(\{0,\ldots,n\})}{w(n,n+1)}
    =
    \sum_{n=0}^{\infty}\frac{1}{(n+1)^2}
    <\infty,
\]
$G$ is stochastically incomplete by
\cite[Theorem~9.25]{KLW21}.

Nevertheless, every nonnegative  $\ell^1$-mild solution
preserves mass. Indeed, let 
$g\in\ell^{1,+}(X,\mu)$,
$u=J_\lambda g$, and set
\[
    v_n:=u(n)^2,\qquad
    j_n:=(n+1)^3(v_n-v_{n+1})\qquad \text{with }j_{-1}:=0.
\]
The resolvent equation yields
\[
    u(n)+\lambda(j_n-j_{n-1})=g(n)
\]
and, therefore, by a telescoping series argument,
\[
    j_n\to
    L:=\frac{\|g\|_1-\|u\|_1}{\lambda}\ge0
\]
as $n \to \infty$.

If $L>0$, then, since $u(n)\to0$,
\[
    u(n)^2
    =\sum_{k=n}^{\infty}\frac{j_k}{(k+1)^3}
    \gtrsim\frac{1}{(n+1)^2},
\]
which gives $u(n)\gtrsim(n+1)^{-1}$, contradicting $u\in\ell^1(X,\mu)$.
Hence, $L=0$ and $\|J_\lambda g\|_1=\|g\|_1$. Consequently, every
implicit-Euler approximation, and thus its Crandall--Liggett limit,
satisfies
\[
    \|S(t)u_0\|_1=\|u_0\|_1
    \qquad \text{for all }t\ge0.
\]
Moreover, $\phi$ satisfies \LG, since
$\phi(r)/r=r\to0$ as $r\searrow0$. Thus, the generalized mass balance
for, this  fixed  $\phi$, does not imply \SCinf{}.
\end{remark}

Theorem~\ref{thm:mild-generalized-mass} does not apply to the fast
diffusion range on graphs of infinite measure, since \LG{} fails when
\(0<m<1\). Nevertheless, on Cayley graphs of polynomial growth, the
energy estimate prevents any loss of mass at infinity in the
supercritical range.

\begin{corollary}
\label{cor:cayley-fde-mass}
Let \(G=\operatorname{Cay}(\Gamma,S)\) be as in
Corollary~\ref{cor:cayley-extinction} and assume, in addition, that
there exists \(C>0\) such that
\begin{equation}\label{eq:cayley-upper-volume-growth}
    V_S(r)\le C(1+r)^N
    \qquad\mbox{for every }r\ge0.
\end{equation}
Thus, \(V_S(r)\asymp r^N\). Set
\[
    m_c:=\frac{N-2}{N}=\frac{2}{\nu},
    \qquad
    \nu=\frac{2N}{N-2}.
\]
If
\[
    m_c<m<1,
\]
then, for every \(0\le u_0\in\ellone\), the mild solution
\[
    u(t)=S(t)u_0
\]
of the homogeneous problem with \(\phi=\phi_m\) conserves its mass:
\[
    \norm{u(t)}{1}=\norm{u_0}{1}
    \qquad\mbox{for every }t\ge0.
\]
\end{corollary}

\begin{proof}
Choose
\[
    m_c<\theta<m,
    \qquad
    q:=1+m-\theta>1,
    \qquad
    \beta:=\frac{m+q-1}{2}
          =m-\frac{\theta}{2}.
\]
We first prove that every resolvent \(J_\lambda\) preserves the mass of
nonnegative data.

Let \(0\le g\in\ellone\cap\ell^\infty(\Gamma)\) and set
\(v:=J_\lambda g\). Repeating the finite-resolvent argument in the proof
of Lemma~\ref{lem:energy-mild}, but stopping before applying the
Sobolev inequality, and then passing to the exhaustion limit by Fatou's
lemma, gives
\begin{equation}\label{eq:cayley-raw-resolvent-energy}
    \norm{v}{q}^q
    +
    \lambda q c_{m,q}\Q\!\left(v^\beta\right)
    \le
    \norm{g}{q}^q.
\end{equation}
In particular, \(\Q(v^\beta)<\infty\).

For \(R\ge2\), write
\[
    B_R:=\{x\in\Gamma\mid \abs{x}_S\le R\}
\]
and define
\[
    \zeta_R(x):=
    \begin{cases}
        1,
        & x\in B_R,\\[1mm]
        \displaystyle 2-\frac{\abs{x}_S}{R},
        & x\in B_{2R}\setminus B_R,\\[2mm]
        0,
        & x\notin B_{2R}.
    \end{cases}
\]
Then \(\zeta_R\in C_c(\Gamma)\), \(0\le\zeta_R\le1\), and
\[
    \abs{\zeta_R(x)-\zeta_R(y)}
    \le\frac1R
    \qquad\mbox{whenever }x\sim y.
\]
Testing the resolvent equation
\[
    v+\lambda\Delta v^m=g
\]
against \(\zeta_R\) gives
\begin{equation}\label{eq:cayley-resolvent-cutoff}
    \sum_{x\in\Gamma}(g(x)-v(x))\zeta_R(x)
    =
    \lambda\Q(v^m,\zeta_R).
\end{equation}

For \(a,b\ge0\), the choice of \(\beta\) gives the elementary estimate
\[
    \abs{a^m-b^m}
    \le
    C_{m,\theta}
    \bigl(a^{\theta/2}+b^{\theta/2}\bigr)
    \abs{a^\beta-b^\beta}.
\]
Therefore, by the Cauchy--Schwarz inequality and the fact that the
Cayley graph is \(\abs{S}\)-regular,
\[
    \abs{\Q(v^m,\zeta_R)}
    \le
    \frac{C_{m,\theta,S}}{R}
    \Q(v^\beta)^{1/2}
    \left(
        \sum_{x\in A_R}v(x)^\theta
    \right)^{1/2},
\]
where
\[
    A_R:=B_{2R+1}\setminus B_{R-1}.
\]
Since \(0<\theta<1\), Hölder's inequality, the
\(\ell^1\)-contractivity of \(J_\lambda\), and
\eqref{eq:cayley-upper-volume-growth} yield
\[
\begin{aligned}
    \sum_{x\in A_R}v(x)^\theta
    &\le
    \abs{A_R}^{1-\theta}
    \left(\sum_{x\in A_R}v(x)\right)^\theta\\
    &\le
    C(1+R)^{N(1-\theta)}
    \norm{g}{1}^{\theta}.
\end{aligned}
\]
Consequently,
\[
    \abs{\Q(v^m,\zeta_R)}
    \le
    C\Q(v^\beta)^{1/2}\norm{g}{1}^{\theta/2}
    (1+R)^{-1+\frac{N}{2}(1-\theta)}.
\]
Since
\[
    \theta>\frac{N-2}{N},
\]
the exponent on the right-hand side is negative. Hence,
\[
    \Q(v^m,\zeta_R)\longrightarrow0
    \qquad\mbox{as }R\to\infty.
\]
Letting \(R\to\infty\) in \eqref{eq:cayley-resolvent-cutoff} and using
dominated convergence gives
\[
    \norm{J_\lambda g}{1}=\norm{g}{1}.
\]

For arbitrary \(0\le g\in\ellone\), take \(g_n:=g\wedge n\).
The \(\ell^1\)-contractivity of \(J_\lambda\) gives
\[
    J_\lambda g_n\longrightarrow J_\lambda g
    \qquad\mbox{in }\ellone,
\]
and hence the preceding identity passes to the limit:
\[
    \norm{J_\lambda g}{1}=\norm{g}{1}.
\]
Thus, every step of the implicit Euler scheme preserves mass.
Finally, the Crandall--Liggett approximation converges to \(S(t)u_0\)
in \(\ellone\), and therefore
\[
    \norm{S(t)u_0}{1}=\norm{u_0}{1}
    \qquad\mbox{for every }t\ge0.
\]
This completes the proof.
\end{proof}

\subsection{Mass balances for classical and bounded pointwise solutions}
\label{ssec:classical-generalized-mass}

We now prove the generalized mass balance for nonnegative classical
solutions and the signed mass balance for possibly signed classical and
bounded pointwise solutions of the \eqref{eq:C-D}, allowing an arbitrary
killing term. The strategy is to first establish an abstract criterion,
Proposition~\ref{theorem:cons_mass}, which derives the balances directly
from the no-flux identity of Lemma~\ref{lem:no-flux-killing}, and then to
verify its hypotheses in two settings: on the one hand, for
$\ell^1(X,\mu)$-classical solutions on graphs with uniformly positive
measure and, on the other hand, for the bounded pointwise solutions
constructed in Section~\ref{sec:bounded-pointwise} on graphs with bounded
degree and finite measure, in which case stochastic completeness at infinity
is automatic. The role of the assumptions is discussed in
Remark~\ref{rem:necessity} below.

The following criterion generalizes the conservation of mass to graphs with
an arbitrary killing term. It shows that, under stochastic completeness at
infinity, the only mass lost by an $\ell^1$-classical solution is the
one removed by the killing term.

\begin{proposition}[Mass-balance criterion]\label{theorem:cons_mass}
Assume that \(G\) satisfies \SCinf{} and let \(u\) be an
\(\ell^1\)-classical solution of the \eqref{eq:C-D} with \(f=0\).
If
\[
    \Phi u(t),\,\Delta\Phi u(t)
    \in
    \ell^1(X,\mu)\cap\ell^\infty(X)
    \qquad\mbox{for every } t\in(0,T),
\]
then
\[
    \sum_{x\in X}\kappa(x)\abs{\phi(u(t,x))}<\infty
    \qquad\mbox{for every } t\in(0,T)
\]
and
\begin{equation}\label{eq:differential-generalized-mass}
    \partial_t\sum_{x\in X}u(t,x)\,\mu(x)
    =
    -\sum_{x\in X}\kappa(x)\phi(u(t,x))
    \qquad\mbox{for every } t\in(0,T).
\end{equation}
If \(u\ge0\), then
\[
    \norm{u(t)}{1}
    +
    \int_0^t\sum_{x\in X}
        \kappa(x)\phi(u(s,x))\dd s
    =
    \norm{u_0}{1}
    \qquad\text{for every }t\in[0,T].
\]
More generally, if
\[
    \int_0^T\sum_{x\in X}\kappa(x)\abs{\phi(u(s,x))}\dd s<\infty,
\]
then
\[
    \sum_{x\in X}u(t,x)\,\mu(x)
    +
    \int_0^t\sum_{x\in X}\kappa(x)\phi(u(s,x))\dd s
    =
    \sum_{x\in X}u_0(x)\,\mu(x)
    \qquad\text{for every }t\in[0,T].
\]
\end{proposition}
\begin{proof}
Fix $t\in(0,T)$ and write $v\coloneqq\Phi u(t)$. By hypothesis,
\[
v,\,\Delta v\in\ell^1(X,\mu)\cap\ell^\infty(X),
\]
so Lemma~\ref{lem:no-flux-killing} gives $Kv\in\ell^1(X,\mu)$ as we assume \SCinf{}, that is,
\[
\sum_{x\in X}\kappa(x)\abs{\phi(u(t,x))}<\infty,
\]
together with the no-flux identity
\begin{equation}\label{eq:zero-flux}
\sum_{x\in X}\Delta\Phi u(t,x)\,\mu(x)
=
\sum_{x\in X}\kappa(x)\phi(u(t,x))
\qquad\mbox{for every } t\in(0,T).
\end{equation}

Consider now the mass function $M\colon[0,T]\to\R$ given by
\[
M(t)\coloneqq\sum_{x\in X}u(t,x)\,\mu(x)
\]
which is well defined since $u(t)\in\ell^1(X,\mu)$ for every $t\in[0,T]$. Because the mass functional is a bounded linear functional on $\ell^1(X,\mu)$ and $u\in C\left([0,T];\ell^1(X,\mu)\right)$, the function $M$ is continuous on $[0,T]$. Moreover, for every $t\in(0,T)$ and every $h\neq0$ sufficiently small,
\[
\left|\frac{M(t+h)-M(t)}{h}
-\sum_{x\in X}\partial_tu(t,x)\,\mu(x)\right|
\le
\norm{\frac{u(t+h)-u(t)}{h}-\partial_tu(t)}{1} \to 0
\]
as $h \to 0$
since $u\in C^1\left((0,T);\ell^1(X,\mu)\right)$. Hence, $M$ is differentiable on $(0,T)$ and, by \eqref{eq:zero-flux} and since $u$ is a solution of the \eqref{eq:C-D} with $f=0$,
\[
M'(t)
=
\sum_{x\in X}\partial_tu(t,x)\,\mu(x)
=
-\sum_{x\in X}\Delta\Phi u(t,x)\,\mu(x)
=
-\sum_{x\in X}\kappa(x)\phi(u(t,x))
\qquad\mbox{for every } t\in(0,T).
\]
This proves \eqref{eq:differential-generalized-mass}.

Assume now that $u\ge0$. Since $\phi$ is monotone increasing and $\phi(0)=0$, the function
\[
q(t)\coloneqq\sum_{x\in X}\kappa(x)\phi(u(t,x))
\]
is nonnegative and, being equal to $-M'$, continuous on $(0,T)$. For
$0<\eps<t<T$, the fundamental theorem of calculus gives
\[
M(t)+\int_\eps^tq(s)\dd s=M(\eps).
\]
Letting $\eps\downarrow0$, using the continuity of $M$ at zero on the
right-hand side and monotone convergence on the left-hand side, yields
\[
M(t)+\int_0^tq(s)\dd s=M(0)
\]
where all the terms are finite. Since $u\ge0$, one has
$M(t)=\norm{u(t)}{1}$ and this is exactly the generalized mass balance
\eqref{eq:generalized-mass-balance}.

For signed solutions, assumption \eqref{eq:signed-killing-integrability}
makes $q$ integrable on $(0,T)$; integrating
\eqref{eq:differential-generalized-mass} over $(\eps,t]$ and letting
$\eps\downarrow0$, now by dominated convergence, gives
\eqref{eq:signed-generalized-mass}. In both cases, the identity at $t=T$
follows by letting $t\uparrow T$, using the continuity of $M$ together
with monotone, respectively dominated, convergence.
\end{proof}

In order to verify the hypotheses of
Proposition~\ref{theorem:cons_mass} for the solutions at hand, we will use
the following elementary inclusions between the $\ell^p$-spaces.

\begin{lemma}\label{l:inclusions}
Let $1\le p<q\le\infty$ with the convention $1/\infty\coloneqq0$.
\begin{enumerate}[(i)]
\item[\textup{(i)}] If $G$ satisfies \UM, then $\ell^p(X,\mu)\subseteq \ell^q(X,\mu)$ and, writing $c\coloneqq\inf_{x\in X}\mu(x)>0$,
\[
\norm{f}{q}\le c^{\frac1q-\frac1p}\norm{f}{p}
\qquad\mbox{for every } f\in\ell^p(X,\mu).
\]
\item[\textup{(ii)}] If $G$ satisfies \FM, then $\ell^q(X,\mu)\subseteq \ell^p(X,\mu)$ and
\[
\norm{f}{p}\le \mu(X)^{\frac1p-\frac1q}\norm{f}{q}
\qquad\mbox{for every } f\in\ell^q(X,\mu).
\]
\end{enumerate}
\end{lemma}

We are now ready to prove the main result of this subsection.

\begin{theorem}\label{t:conservation}
Let $G=(X,w,\kappa,\mu)$ be a graph with arbitrary killing term
$\kappa\ge0$ and consider the \eqref{eq:C-D} with $f= 0$.
\begin{enumerate}[(1)]
\item[\textup{(i)}] Suppose that $G$ satisfies \SCinf{} and \UM{} and that $\Phi$ satisfies
\C. Then, every nonnegative $\ell^1$-classical solution satisfies
the generalized mass balance
\[
\norm{u(t)}{1}
+
\int_0^t\sum_{x\in X}\kappa(x)\phi(u(s,x))\dd s
=
\norm{u_0}{1}
\qquad\mbox{for every } t\in[0,T].
\]
Every signed $\ell^1$-classical solution satisfies the differential
identity \eqref{eq:differential-generalized-mass} and, whenever
\eqref{eq:signed-killing-integrability} holds, the integrated signed
balance \eqref{eq:signed-generalized-mass}.
\item[\textup{(ii)}] Suppose that $G$ satisfies \BD{} and \FM. Then, every bounded
pointwise solution satisfies the signed
balance
\[
\sum_{x\in X}u(t,x)\,\mu(x)
+
\int_0^t\sum_{x\in X}\kappa(x)\phi(u(s,x))\dd s
=
\sum_{x\in X}u_0(x)\,\mu(x)
\qquad\mbox{for every } t\in[0,T].
\]
\end{enumerate}
In particular, if $\kappa=0$, then, in both settings, every such
solution satisfies the conservation of mass property of
Definition~\ref{def:mass}.
\end{theorem}
\begin{proof}
In both cases, we verify the hypotheses of Proposition~\ref{theorem:cons_mass}: namely, that $G$ satisfies \SCinf, that $u$ is an $\ell^1$-classical solution and that $\Phi u(t),\Delta\Phi u(t)\in\ell^1(X,\mu)\cap\ell^\infty(X)$ for every $t\in(0,T)$.

(i): Stochastic completeness at infinity holds by hypothesis. Let $u$ be an $\ell^1$-classical solution and fix $t\in(0,T)$. By Definition~\ref{def:classical_solution}, $u(t)\in\ell^1(X,\mu)$ and $\Delta\Phi u(t)\in\ell^1(X,\mu)$, while $\Phi u(t)\in\ell^1(X,\mu)$ by \C. Since $G$ satisfies \UM, Lemma~\ref{l:inclusions} gives $\ell^1(X,\mu)\subseteq\ell^\infty(X)$, and hence
\[
\Phi u(t),\,\Delta\Phi u(t)\in\ell^1(X,\mu)\cap\ell^\infty(X).
\]
All the conclusions now follow from Proposition~\ref{theorem:cons_mass}.

(ii): We first recall that \BD{} implies \SCinf, see
\cite[Corollary~27]{KL10}. Indeed, set
$D\coloneqq\sup_{x\in X}\Deg(x)$, which is finite by \BD. Under \BD, the
operator $\Delta$ is bounded on $\ell^\infty(X)$ with norm at most $2D$,
see \cite[Theorem~2.15]{KLW21}. Hence, if
$0<\lambda<(2D)^{-1}$ and $h\in\ell^\infty(X)$ satisfies
$h+\lambda\Delta h=0$, then
\[
\norm{h}{\infty}
=
\lambda\norm{\Delta h}{\infty}
\le
2\lambda D\norm{h}{\infty}
\]
which forces $h=0.$ Hence, $G$ satisfies \SCinf{}. 

Let now $u$ be a bounded pointwise solution. By \BD, \FM{} and Lemma~\ref{l:classical-solution}, $u$ is an $\ell^p$-classical solution for every $p\in[1,\infty]$ and, in particular, an $\ell^1(X,\mu)$-classical solution. Set
\[
R\coloneqq\sup_{(s,x)\in[0,T]\times X}|u(s,x)|<\infty
\qquad \text{and} \qquad
C\coloneqq\sup_{|r|\le R}|\phi(r)|<\infty.
\]
Then, for every $t\in(0,T)$,
\[
\norm{\Phi u(t)}{\infty}\le C
\qquad\mbox{and}\qquad
\norm{\Delta\Phi u(t)}{\infty}\le 2DC
\]
and, by \FM{} and Lemma~\ref{l:inclusions}, both functions also belong to
$\ell^1(X,\mu)$. Moreover, since $\kappa(x)\le\Deg(x)\mu(x)\le D\mu(x)$
for every $x\in X$,
\[
\int_0^T\sum_{x\in X}\kappa(x)\abs{\phi(u(s,x))}\dd s
\le
TC\sum_{x\in X}\kappa(x)
\le
TCD\mu(X)
<\infty.
\]
Thus, \eqref{eq:signed-killing-integrability} holds, and
Proposition~\ref{theorem:cons_mass} gives the signed mass balance.

Finally, if $\kappa=0$, all the killing integrals vanish: in
case~(ii), and in case~(i) for nonnegative solutions, the balances reduce
to
\[
\sum_{x\in X}u(t,x)\,\mu(x)
=
\sum_{x\in X}u_0(x)\,\mu(x),
\]
while, for signed solutions in case~(i),
condition \eqref{eq:signed-killing-integrability} holds trivially and
\eqref{eq:signed-generalized-mass} gives the same conclusion.
\end{proof}

For the porous medium nonlinearity, condition \C{} in Theorem~\ref{t:conservation}~(i) is automatically satisfied.

\begin{corollary}\label{cor:conservation-pme}
Let $\phi(s)=s^m$ with $m\ge1$ and let $G$ be a graph satisfying \SCinf{} and \UM. Then, every positive 
$\ell^1$-classical solution of \eqref{eq:C-D} with $f=0$ satisfies the generalized mass balance
\[
\norm{u(t)}{1}
+
\int_0^t\sum_{x\in X}\kappa(x)\,u(s,x)^m\dd s
=
\norm{u_0}{1}
\qquad\mbox{for every } t\in[0,T].
\]
In particular, if $\kappa=0$, every such solution satisfies the conservation of mass property.
\end{corollary}
\begin{proof}
By Theorem~\ref{t:conservation}~(i), it is enough to show that $\Phi$ satisfies \C, see also \cite[Remark~6]{bianchi2022generalized} or \cite[Proposition~2.4]{bksw}. Let $v\in\ell^1(X,\mu)$. By \UM{} and Lemma~\ref{l:inclusions}, $v\in\ell^\infty(X)$ and, since $m\ge1$,
\[
\sum_{x\in X}|v(x)|^m\,\mu(x)
\le
\norm{v}{\infty}^{m-1}\sum_{x\in X}|v(x)|\,\mu(x)
=
\norm{v}{\infty}^{m-1}\norm{v}{1}
<\infty,
\]
that is, $\Phi v\in\ell^1(X,\mu)$.
\end{proof}

\begin{remark}[On the assumptions]\label{rem:necessity}
\begin{enumerate}[(i)]
\item The ordinary conservation conclusion cannot be expected when
\(\kappa\not =0\), because the killing term acts as an absorption term.
This is already visible on finite graphs: if $X$ is finite and $u$ is a pointwise solution of the \eqref{eq:C-D} with $f=0$, then, summing the equation over $X$ and using the symmetry of $w$, we get
\[
\frac{d}{dt}\sum_{x\in X}u(t,x)\,\mu(x)
=
-\sum_{x\in X}\Delta\Phi u(t,x)\,\mu(x)
=
-\sum_{x\in X}\kappa(x)\phi\left(u(t,x)\right)
\qquad \text{for all } t\in(0,T)
\]
by Lemma~\ref{lem:no-flux-killing} as finite graphs are stochastically complete at
infinity.
For a nonnegative solution, the right-hand side is nonpositive, and it is
strictly negative at time \(t\) whenever
\(\kappa(x)\phi(u(t,x))>0\) for some \(x\in X\). The generalized balance
\eqref{eq:generalized-mass-balance} is the correct replacement: it adds back
exactly the mass dissipated by the killing term and, by
Theorem~\ref{t:conservation}, under stochastic completeness at infinity no
further loss occurs.
\item Stochastic completeness at infinity is the natural assumption in Theorem~\ref{t:conservation}~(i): already in the linear case $\phi=\operatorname{id}$, the validity of the generalized mass balance for all nonnegative initial data in $\ell^1(X,\mu)$ is equivalent to \SCinf; see \cite[Chapter~7]{KLW21}. For a fixed general nonlinearity $\phi$, our argument establishes the implication from \SCinf{} to the balance, the converse should not be expected without additional hypotheses.
\end{enumerate}
\end{remark}

\begin{remark}[Relation to curvature]
For graphs without killing term, stochastic completeness holds under suitable lower Ricci curvature bounds; see \cite{HL17} for Bakry-\'{E}mery curvature and \cite{MW19} for Ollivier curvature. Combined with such criteria, Theorem~\ref{t:conservation}~(i) and Corollary~\ref{cor:conservation-pme} give a counterpart on graphs to the conservation of mass results on Riemannian manifolds with a lower Ricci curvature bound found in \cite{bianchi2018laplacian}.
\end{remark}

\appendix

\section{Auxiliary results}\label{sec:appendix}

The auxiliary results below are used in
Subsections~\ref{ssec:energy} and~\ref{ssec:ext-smooth}. We retain the
signed-power convention
\[
    a^\rho:=a\abs{a}^{\rho-1},
    \qquad a\in\R,\quad \rho>0.
\]

We begin with an elementary inequality for signed powers.

\begin{lemma}\label{lem:signed-power-ineq}
For every $a,b\in\R$ and every $\sigma,\tau>0$,
\begin{equation}\label{eq:power-ineq}
    \bigl(b^\sigma-a^\sigma\bigr)
    \bigl(b^\tau-a^\tau\bigr)
    \ge
    c_{\sigma,\tau}
    \left|b^{\frac{\sigma+\tau}{2}}
          -a^{\frac{\sigma+\tau}{2}}\right|^2
    \qquad \text{where} \qquad 
    c_{\sigma,\tau}
    :=
    \frac{4\sigma\tau}{(\sigma+\tau)^2}.
\end{equation}
Note that $c_{\sigma,\tau}\le1$ by the AM--GM inequality and that, for $\sigma=m$ and $\tau=q-1$, the constant $c_{\sigma,\tau}$ coincides with $c_{m,q}$ in \eqref{eq:constants}, while $\frac{\sigma+\tau}{2}=r_q$.
\end{lemma}

\begin{proof}
Set
\[
    r:=\frac{\sigma+\tau}{2}.
\]
First assume that $ab\ge0$. By the oddness of signed powers (both sides are invariant under $(a,b)\mapsto(-a,-b)$) and the symmetry in $a,b$, it is
enough to prove the estimate for
$0\le a\le b$. In this case,
\[
    b^\sigma-a^\sigma
    =
    \sigma\int_a^b s^{\sigma-1}\,ds
    \qquad \text{and} \qquad
    b^\tau-a^\tau
    =
    \tau\int_a^b s^{\tau-1}\,ds.
\]
By the Cauchy--Schwarz inequality,
\[
\begin{aligned}
    \bigl(b^\sigma-a^\sigma\bigr)
    \bigl(b^\tau-a^\tau\bigr)
    &=
    \sigma\tau
    \left(\int_a^b s^{\sigma-1}\,ds\right)
    \left(\int_a^b s^{\tau-1}\,ds\right) \\
    &\ge
    \sigma\tau
    \left(\int_a^b s^{r-1}\,ds\right)^2
    =
    \frac{4\sigma\tau}{(\sigma+\tau)^2}
    \left(b^r-a^r\right)^2
\end{aligned}
\]
which is \eqref{eq:power-ineq} in the case $ab\ge0$.

Now assume that $ab<0$; up to replacing $(a,b)$ with $(-a,-b)$, we may suppose $a<0<b$, and we write $A:=\abs{a}$ and $B:=\abs{b}$. Then
\[
    \bigl(b^\sigma-a^\sigma\bigr)
    \bigl(b^\tau-a^\tau\bigr)
    =
    \left(A^\sigma+B^\sigma\right)
    \left(A^\tau+B^\tau\right).
\]
Another application of the Cauchy--Schwarz inequality gives
\[
    \left(A^\sigma+B^\sigma\right)
    \left(A^\tau+B^\tau\right)
    \ge
    \left(A^r+B^r\right)^2
    =
    \left|b^r-a^r\right|^2.
\]
Since $c_{\sigma,\tau}\le 1$, this implies \eqref{eq:power-ineq} in the
case $ab<0$ as well.
\end{proof}

We next provide the scalar comparison principle used to deduce the extinction
and smoothing estimates from the integral energy inequality
\eqref{eq:energy-integral}.

\begin{lemma}\label{lem:ode-comparison}
Let $C>0$ and $\delta>0$ with $\delta\neq1$, and let
$Y\colon[0,\infty)\to[0,\infty)$ be measurable and satisfy
\[
    Y(t)+C\int_s^t Y(\tau)^\delta\,\dd\tau
    \le Y(s)
    \qquad\mbox{for all }0\le s\le t.
\]
Then $Y$ is nonincreasing and the following hold:
\begin{enumerate}
    \item[\textup{(i)}] If $0<\delta<1$, then
    \[
        Y(t)
        \le
        \left[
            Y(0)^{1-\delta}
            -C(1-\delta)t
        \right]_+^{\frac{1}{1-\delta}}
        \qquad\mbox{for every }t\ge0.
    \]
    In particular,
    \[
        Y(t)=0
        \qquad\mbox{for every }\;
        t\ge
        \frac{Y(0)^{1-\delta}}{C(1-\delta)}.
    \]

    \item[\textup{(ii)}] If $\delta>1$ and $Y(0)>0$, then
    \[
        Y(t)
        \le
        \left[
            Y(0)^{1-\delta}
            +C(\delta-1)t
        \right]^{-\frac{1}{\delta-1}}
        \qquad\mbox{for every }t\ge0.
    \]
    Consequently, whether or not $Y(0)=0$,
    \[
        Y(t)
        \le
        \left[
            C(\delta-1)t
        \right]^{-\frac{1}{\delta-1}}
        \qquad\mbox{for every }t>0.
    \]
\end{enumerate}
\end{lemma}

\begin{proof}
It is a refinement of the Bihari--Langenhop integral inequality; see
\cite[Section~2.3, Theorems~2.3.1--2.3.2]{Pachpatte1998Ch2}. For $0\le s\le t$, the hypothesis gives
\[
    Y(t)
    \le
    Y(t)+C\int_s^tY(\tau)^\delta\,\dd\tau
    \le
    Y(s),
\]
so $Y$ is nonincreasing. In particular, if $Y(s)=0$ for some
$s\ge0$, then $Y(t)=0$ for every $t\ge s$.

Fix $t>0$. If $Y(t)=0$, there is nothing to prove. We may therefore
assume that $Y(t)>0$ and define
\[
    z(s)
    :=
    Y(t)+C\int_s^tY(\tau)^\delta\,\dd\tau,
    \qquad 0\le s\le t.
\]
Then $z\in AC([0,t])$,
\[
    Y(t)\le z(s)\le Y(s)
    \qquad\mbox{for every }0\le s\le t,
\]
and
\[
    z'(s)=-C\,Y(s)^\delta
    \qquad\mbox{for almost every }s\in(0,t).
\]
In particular, $z(s)>0$ on $[0,t]$.

Suppose first that $\delta>1$. Set
\[
    G(r)
    :=
    \int_{Y(t)}^r\xi^{-\delta}\,\dd\xi
    =
    \frac{r^{1-\delta}-Y(t)^{1-\delta}}{1-\delta}.
\]
For almost every $s\in(0,t)$,
\[
    \frac{\dd}{\dd s}G(z(s))
    =
    z(s)^{-\delta}z'(s)
    =
    -C\left(\frac{Y(s)}{z(s)}\right)^\delta
    \le -C,
\]
because $z(s)\le Y(s)$. Integrating over $[s,t]$ and using
$z(t)=Y(t)$ gives
\[
    G(z(s))\ge C(t-s),
\]
and hence
\[
    Y(t)^{1-\delta}
    \ge
    z(s)^{1-\delta}+C(\delta-1)(t-s).
\]
Since $z(s)\le Y(s)$ and $1-\delta<0$,
\[
    Y(t)^{1-\delta}
    \ge
    Y(s)^{1-\delta}+C(\delta-1)(t-s).
\]
Therefore,
\[
    Y(t)
    \le
    \left[
        Y(s)^{1-\delta}
        +C(\delta-1)(t-s)
    \right]^{-\frac{1}{\delta-1}}.
\]
Taking $s=0$ proves the first estimate in \textup{(ii)}. Dropping the
nonnegative term $Y(0)^{1-\delta}$ gives the second one. If $Y(0)=0$,
then $Y\equiv0$ by monotonicity, so the latter estimate remains valid.

Suppose now that $0<\delta<1$. Define
\[
    G(r)
    :=
    \int_0^r\xi^{-\delta}\,\dd\xi
    =
    \frac{r^{1-\delta}}{1-\delta}.
\]
The same computation gives
\[
    \frac{\dd}{\dd s}G(z(s))\le-C
    \qquad\mbox{for almost every }s\in(0,t).
\]
Integrating over $[s,t]$ and using $z(t)=Y(t)$ yields
\[
    G(z(s))\ge G(Y(t))+C(t-s).
\]
Since $G$ is increasing and $z(s)\le Y(s)$, we obtain
\[
    Y(t)^{1-\delta}
    \le
    Y(s)^{1-\delta}-C(1-\delta)(t-s).
\]
Taking $s=0$ gives
\[
    Y(t)
    \le
    \left[
        Y(0)^{1-\delta}-C(1-\delta)t
    \right]_+^{\frac{1}{1-\delta}}.
\]
Indeed, if the expression inside the brackets is nonpositive, the
assumption $Y(t)>0$ leads to a contradiction, and therefore $Y(t)=0$.
This also proves the asserted extinction estimate.
\end{proof}

\section*{Acknowledgments}
D.~B.~is supported by the Startup Fund of Sun Yat-sen University.
B.~H.~is supported by NSFC, no.~12371056.
A.~G.~S.~is a member of the GNAMPA-INdAM group {``Equazioni Differenziali e Sistemi Dinamici''}.
R.~K.~W.~is supported by the Simons Foundation, in the form
of a Travel Support for Mathematicians gift, and was supported by the PSC-CUNY, in the form
of a Department Chair CUNY Research Foundation (RF) Account. 
\printbibliography
\end{document}